\documentclass[11pt,reqno]{amsart}
\usepackage{amsmath,amssymb,latexsym,esint,cite,mathrsfs}
\usepackage{verbatim,wasysym}
\usepackage[left=1.8cm,right=1.8cm,top=2.4cm,bottom=2.4cm]{geometry}
\usepackage{tikz,enumitem,graphicx, subfig, microtype, color}
\usepackage{epic,eepic}

\usepackage[colorlinks=true,urlcolor=blue, citecolor=red,linkcolor=blue,
linktocpage,pdfpagelabels, bookmarksnumbered,bookmarksopen]{hyperref}
\usepackage[hyperpageref]{backref}
\usepackage[english]{babel}
\usepackage{appendix}

\numberwithin{equation}{section}

\newtheorem{thm}{Theorem}[section]
\newtheorem{lem}[thm]{Lemma}
\newtheorem{cor}[thm]{Corollary}
\newtheorem{Prop}[thm]{Proposition}

\newtheorem{Rem}[thm]{Remark}

\renewcommand{\epsilon}{\varepsilon}

\begin{document}
	\title[On an eigenvalue problem related to the HLS critical exponent]
	{Qualitative properties of single blow-up solutions for nonlinear Hartree equation with slightly subcritical exponent}
	
\author[A. Cannone]{Alessandro Cannone}
\address{\noindent Alessandro Cannone  \newline
	Dipartimento di Matematica, Universit\`{a} degli Studi di Bari Aldo Moro,\newline
	Via Orabona 4, 70125 Bari, Italy.}\email{alessandro.cannone@uniba.it}

\author[S. Cingolani]{Silvia Cingolani$^\dag$}
	\address{\noindent Silvia Cingolani  \newline
		Dipartimento di Matematica, Universit\`{a} degli Studi di Bari Aldo Moro,\newline
		Via Orabona 4, 70125 Bari, Italy.}\email{silvia.cingolani@uniba.it}

	\author[M. Yang]{Minbo Yang$^\ddag$}
	\address{\noindent Minbo Yang  \newline
		School of Mathematical Sciences, Zhejiang Normal University,\newline
		Jinhua 321004, Zhejiang, People's Republic of China.}\email{mbyang@zjnu.edu.cn}
	
	\author[S. Zhao]{Shunneng Zhao$^\S$}
	\address{\noindent Shunneng Zhao  \newline
		Dipartimento di Matematica, Universit\`{a} degli Studi di Bari Aldo Moro,\newline Via Orabona 4, 70125 Bari, Italy.
		\vspace{2mm}
		\newline
School of Mathematical Sciences, Zhejiang Normal University,\newline
		Jinhua 321004, Zhejiang, People's Republic of China.}
	\email{snzhao@zjnu.edu.cn}



\thanks{Alessandro Cannone is supported by D.M. 2023 n. 118- PNRR, "PDEs from Quantum Science", CUP: H91I23000500007, and by GNAMPA research project  "Problemi di ottimizzazione in PDEs da modelli biologici" from INDAM, CUP: E5324001950001}
	\thanks{$^\dag$Silvia Cingolani is supported by PNRR MUR project PE0000023 NQSTI - National Quantum Science and Technology Institute (CUP H93C22000670006) and partially supperted by INdAM-GNAMPA}
	
	\thanks{$^\ddag$Minbo Yang was partially supported by National Key Research and Development Program of China (No. 2022YFA1005700) and National Natural Science Foundation of China (12471114).}

	\thanks{$^\S$Shunneng Zhao was partially supported by PNRR MUR project PE0000023 NQSTI - National Quantum Science and Technology Institute (CUP H93C22000670006) and National Natural Science Foundation of China (12401146, 12261107) and Natural Science Foundation of Zhejiang Province (LMS25A010007)}
	
	\allowdisplaybreaks
	
	\begin{abstract}
		{\small
In this paper, we study the qualitative properties of single blow-up solutions to the nonlocal equations with slightly subcritical exponents
\begin{equation*}
				-\Delta u=(|x|^{-(n-2)}\ast u^{p-\epsilon})u^{p-1-\epsilon}\quad \mbox{in}~~\Omega,~~ u=0\quad \mbox{on}~~\partial\Omega,
			\end{equation*}
where $\Omega$ is a smooth bounded domain in $\mathbb{R}^n$ for $n=3,4,5$, $\ast$ denotes the standard convolution, $\epsilon>0$ is a small parameter and $p=\frac{n+2}{n-2}$ is $\mathcal{D}^{1,2}$ energy-critical exponent.
By exploiting various local Pohozaev identities and blow-up analysis,
we provide a number of estimates on the
first $(n+2)$-eigenvalues and their corresponding eigenfunctions, and examine the qualitative behavior of the eigenpairs $(\lambda_{i,\epsilon}, v_{i,\epsilon})$ to the linearied problem of the above nonlocal equations for $i=1,\cdots,n+2$. As a corollary, we derive the Morse index of a single-bubble solution in a nondegenerate setting.
		}
	\end{abstract}
	
	\vspace{3mm}
	
	\maketitle
	\section{Introduction}
	\subsection{Motivation and main results}
	In this paper, we are concerned with the following nonlocal problem
\begin{equation}\label{ele-1.1}
\left\lbrace
\begin{aligned}
    &-\Delta u=(|x|^{-{\alpha}}\ast u^{p_{\alpha}-\epsilon})u^{p_{\alpha}-1-\epsilon} \quad \mbox{in}\quad \Omega,\\
    &u>0\quad \mbox{in}\quad\hspace{1mm} \Omega,\\
    &u=0\quad \mbox{on}\hspace{2.5mm}\partial\Omega,
   \end{aligned}
\right.
\end{equation}
where $\Omega$ is a smooth bounded domain in $\mathbb{R}^n$ for $n\geq3$, $\alpha\in(0,n)$, $\epsilon>0$ is a small parameter and the exponent $p_{\alpha}=\frac{2n-\alpha}{n-2}$ is a threshold on the existence of a solution to \eqref{ele-1.1}.
If $\epsilon>0$, one
can find a solution to \eqref{ele-1.1} by applying  standard variational argument with the compact embedding $H^{1}(\Omega)\hookrightarrow L^{\frac{2n(p-\epsilon)}{2n-\alpha}}(\Omega)$. If $\epsilon=0$ and $\Omega$ is star-shaped, then an application of the nonlocal-type Pohozaev identity in \cite{GY18} gives nonexistence of a nontrivial solution for \eqref{ele-1.1}.
In the recent paper \cite{cw}, Chen and Wang
proved that if a solution $u_{\epsilon}$ of \eqref{ele-1.1} satisfies
\begin{equation*}
|\nabla u_\epsilon|^2\rightarrow C_{HLS}^{\frac{2n-\alpha}{n+2-\alpha}}\delta_{x_0}~~~\text{as}~~~\epsilon\rightarrow0
\end{equation*}
for $n\geq3$ and $\alpha\in(0, \min\{n,4\})$,
then the concentrate point $x_{0}\in \Omega$ is a critical point of the Robin function $\phi$ (see below). In addition, they also used the finite dimensional reduction method to give a converse result, i.e., for $\epsilon$ sufficiently small, \eqref{ele-1.1} has a family of solutions $u_{\epsilon}$ concentrating around $x_0$ under certain dimensional restriction. Here $\delta(x)$ denotes the Dirac measure centered at the origin, and $C_{HLS}>0$ is a constant depending on the dimension $n$ and parameter $\alpha$ (see \eqref{mini-p}),
 which can be characterized by the following minimization problem
\begin{equation}\label{mini-p}
\frac{1}{C_{HLS}}:=\inf\Big\{\|\nabla u\|_{L^2}\big|u\in \dot{H}^{1}(\mathbb{R}^n),\hspace{2mm}\Big(\int_{\mathbb{R}^n}\big(|x|^{-\alpha} \ast u^{p_{\alpha}}\big)u^{p_{\alpha}} dx\Big)^{\frac{1}{p_{\alpha}}}=1\Big\},
 \end{equation}
 where $n\geq 3$, $\alpha\in(0,n)$ and $\dot{H}^{1}(\mathbb{R}^n):=\mathcal{D}^{1,2}(\mathbb{R}^n)$, the completion of $C_{0}^{\infty}(\mathbb{R}^n)$ with respect to the norm $\|\nabla u\|_{L^2(\mathbb{R}^n)}$.
It is well-known that the critical nonlocal Hartree equation, up to suitable scaling,
\begin{equation}\label{ele-1.1-1}
    \Delta u+(|x|^{-\alpha}\ast u^{p_{\alpha}})u^{p_{\alpha}-1}=0 \quad \mbox{in}\quad \mathbb{R}^n,
    \end{equation}
is the Euler-Lagrange equation for the minimization problem \eqref{mini-p}.
Furthermore, the authors in \cite{DAIQIN, GY18, DY19, GHPS19} independently computed the optimal $C_{HLS}>0$ and classified all positive solutions of \eqref{ele-1.1-1} as functions of the form
\begin{equation}\label{defU}
W[\xi,\mu](x)=\tilde{c}_{n,\alpha}\big(\frac{\mu}{1+\mu^2|x-\xi|^2}\big)^{\frac{n-2}{2}},\hspace{1mm}\mu\in\mathbb{R}^{+},\hspace{1mm}\xi\in\mathbb{R}^n,
\end{equation}
and obtained the optimal constant in \eqref{mini-p}
    \begin{equation*}
C_{HLS}=S\left[\frac{\Gamma((n-\alpha)/2)\pi^{\alpha/2}}{\Gamma(n-\alpha/2)}\left(\frac{\Gamma(n)}{\Gamma(n/2)}\right)^{1-\frac{\alpha}{n}}\right]^{(2-n)/(2n-\alpha)}.
    \end{equation*}
Here the constant $\tilde{c}_{n,\alpha}$ is given by
\begin{equation}\label{fU}
    \tilde{c}_{n,\alpha}:=\big[n(n-2)\big]^{\frac{n-2}{4}}S^{\frac{(n-\alpha)(2-n)}{4(n-\alpha+2)}}C_{n,\alpha}^{\frac{2-n}{2(n-\alpha+2)}},
      \end{equation}
where $S$ is the best Sobolev constant and $C_{n,\alpha}$ is the best constant of the following classical Hardy-Littlewood-Sobolev inequality \cite{H-L-1928,S1963}
    \begin{equation}\label{hlsi}
    \int_{\mathbb{R}^n}\int_{\mathbb{R}^n}f(x)|x-y|^{-\alpha} g(y)dxdy\leq C(n,r,t,\alpha)\|f\|_{L^r(\mathbb{R}^n)}\|g\|_{L^t(\mathbb{R}^n)}
    \end{equation}
with $\alpha\in(0,n)$, $1<r,t<\infty$ and $\frac{1}{r}+\frac{1}{t}+\frac{\alpha}{n}=2$. Moreover, in the specific case where $t=r=\frac{2n}{2n-\alpha}$, Lieb in \cite{Lieb83} classified the extremal function getting a sharp constant by rearrangement and symmetrisation:
    \begin{equation}\label{defhlsbc}
    C_{n,r,t,\alpha}:=C_{n,\alpha}=\frac{\Gamma((n-\alpha)/2)\pi^{\alpha/2}}{\Gamma(n-\alpha/2)}\left(\frac{\Gamma(n)}{\Gamma(n/2)}\right)^{1-\frac{\alpha}{n}},
    \end{equation}
   and the equality holds if and only if
    \begin{equation*}
   f(x)=g(x)=a\Big(\frac{1}{1+\mu^2|x-x_0|^2}\Big)^{\frac{2n-\alpha}{2}}
    \end{equation*}
    for some $a\in \mathbb{C}$, $\mu\in \mathbb{R}\backslash\{0\}$ and $x_0\in \mathbb{R}^n$.

Concerning the asymptotic behavior of the least energy solutions of \eqref{ele-1.1} as $\epsilon\rightarrow0$ has already been studied in \cite{Cingolani} for the cases $p_\alpha:=p=\frac{n+2}{n-2}$ and $n=3,4,5$
 if
\begin{equation}\label{minimi}
\frac{\int_{\Omega}|\nabla u_{\epsilon}|^2dx}{\left[\int_{\Omega}(|x|^{-(n-2)} \ast|u_{\epsilon}|^{p-\epsilon})|u_{\epsilon}|^{p-\epsilon} dx\right]^{\frac{1}{p-\epsilon}}}=C_{HLS}+o(1)\quad\mbox{as}\quad\epsilon\rightarrow0
\end{equation}
for some positive constant $C_{HLS}$ depending only on $n$.

In addition to the nonlocal equations \eqref{ele-1.1} with slightly subcritical exponents, the asymptotic behavior of solutions of the Hartree-type Brezi-Nirenberg problem and the qualitative characteristics of the solutions to its corresponding eigenvalue problem have been extensively studied in \cite{yyz,YZ} and \cite{PWY}, respectively.
Regarding studies on the nonlinear Choquard equation in the bounded domains and related topics, we refer the reader to \cite{Cingolani-1,GHP,GHP-1,V-Moroz,Moroz-2} and references therein.

In particular, Cingolani et al. in \cite{Cingolani} established the following results:\\
\\
{\it {\bf Theorem A.}
 Assume that $n=3,4,5$, $p=\frac{n+2}{n-2}$ and $\epsilon$ is sufficiently small. Let $u_\epsilon$ be a solution of \eqref{ele-1.1} satisfying \eqref{minimi}.
 Then if $x_{\epsilon}$ is a maximum point of $u_{\epsilon}$, i.e. $u_{\epsilon}(x_{\epsilon})=\|u_{\epsilon}\|_{L^{\infty}(\Omega)}$, we have that $x_{\epsilon}\rightarrow x_0\in\Omega$ as $\epsilon\rightarrow0$ and the following results hold:
\begin{itemize}
\item[$(a)$]
$x_0$ is a critical point of the Robin's function $\phi(x)$ of $\Omega$ and
\begin{equation}\label{Fn}
\lim\limits_{\epsilon\rightarrow0}\epsilon\|u_\epsilon\|_{L^{\infty}(\Omega)}^2=F_n,
\end{equation}
where $F_n=-\widetilde{\alpha}_{n}^2\frac{n+2}{n-2}\big[\frac{1}{C_{HLS}}\big]^{\frac{n+2}{4}}\big\|W\big\|^{2(2^{\ast}-1)}_{L^{2^{\ast}-1}(\mathbb{R}^n)}\phi(x_0)$,  with
$$\widetilde{\alpha}_{n}=\left[\frac{n(n-2)}{\sqrt{S}}\right]^{\frac{n-2}{4}} \frac{\pi^{\frac{n}{2}}\Gamma(1)}{\Gamma(\frac{n+2}{2})}\left\{\frac{\Gamma(1)\pi^{(n-2)/2}}{\Gamma((n+2)/2)}\left(\frac{\Gamma(n)}{\Gamma(n/2)}\right)^{2/n}\right\}^{\frac{2-n}{8}}, $$
 where $2^*:=\frac{2n}{n-2}$ is a critical Sobolev exponent,  $\Gamma(s)=\int_0^{+\infty} x^{s-1}e^{-x}\,dx$ and $S$ is the best Sobolev constant in $\mathbb{R}^n$.
\item[$(b)$] For any $x\in\Omega\setminus\{x_0\}$, it holds that
\begin{equation}\label{ifini}
\lim\limits_{\epsilon\rightarrow0^{+}}\|u_{\epsilon}\|_{L^{\infty}(\Omega)}u_{\epsilon}(x)=\mathcal{K}_nG(x,x_{0})\hspace{2mm}\mbox{in} \hspace{2mm}C^{1}(\Omega\setminus\{x_0\})
\end{equation}
with $\mathcal{K}_n=\widetilde{\alpha}_{n}\big\|W\big\|^{2^{\ast}-1}_{L^{2^{\ast}-1}(\mathbb{R}^n)}$,
where $\widetilde{\alpha}_{n}$ is defined in $(a)$ and $G$ is the Green function.
\item[$(c)$]
\begin{equation}\label{bepus}
\lim\limits_{\epsilon\rightarrow0^{+}}\|u_{\epsilon}\|^{\epsilon}_{L^{\infty}(\Omega)}=1.
\end{equation}
\item[$(d)$] There exists a positive constant $C$, independent of $\varepsilon$, such that
\begin{equation}\label{dus}
u_{\epsilon}(x)\leq C\|u_{\epsilon}\|_{L^{\infty}(\Omega)}W\left(\|u_{\epsilon}\|_{L^{\infty}(\Omega)}^{\frac{4-(n-2)\varepsilon}{2(n-2)}}(x-x_{\varepsilon})\right).
\end{equation}
\end{itemize}
}

\par
	
In this line of research, the aim of this paper is to study qualitative properties of single blow-up solutions to the Hartree-type eigenvalue problem \begin{equation}\label{ele-2}
\left\lbrace
\begin{aligned}
    &-\Delta v=\lambda \mathcal{G}_{\epsilon}[v] \quad \mbox{in}\quad \Omega,\\
    &v=0\quad \mbox{on}\hspace{2.5mm}\partial\Omega,\hspace{2mm}\|v\|_{L^{\infty}(\Omega)}=1,
   \end{aligned}
\right.
\end{equation}
with
	\begin{equation}\label{defanndg2}
		\mathcal{G}_{\epsilon}[v]:=(p-\epsilon)\Big(|x|^{-(n-2)} \ast \big(u_{\epsilon}^{p-1-\epsilon}v\big)\Big)u_{\epsilon}^{p-1-\epsilon}
		+(p-1-\epsilon)\Big(|x|^{-(n-2)} \ast u_{\epsilon}^{p-\epsilon}\Big)u_{\epsilon}^{p-2-\epsilon}v,
	\end{equation}
where $p=\frac{n+2}{n-2}>2$ is the $\mathcal{D}^{1,2}$-energy-critical exponent and $\epsilon>0$ is close to zero.

Firstly, we concentrate on behavior of the first eigenvalue and eigenvector. Let $\tilde{v}_{i,\epsilon}$ be a dilation of $v_{i,\epsilon}$ corresponding to $\lambda_{i,\epsilon}$ defined as
\begin{equation}\label{vie}
\tilde{v}_{i,\epsilon}(x)=v_{i,\epsilon}\left(\big\|u_{\varepsilon}\big\|_{L^{\infty}(\Omega)}^{-(\frac{4-(n-2)\epsilon}{2(n-2)})}x+x_{\epsilon}\right)\quad\mbox{for}\quad x\in\Omega_{\epsilon}:=\big\|u_{\epsilon}\big\|_{L^{\infty}(\Omega)}^{\frac{4-(n-2)\epsilon}{2(n-2)}}(\Omega-x_{\epsilon}),
\end{equation} for $i=1,\cdots,n+2$.
\begin{thm}\label{Figalli-1}
		Assume that $n=3,4,5$, $p=\frac{n+2}{n-2}$ and $\epsilon$ is sufficiently small. Then
	\begin{equation}\label{vie-1-1}	
\lambda_{1,\epsilon}\rightarrow(2p-1)^{-1}\quad\mbox{and}\quad\tilde{v}_{1,\epsilon}\rightarrow W[0,1](x)\quad\mbox{in}\quad C_{loc}^{1}(\mathbb{R}^n).
\end{equation}
Moreover, if $\epsilon$ is small enough, the eigenvalue $\lambda_{1,\epsilon}$ is simple and
$$
\lim\limits_{\epsilon\rightarrow0^{+}}\big\|u_{\epsilon}\big\|^2_{L^{\infty}(\Omega)}v_{1,\epsilon}(x)=   \Gamma_n\tilde{c}_{n,n-2}^{2(p-2)}\big\|W[0,1]\big\|^{p}_{L^{p}(\mathbb{R}^n)}G(x,x_{0})$$
with
$$\Gamma_n:=I(\frac{n-2}{2})S^{\frac{(2-n)}{8}}C_{n,n-2}^{\frac{2-n}{8}}[n(n-2)]^{\frac{n-2}{4}},\,\, I(s)=\frac{\pi^{\frac{n}{2}}\Gamma(\frac{n-2s}{2})}{\Gamma(n-s)} \ \ \mbox{and }\ \ \Gamma(s)=\int_0^{+\infty} x^{s-1}e^{-x}\,dx,~s>0$$
Here $C_{n,n-2}$ is defined in \eqref{defhlsbc} and the convergence is in $C^{1,\tilde{\alpha}}(\omega)$ with $\omega$ any subdomain of $\Omega$ not containing $x_0$, and $G$ denotes the
Green's function of the Laplacian $-\Delta$ in $\Omega$ with the Dirichlet boundary condition.
\end{thm}

Furthermore, we can see the Robin's function as $\phi(x)=H(x,x)$,  where  $H(x,y)$ is given as follows
\begin{equation*}
\Delta H(x,y) =0\quad\mbox{in}\quad\Omega,\quad\quad
	H(x,y)=-\frac{1}{(n-2)\omega_n|x-y|^{n-2}}\quad\mbox{on}\quad\partial\Omega.
\end{equation*}
The function $H$ is nothing but the regular part of the Green function. Indeed then we have
   \begin{equation}\label{Robin}
	H(x,y)=G(x,y)-\frac{1}{(n-2)\omega_n|x-y|^{n-2}},
\end{equation}
    where $\omega_n$ is the area of the unit sphere in $\mathbb{R}^n$.	
\begin{Rem}
It is easy to see from the maximum principle that $\phi(x)<0$ in $\Omega$.
\end{Rem}	
We have the following asymptotic behavior of the eigenfunctions $\tilde{v}_{i,\epsilon}$ for $i=2,\cdots,n+1$.
	\begin{thm}\label{Figalli}
		Assume that $n=3,4,5$, $p=\frac{n+2}{n-2}$ and $\epsilon$ is sufficiently small. Then
	\begin{equation}\label{vie-1}	
\tilde{v}_{i,\epsilon}\rightarrow-(n-2)\sum_{k=1}^{n}\frac{\alpha_k^ix_k}{(1+|x|^2)^{\frac{n}{2}}}\quad\mbox{as}\hspace{2mm}\epsilon\rightarrow0\hspace{2mm}\mbox{in}\hspace{2mm}C_{loc}^{1}(\mathbb{R}^n),
\end{equation}	
\begin{equation}\label{vie-2}	
\frac{v_{i,\epsilon}}{\epsilon^{\frac{n-1}{n-2}}}\rightarrow A_0\sum_{k=1}^{n}\alpha_k^i\frac{\partial G}{\partial w_k}(x,x_0)\quad\mbox{as}\hspace{2mm}\epsilon\rightarrow0\hspace{2mm}\mbox{in}\hspace{2mm}C_{loc}^{1}(\overline{\Omega}\setminus\{x_0\}),
\end{equation}	
for $i=2,\cdots,n+1$,
 where the vectors  $\alpha^{i}=(\alpha_1^{i},\cdots,\alpha_{n}^i)\neq0$ in $\mathbb{R}^n$ and some suitable constant $A_0\in\mathbb{R}$ and satisfying
$$A_0=\frac{(p-2)\tilde{c}_{n,n-2}^{2(p-2)}}{pF_n^{\frac{n-1}{n-2}}}\frac{\pi^{\frac{n}{2}}\Gamma(1)}{\Gamma(\frac{n+2}{2})}S^{\frac{(2-n)}{8}}C_{n,n-2}^{\frac{2-n}{8}}[n(n-2)]^{\frac{n-2}{4}}\int_{\mathbb{R}^n}W^{p}[0,1](w)dw
$$
and $F_n=-\widetilde{\alpha}_{n}^2\frac{n+2}{n-2}\big[\frac{1}{C_{HLS}}\big]^{\frac{n+2}{4}}\big\|W[0,1]\big\|^{2p}_{L^{p}(\mathbb{R}^n)}\phi(x_0)>0$ with
$$\widetilde{\alpha}_{n}=\left[\frac{n(n-2)}{\sqrt{S}}\right]^{\frac{n-2}{4}} \frac{\pi^{\frac{n}{2}}\Gamma(1)}{\Gamma(\frac{n+2}{2})}\left\{\frac{\Gamma(1)\pi^{(n-2)/2}}{\Gamma((n+2)/2)}\left(\frac{\Gamma(n)}{\Gamma(n/2)}\right)^{2/n}\right\}^{\frac{2-n}{8}}, $$
where $\Gamma(s)=\int_0^{+\infty} x^{s-1}e^{-x}\,dx,~s>0$ and $S$ is the best Sobolev constant in $\mathbb{R}^n$.
	\end{thm}
	For a further understanding of the eigenvalue behavior  $\lambda_{i,\epsilon}$, we establish a connection between the eigenvalues $\lambda_{i,\epsilon}$ and the eigenvalues of the Hessian matrix of the Robin function $\phi(x)$ for any $i=2,\cdots,n+1$. Our result can be stated as follows.
	\begin{thm}\label{remainder terms}
		Let the assumptions of Theorem \ref{Figalli} be satisfied. Assume that $\nu_1\leq\nu_2\leq\cdots\leq\nu_n$ the eigenvalues of the Hassian matrix $D^2\phi(x_0)$ of the Robin function at $x_0$.
		Then there exist some constant $\widetilde{\mathcal{H}}>0$ such that for $\varepsilon$ small, it holds
		\begin{equation*}
			\lambda_{i,\epsilon}=1- \widetilde{\mathcal{H}}\nu_{i-1}+o(\epsilon^{\frac{n}{n-2}}) \quad\mbox{for any}\hspace{2mm}i=2,\cdots,n+1,
		\end{equation*}
with $$\widetilde{\mathcal{H}}:=\frac{pn\tilde{\alpha}_n|\phi(x_0)|C_{HLS}^{-\frac{n+2}{4}}}{2(n-2)(p-2)\Gamma_n\tilde{c}_{n,n-2}^{2(p-2)}}\frac{[\int_{\mathbb{R}^n}W[0,1](x)dx]^3}{\int_{\mathbb{R}^n}W^{p-1}[0,1]|\nabla W[0,1]|^2dx}>0.$$
Moreover, the vectors $\alpha^i$ of \eqref{vie-1} are the eigenvectors corresponding to $\nu_i$ for $i=2,\cdots,n+1.$
	\end{thm}
It is natural to investigate the $(n+2)$-th eigenvalue of \eqref{ele-2}-\eqref{defanndg2} and some qualitative properties of the corresponding eigenfunction. Then we establish the following results.
	\begin{thm}\label{thmprtb}
		Let the assumptions of Theorem \ref{Figalli} be satisfied.
		Then we have
		\begin{equation}\label{rtbdy46}	
\tilde{v}_{n+2,\epsilon}\rightarrow \beta \frac{1-|x|^2}{(1+|x|^2)^{\frac{n}{2}}}\quad\mbox{as}\hspace{2mm}\epsilon\rightarrow0\hspace{2mm}\mbox{in}\hspace{2mm}C_{loc}^{1}(\overline{\Omega}\setminus\{x_0\})
\end{equation}
with $\beta\neq0$ and there exist a constant $\mathcal{C}_0$ such that
	\begin{equation}\label{rt}	
\lambda_{n+2,\epsilon}=1-\epsilon(\mathcal{C}_0+o(1))
\end{equation}
with
$$
\mathcal{C}_0=-\frac{(n-2)\mathcal{K}_n \mathcal{M}_0F_n}{(2p-1)\Gamma_n\tilde{c}_{n,n-2}^{2(p-2)}}\Big(\int_{\mathbb{R}^n}W^{p-1}[0,1](x)\big(\frac{1-|x|^2}{(1+|x|^2)^{\frac{n}{2}}}\big)^2dx\Big)^{-1}\phi(x_0)<0,$$
where $F_n$ is defined in Theorem \ref{Figalli}.
Furthermore, if $\epsilon$ is small enough, the eigenvalue $\lambda_{n+2,\epsilon}$ is simple and corresponding eigenfunction $v_{n+2,\epsilon}$ has only two nodal regions, and the closure of the nodal set of the eigenfunction $v_{n+2,\epsilon}$ does not intersect the boundary of $\Omega$.
\end{thm}
	
As a consequence of Theorems \ref{Figalli-1}-\ref{remainder terms}, we have the following Corollary.
\begin{cor}\label{emm-1}
Let the assumptions of Theorem \ref{Figalli} be satisfied. Let $ind(u_{\epsilon})$ and $ind(-D^2\phi(x_0))$ be a the Morse index of the solution $u_{\epsilon}$ to
\eqref{ele-1.1} and a critical point $x_0$ of the function $\phi(x)$. Then
$$1\leq1+ind(-D^2\phi(x_0))\leq ind(u_{\epsilon})\leq n+1$$
for sufficiently small $\epsilon>0$. Therefore if the matrix $D^2\phi(x_0)$ non-degenerate, then we have
$$ind(u_{\epsilon})=1+ind(-D^2\phi(x_0)).$$
\end{cor}	

\begin{Rem}
By the discussion before, our results hold for the dimensions $n=3,4,5$ and the Newtonian potential instead of the Riesz potential. We believe however that this restriction is only technical and we conjecture that the results of this paper also hold for the following parameter region:
\begin{equation}\label{paramater}
\alpha=n-2\hspace{2mm}\mbox{if}\hspace{2mm}n\geq6, \hspace{2mm}\alpha \in(0,n)\hspace{2mm}\mbox{with}\hspace{2mm}(0,4]\hspace{2mm}\mbox{but}\hspace{2mm}\alpha\hspace{2mm}\neq n-2\hspace{2mm}\mbox{if}\hspace{2mm}n\geq3.
\end{equation}
\end{Rem}
Finally, we would like to mention some related results to our problem. The problem \eqref{ele-1.1} can be understood as the nonlocal version of the following problem with nonlinearities of slightly subcritical growth
 \begin{equation}\label{SC}
 -\Delta u
=u^{\frac{n+2}{n-2}-\epsilon}
 ,~~~u>0,~~~\text{in}~~\Omega,~~~
 		u=0,~~~\text{on}~~\partial\Omega,
 \end{equation}
 where $\Omega$ is a bounded domain in $\mathbb{R}^n$ for $n\geq3$ and $\epsilon$ is a small parameter. The studies on the asymptotic behavior of radial solutions of \eqref{SC} were initiated by Atkinson and Peletier in \cite{ATKINSON-1986} by using ODE technique in the unit ball of $\mathbb{R}^3$. Later, Brezis and Peletier \cite{BP} used the method of PDE to obtain the same results as that in \cite{ATKINSON-1986} for the spherical domains. Finally, the same kind of results hold for nonspherical domain, which was settled by Han in \cite{HANZCHAO} (independently by Rey in \cite{Rey-1989}). Moreover, this result was extended in \cite{Musso-Pistoia-2002}, where Musso and Pistoia obtained the existence of multi-peak solutions for certain domains, the elliptic systems in \cite{CKIM-1,Guerra} and the fractional cases \cite{CKL}. (See also \cite{LiWZ,Rey-1990,JW0,GW,HL}).
 Later, Grossi and Pacella in \cite{GP05} obtained the asymptotic behavior of the eigenvalues $\lambda$ for the problem
 \begin{equation}\label{SC-1}
 -\Delta v
=n(n-2)\lambda u_{\epsilon}^{p-\epsilon}v
 ,~~~v>0,~~~\text{in}~~\Omega,~~~
 		v=0,~~~\text{on}~~\partial\Omega,
 \end{equation}
and the Morse index of $u_{\epsilon}$, as well as their result was extended later to the multi-bubble solutions case in \cite{CKL-1}. For one-bubble solutions case,  Takahashi in \cite{T-1} analyzed the linearized Brezis-Nirenberg problem. In \cite{GG09}, Gladiali and Grossi studied the asymptotic behavior of the eigenvalues and Morse index of one-bubble solution for the linearized Gelfand problem. Later that, the corresponding problem of multi-bubble solutions was proved in \cite{GGO-1}, and further the qualitative properties of the first $m$ eigenfunctions was established by Gladiali et al. in \cite{GGO}. For the studies of the related topic, see also \cite{B-L-R,dgp,LTX} and the references therein for earlier works. Our results are motivated by the work of Gross and Pacella \cite{GP05} on the classical local problem \eqref{SC-1} and it can be viewed as non-local counterpart of the qualitative properties results in \cite{GP05}.
	
\subsection{Structure of the paper}
The paper is organized as follows. In section \ref{section2}, we establish a decay estimate of the rescaled eigenfunction $\tilde{v}_{i,\epsilon}$ for $i\in\mathbb{N}$. In section \ref{asymptotic}, we are devoted to prove Theorem \ref{Figalli-1}. In section \ref{section3-1}, we give the estimates of the eigenvalues $\lambda_{i,\epsilon}$ for $i=2,\dots,n+1$, and further combined with weak limit characterization of the rescaled eigenfunction $v_{i,\epsilon}$ in section \ref{section4}, we then complete the proof of Theorem \ref{Figalli}. In section \ref{sangshen}, we obtain the asymptotic behaviour of the eigenvalues $\lambda_{i,\epsilon}$ for $i=2,\cdots,n+1$, and we establish Theorem \ref{remainder terms}. Finally, in section \ref{section7}, we establish the qualitative characteristics for $(n+2)$-th eigenvalue and corresponding to eigenfunction, and the Morse index of $u_{\epsilon}$, that is, Theorem \ref{thmprtb} and Corollary \ref{emm-1}.
	
Throughout this paper, $c$ and $C$ are indiscriminately used to denote various absolutely positive constants. We will use big $O$ and small $o$ notations to describe the limit behavior of a certain quantity as $\epsilon\rightarrow0$.
	
\section{Decay estimate of the rescaled eigenfunction $\tilde{v}_{i,\epsilon}$}\label{section2}
In this section, we are devoted to establish a decay estimate of rescaled least energy solutions $v_{i,\epsilon}$ for $i\in\mathbb{N}$. For the simplicity of notations, we write $W(x)$ instead of $W[0,1](x)$ and $\|\cdot\|_{\infty}$ instead of the norm of $L^{\infty}(\Omega)$ in the sequel, namely that $W(x)$ solves
\begin{equation}\label{wx}
    -\Delta u=C_N(|x|^{-(n-2)}\ast u^{p})u^{p-1} \quad \mbox{in}\quad \mathbb{R}^n,
    \end{equation}
where $C_N=\tilde{c}_{n,n-2}^{2(p-2)}$ and $p=\frac{n+2}{n-2}$.

Let us choose $x_\epsilon\in\Omega$ and the number $\lambda_{\epsilon}>0$ by
\begin{equation}\label{miu}
\mu_{\epsilon}^{\frac{2(n-2)}{4-(n-2)\epsilon}}=\big\|u_{\epsilon}\big\|_{\infty}=u_\epsilon(x_\epsilon).
\end{equation}
We define a family of rescaled functions
\begin{equation}\label{miu-1}
\tilde{u}_{\epsilon}(x)=\mu_{\epsilon}^{-\frac{2(n-2)}{4-(n-2)\epsilon}}u_{\epsilon}(\mu_{\epsilon}^{-1}x+x_{\epsilon})\quad\mbox{for}\quad x\in\Omega_{\epsilon}=\mu_{\epsilon}(\Omega-x_{\epsilon}).
\end{equation}	
It follows from \eqref{dus} that
\begin{equation}\label{decay1}
\tilde{u}_{\epsilon}(x)\leq CW(x)\quad \mbox{for}\quad x\in\mathbb{R}^n.
\end{equation}
We first state the following lemma which is useful in our analysis (see \cite{WY} for the proof).
\begin{lem}\label{Lem6.1}
For any constant $0<\sigma<n-2$, there is a constant $C>0$ such that
\begin{equation}\label{Lem6.1-0}
\int_{\mathbb{R}^n}\frac{1}{|y-x|^{n-2}}\frac{1}{(1+|x|)^{2+\sigma}}dx\leq
\frac{C}{(1+|y|)^\sigma}.
\end{equation}
\end{lem}
\begin{lem}\label{L6.2}
For any constant $\sigma\geq n-2-\frac{\alpha}{2}$ and $\alpha\in(0,4)$, there is a constant $C>0$ such that
\begin{equation}\label{L6.2-0}
\int_{\mathbb{R}^n}\frac{1}{|y-x|^{\frac{2n(n-2)}{2n-\alpha}}}\frac{1}{(1+|x|)^{\frac{2n(2+\sigma)}{2n-\alpha}}}dx\leq
\frac{C}{(1+|y|)^\frac{n(2\sigma+\alpha)}{2n-\alpha}}.
\end{equation}
\end{lem}
\begin{proof}
See Lemma 3.6 in \cite{SYZ}.
\end{proof}
We conduct a decay estimate for solutions of the eigenvalue problem \eqref{ele-2}-\eqref{defanndg2}.
\begin{lem}
Assume that $v_{i,\epsilon}$ is a solution of \eqref{ele-2}-\eqref{defanndg2}, $i\in\mathbb{N}$. Then there exists a positive constant $C$, independent of $\varepsilon$, such that
\begin{equation}\label{decay22}
v_{i,\epsilon}(x)\leq CW\left(\|u_{\epsilon}\|_{\infty}^{\frac{4-(n-2)\epsilon}{2(n-2)}}(x-x_{\epsilon})\right)\quad \mbox{for}\quad x\in\mathbb{R}^n,
\end{equation}
for each point $x\in\Omega$, and there holds
\begin{equation}\label{day2}
|\tilde{v}_{i,\epsilon}(x)|\leq CW(x)\quad \mbox{for}\quad x\in\mathbb{R}^n,
\end{equation}
where $\tilde{v}_{i,\epsilon}$ can be found in \eqref{vie}.	
\end{lem}	
\begin{proof}
We have that the function $v_{i,\epsilon}$ verifies,
\begin{equation}\label{e-1}
\left\lbrace
\begin{aligned}
    &-\Delta v_{i,\epsilon}=\lambda_{i,\epsilon}C_N \mathcal{G}_{\epsilon}[v_{i,\epsilon}] \quad \mbox{in}\quad \Omega,\\
    &v_{i,\epsilon}=0\quad \mbox{on}\hspace{2.5mm}\partial\Omega,\hspace{2mm}
    \big\|v_{i,\epsilon}\big\|_{\infty}=1,
   \end{aligned}
\right.
\end{equation}
with
	\begin{equation}\label{de-1}
		\mathcal{G}_{\epsilon}[v_{i,\epsilon}]:=(p-\epsilon)\Big(|x|^{-(n-2)} \ast \big(u_{\epsilon}^{p-1-\epsilon}v_{i,\epsilon}\big)\Big)u_{\epsilon}^{p-1-\epsilon}
		+(p-1-\epsilon)\Big(|x|^{-(n-2)} \ast u_{\epsilon}^{p-\epsilon}\Big)u_{\epsilon}^{p-2-\epsilon}v_{i,\epsilon}.
	\end{equation}
Then we get from the integral representation of $v_{i,\epsilon}$
that
\begin{equation*}
v_{i,\epsilon}(x)=\lambda_{i,\epsilon}C_N
\Big[\int_{\Omega}\int_{\Omega}\frac{u_\epsilon^{p-1-\epsilon}(z)v_{i,\epsilon}(z)u_\epsilon^{p-1-\epsilon}(y)}{|y-z|^{n-2}}G(x,y)dydz+\int_{\Omega}\int_{\Omega}\frac{u_\epsilon^{p-\epsilon}(z)u_\epsilon^{p-2-\epsilon}(y)v_{i,\epsilon}(y)}{|y-z|^{n-2}}G(x,y)dydz\Big].
\end{equation*}
Taking $\alpha=n-2$ and $\sigma=2$ in \eqref{L6.2-0} and by \eqref{decay1}-\eqref{Lem6.1}, \eqref{e-1}, \eqref{Gx}, the definition of $W$ and the Hardy-Littlewood-Sobolev inequality, we have
\begin{equation*}
\begin{split}
&\int_{\Omega}\int_{\Omega}\frac{u_\epsilon^{p-1-\epsilon}(z)v_{i,\epsilon}(z)u_\epsilon^{p-1-\epsilon}(y)}{|y-z|^{n-2}}G(x,y)dydz\\&
\leq C\int_{\Omega}\int_{\Omega} \big\|u_{\epsilon}\big\|_{\infty}^{p-1-\varepsilon}\frac{W^{p-1-\epsilon}\big(\|u_{\epsilon}\|_{\infty}^{\frac{4-(n-2)\epsilon}{2(n-2)}}(y-x_{\epsilon})\big)}{|x-y|^{n-2}}\big\|u_{\epsilon}\big\|_{\infty}^{p-1-\varepsilon}\frac{W^{p-1-\epsilon}\big(\|u_{\epsilon}\|_{\infty}^{\frac{4-(n-2)\epsilon}{2(n-2)}}(z-x_{\epsilon})\big)}{|y-z|^{n-2}}dydz
+O(\epsilon^2)\\&
\leq C\bigg[\int_{\mathbb{R}^n}\frac{1}{|x-y|^{\frac{2n(n-2)}{n+2}}}\frac{\|u_{\epsilon}\|_{\infty}^{\frac{2n(p-1-\varepsilon)}{n+2}}}{(1+\|u_{\epsilon}\|_{\infty}^{\frac{4-(n-2)\epsilon}{2(n-2)}}|y-x_{\epsilon}|)^{\frac{2n(n-2)(p-1)}{n+2}}}dy\bigg]^{\frac{n+2}{2n}}
\\&
\hspace{3mm}\times C\bigg[\int_{\mathbb{R}^n}\frac{\|u_{\epsilon}\|_{\infty}^{\frac{2n(p-1-\varepsilon)}{n+2}}}{(1+\|u_{\epsilon}\|_{\infty}^{\frac{4-(n-2)\epsilon}{2(n-2)}}|z-x_{\epsilon}|)^{\frac{2n(n-2)(p-1)}{n+2}}}dz\bigg]^{\frac{n+2}{2n}}
+O(\epsilon^2)\\&\leq C \bigg[\int_{\mathbb{R}^n}\frac{1}{\big|\|u_{\epsilon}\|_{\infty}^{\frac{4-(n-2)\epsilon}{2(n-2)}}(x-x_{\epsilon})+y\big|^{\frac{2n(n-2)}{n+2}}}\frac{1}{(1+\big|y\big|)^{\frac{2n(n-2)(p-1)}{n+2}}}dy\bigg]^{\frac{n+2}{2n}}
\\&\leq C\bigg(\frac{1}{1+\|u_{\epsilon}\|_{\infty}^{\frac{4-(n-2)\epsilon}{2(n-2)}}|x-x_{\epsilon}|}\bigg)^{\frac{n+2}{2}}+O(\epsilon^2),
\end{split}
\end{equation*}
where we have exploited the fact that
\begin{equation*}
\begin{split}
W^{p-\epsilon}=W^{p}(1-\epsilon\log W+O(\epsilon^2))&=W^{p}-\epsilon W^{p}\log W+O(\epsilon^2W^{p}))
\end{split}
\end{equation*}
by Taylor's theorem.
Combining the following identity (see \cite{gyz})
\begin{equation}\label{p1-00}
\int_{\mathbb{R}^n}\frac{W^{p}(y)}{|x-y|^{n-2}}dy
=I(\frac{n-2}{2})S^{\frac{(2-n)}{8}}C_{n,n-2}^{\frac{2-n}{8}}[n(n-2)]^{\frac{n-2}{4}}W^{2^{\ast}-p}(x)=:\Gamma_nW^{2^{\ast}-p}(x),
\end{equation}
where
$$
I(s)=\frac{\pi^{\frac{n}{2}}\Gamma(\frac{n-2s}{2})}{\Gamma(n-s)}, \ \ \mbox{and }\Gamma(s)=\int_0^{+\infty} x^{s-1}e^{-x}\,dx,~s>0,
$$
then we similarly compute and get
\begin{equation*}
\int_{\Omega}\int_{\Omega}\frac{u_\epsilon^{p-\epsilon}(z)u_\epsilon^{p-2-\epsilon}(y)v_{i,\epsilon}(y)}{|y-z|^{n-2}}G(x,y)dydz\leq
C\bigg(\frac{1}{1+\|u_{\epsilon}\|_{\infty}^{\frac{4-(n-2)\epsilon}{2(n-2)}}|x-x_{\epsilon}|}\bigg)^{2}+O(\epsilon^2).
\end{equation*}
Here and several times in the sequel, we use the elementary identity \eqref{p1-00}.
As a consequence,
$$v_{i,\epsilon}\leq C\bigg(\frac{1}{1+\|u_{\epsilon}\|_{\infty}^{\frac{4-(n-2)\epsilon}{2(n-2)}}|x-x_{\epsilon}|}\bigg)^{2}+O(\epsilon^2).$$
Next repeating the above process, we know
\begin{equation*}
v_{i,\epsilon}\leq C\bigg(\frac{1}{1+\|u_{\epsilon}\|_{\infty}^{\frac{4-(n-2)\epsilon}{2(n-2)}}|x-x_{\epsilon}|}\bigg)^{4}+O(\epsilon^4).
\end{equation*}
Finally, we can proceed as in the above argument for finite number of times to  conclude that \eqref{decay22}, as desired.
\end{proof}

\section{Asymptotic behavior of the eigenpair $(\lambda_{1,\epsilon},v_{1,\epsilon})$}\label{asymptotic}
In this section, we are devoted to prove Theorem \ref{Figalli-1}. Before proving the main theorem, we need the following result.
\begin{Prop}\label{prondgr}
Assume $n\geq 3$, $0<\alpha<n$ with $0<\alpha\leq4$ and $p^{\ast}=\frac{2n-\alpha}{n-2}$. Then the space of all bounded solutions
to the linearized equation
\begin{equation*}
-\Delta v-pC_N\left(|x|^{-\alpha} \ast (W^{p-1}[0,1]v)\right)W^{p-1}[0,1]
-(p-1)C_N\left(|x|^{-\alpha} \ast W^{p}[0,1]\right)W^{p-2}[0,1]v=0
\end{equation*}
is spanned by in $\mathcal{D}^{1,2}(\mathbb{R}^n)$ of the form
\begin{equation*}
-(n-2)\frac{x_1}{(1+|x|^2)^{\frac{n}{2}}},\cdots,-(n-2)\frac{x_n}{(1+|x|^2)^{\frac{n}{2}}}\quad\mbox{and}\quad\frac{1-|x|^2}{(1+|x|^2)^{\frac{n}{2}}}.
\end{equation*}
\end{Prop}
\begin{proof}
For the proof, refer to \cite{DSB213, gyz,XLi}.
\end{proof}
From Proposition \ref{prondgr} regarding the nondegeneracy of $W[0,1](x)$, we have the following result (cf. \cite{DSB213}).
\begin{Prop}\label{propep}
    Let $\lambda_i$, for $i=1,2,\ldots,$ denote the eigenvalues of the
     following eigenvalue problem
    \begin{equation*}
    -\Delta v
    =\lambda_iC_N\big[p\big(|x|^{-(n-2)} \ast (W^{p-1}[0,1]v)\big)W^{p-1}[0,1]
    +(p-1)\big(|x|^{-(n-2)} \ast W^{p}[0,1]\big)W^{p-2}[0,1]v\big]
    \end{equation*}
for all $v\in \mathcal{D}^{1,2}(\mathbb{R}^n)$. Then $\lambda_1=(2p-1)^{-1}$ is simple and the corresponding eigenfunction is $W[0,1](x)$, and $\lambda_{2}=\lambda_{3}=\cdots=\lambda_{n+2}=1$ with the corresponding $(n+1)$-dimensional eigenfunction space spanned by
    \[
    \left\{\frac{n-2}{2}W[0,1]+x\cdot\nabla W[0,1],\quad \partial_{x_1} W[0,1],\ldots,\partial_{x_n} W[0,1]\right\}.
    \]
    Furthermore $\lambda_{n+2}\leq\lambda_{n+3}\leq\cdots$.
    \end{Prop}	
   Moreover, we need to recall the following regularity result.
    \begin{lem}[\cite{HANZCHAO}]\label{regular}
Let $u$ solve
\begin{equation*}
\begin{cases}
-\Delta u=f\quad\mbox{in}\hspace{2mm}\Omega\subset\mathbb{R}^n,\\
u=0\quad\quad\hspace{2mm}\mbox{on}\hspace{2mm}\partial\Omega.
\end{cases}
\end{equation*}
Then
\begin{equation}\label{regu}
\|u\|_{W^{1,q}(\Omega)}+\|\nabla u\|_{C^{0,\tilde{\alpha}}(\omega^{\prime})}\leq C\big(\|f\|_{L^1(\Omega)}+\|f\|_{L^\infty(\omega)}\big)
\end{equation}
for $q<n/(n-1)$ and $\tilde{\alpha}\in(0,1)$. Here $\omega$ be a neighborhood of $\partial\Omega$ and $\omega^{\prime}\subset\omega$ is a strict subdomain of $\omega$.
\end{lem}
Now we are ready to prove Theorem \ref{Figalli-1}.	
\begin{proof}[\textbf{Proof of Theorem \ref{Figalli-1}}]
First of all, set $p_{\epsilon}^{(1)}=(p-\epsilon),~ p_{\epsilon}^{(2)}=(p-1-\epsilon)$, then by the variational
characterization of the eigenvalue $\lambda_{1,\epsilon}$, we find
\begin{equation*}
\begin{split}
			\lambda_{1,\epsilon}&=\min\limits_{v\in H_{0}^1(\Omega)\setminus\{0\}}\frac{\int_{\Omega}|\nabla f(x)|^2dx}{p_{\epsilon}^{(1)}\int_{\Omega}\Big(|x|^{-(n-2)} \ast \big(u_{\epsilon}^{p-1-\epsilon}f\big)\Big)u_{\epsilon}^{p-1-\epsilon}fdx
		+p_{\varepsilon}^{(2)}\int_{\Omega}\Big(|x|^{-(n-2)} \ast u_{\epsilon}^{p-\epsilon}\Big)u_{\epsilon}^{p-2-\epsilon}f^2dx}.
\end{split}
		\end{equation*}
Choosing $f=u_{\epsilon}$, due to $\tilde{u}_{\epsilon}\rightarrow W[0,1]$ in $H_{0}^1{\Omega}$ as $\epsilon\rightarrow0$, then we get
\begin{equation*}
\begin{split}
			\lambda_{1,\epsilon}&\leq\frac{\int_{\Omega}|\nabla u_{\epsilon}(x)|^2dx}{\big(p_{\epsilon}^{(1)}+p_{\varepsilon}^{(2)}\big)\int_{\Omega}\Big(|x|^{-(n-2)} \ast u_{\epsilon}^{p-\epsilon}\Big)u_{\epsilon}^{p-\epsilon}(x)dx}\\&
=\frac{\int_{\mathbb{R}^n}|\nabla W[0,1](x)|^2dx+o(1)}{\big(2p-1\big)\int_{\mathbb{R}^n}\int_{\mathbb{R}^n}\frac{W^p[0,1](x)W^p[0,1](y)}{|x-y|^{n-2}}dxdy+o(1)}
\end{split}
		\end{equation*}
as $\epsilon\rightarrow0$. This implies that
$$\limsup\limits_{\epsilon\rightarrow0}\lambda_{1,\epsilon}\leq\frac{1}{2p-1}.$$
Then up to a subsequence we have $\lambda_{1,\epsilon}\rightarrow\lambda_1\in[0,(2p-1)^{-1}]$ for $\epsilon$ sufficiently small.
Moreover, similarly to the argument in Lemma \ref{limitlama}, up to a subsequence, there exists a function $f_0$, such that we have $\tilde{v}_{1,\epsilon}\rightarrow f_0$ in $C_{loc}^{1}(\mathbb{R}^n)$, where $f_0$ satisfies
\begin{equation}\label{e-00-1}
\left\lbrace
\begin{aligned}
    &-\Delta f_{0}=p\lambda_1C_N\Big(|x|^{-(n-2)} \ast \big(W^{p-1}[0,1]f_{0}\big)\Big)W^{p-1}[0,1]\\&
		\quad\quad~~+(p-1)\lambda_1C_N\Big(|x|^{-(n-2)} \ast W^{p}[0,1]\Big)W^{p-2}[0,1]f_{0} \quad \mbox{in}\quad \mathbb{R}^n.
   \end{aligned}
\right.
\end{equation}
It follows from Proposition \ref{propep} that
$$\lambda_1=(2p-1)^{-1}\quad\mbox{and}\quad f_0=W[0,1](x),$$
and the result \eqref{vie-1-1} follows.

Now we will show that $\lambda_{1,\epsilon}$ is simple. Then we may assume by contradiction that there exist at least two eigenfunctions $v_{1,\epsilon}^{(1)}$ and $v_{1,\epsilon}^{(2)}$ corresponding to $\lambda_{1,\epsilon}$ orthogonal in the space $H_{0}^{1}(\Omega)$, and by \eqref{vie-1-1} and \eqref{decay1}, as $\epsilon$ small enough we deduce that
\begin{equation*}
	\begin{split}	(p-\epsilon)\int_{\Omega}\int_{\Omega}&\frac{u_{\epsilon}^{p-1-\epsilon}(y){v_{1,\epsilon}}^{(1)}(y)u_{\epsilon}^{p-1-\epsilon}(x)v_{1,\epsilon}^{(2)}(x)}{|x-y|^{n-2}}dxdy \\&+(p-1-\epsilon)\int_{\Omega}\int_{\Omega}\frac{u_{\varepsilon}^{p-\epsilon}(y)u_{\epsilon}^{p-2-\epsilon}(x)v_{1,\epsilon}^{(1)}(x)v_{1,\epsilon}^{(2)}(x)}{|x-y|^{n-2}} dxdy=0\\&
\Rightarrow(2p-1)\int_{\mathbb{R}^n}\int_{\mathbb{R}^n}\frac{W^p[0,1](x)W^p[0,1](y)}{|x-y|^{n-2}}dxdy=0,
\end{split}
\end{equation*}
a contradiction for \eqref{p1-00}.

Furthermore, we have
We have
\begin{equation}\label{laplacian}
\begin{split}
-\Delta\big(\big\|u_{\epsilon}\big\|^2_{\infty}v_{1,\epsilon}\big)=&\lambda_{1,\epsilon}C_N\big\|u_{\epsilon}\big\|^2_{\infty}\Big[(p-\epsilon)\big(|x|^{-{(n-2)}}\ast u_\epsilon^{p-1-\epsilon}v_{1,\epsilon}\big)u_\epsilon^{p-1-\epsilon}\\&+(p-1-\epsilon)\big(|x|^{-{(n-2)}}\ast u_\epsilon^{p-\epsilon}\big)u_\epsilon^{p-2-\epsilon}v_{1,\epsilon}\Big]:=C_N\lambda_{1,\epsilon}\big\|u_{\epsilon}\big\|^2_{\infty}f_{\epsilon}\quad\mbox{in}\quad \Omega.
\end{split}
\end{equation}
We integrate the right-hand side of \eqref{laplacian}
$$\int_{\Omega}\int_{\Omega}\big\|u_{\epsilon}\big\|^2_{\infty}\frac{u_\epsilon^{p-1-\epsilon}(y)v_{1,\epsilon}(y)u_\epsilon^{p-1-\epsilon}(x)}{|x-y^{n-2}} dxdy=\frac{\mu_{\epsilon}^{\frac{4(n+2)}{4-(n-2)\epsilon}}}{\mu_{\epsilon}^{n+2}}\mu_{\epsilon}^{-\frac{4(n-2)\epsilon}{4-(n-2)\epsilon}}\int_{\Omega_{\epsilon}}\int_{\Omega_{\epsilon}}\frac{\tilde{u}_\epsilon^{p-1-\epsilon}(y)\tilde{v}_{1,\epsilon}(y)\tilde{u}_\epsilon^{p-1-\epsilon}(x)}{|x-y|^{n-2}} dxdy,$$
and
$$\int_{\Omega}\int_{\Omega}\big\|u_{\epsilon}\big\|^2_{\infty}\frac{u_\epsilon^{p-\epsilon}(y)u_\epsilon^{p-1-\epsilon}(x)v_{1,\epsilon}(x)}{|x-y|^{n-2}} dxdy=\frac{\mu_{\epsilon}^{\frac{4(n+2)}{4-(n-2)\epsilon}}}{\mu_{\epsilon}^{n+2}}\mu_{\epsilon}^{-\frac{4(n-2)\epsilon}{4-(n-2)\epsilon}}\int_{\Omega_{\epsilon}}\int_{\Omega_{\epsilon}}\frac{\tilde{u}_\epsilon^{p-\epsilon}(y)\tilde{v}_{1,\epsilon}(x)\tilde{u}_\epsilon^{p-2-\epsilon}(x)}{|x-y|^{n-2}} dxdy.$$
Therefore, combining \eqref{vie-1-1}, \eqref{decay1} and \eqref{day2} by dominated convergence,
we get
\begin{equation*}
\begin{split}
\lim\limits_{\epsilon\rightarrow0^{+}}\int_{\Omega}\lambda_{1,\epsilon}\big\|u_{\epsilon}\big\|^2_{\infty}C_Nf_{\epsilon}&=C_N\int_{\mathbb{R}^n}\int_{\mathbb{R}^n}\frac{W^{p}(y)W^{p-1}(x)}{|x-y|^{n-2}}dxdy\\&
=\Gamma_{n}C_N\big\|W\big\|^{p}_{L^{p}(\mathbb{R}^n)}<\infty.
\end{split}
\end{equation*}
Also using the bound \eqref{decay1}, \eqref{day2} and
$$
\int_{\overline{\Omega}\setminus\{x_0\}}\frac{1}{|x-y|^{n-2}}\frac{1}{|y-x_0|^{(n-2)(p-\epsilon)}}dy<\infty,
$$
we find
\begin{equation*}
\begin{split}
\lambda_{1,\epsilon}\big\|u_{\epsilon}\big\|^2_{\infty}C_Nf_{\epsilon}\leq
M\mu_{\epsilon}^{\frac{2(n-2)(p-1-\epsilon)}{4-(n-2)\epsilon}[\frac{n-2}{2}(2p-1-2\epsilon)-2]}\frac{1}{|x-x_{0}|^{(n-2)(p-1-\epsilon)}}
\end{split}
\end{equation*}
for $x\neq x_0$ and some $M>0$. Then $\lambda_{1,\epsilon}\|u_{\epsilon}\|^2_{\infty}f_{\epsilon}\rightarrow0$ for $x\neq x_0$ by \eqref{bepus}.
From here and \eqref{laplacian} we deduce that
$$
-\Delta(\big\|u_{\epsilon}\big\|^2_{\infty}v_{1,\epsilon})\rightarrow\Gamma_{n}C_N\big\|W\big\|^{p}_{L^{p}(\mathbb{R}^n)}\delta_{x=x_0}
$$
in the sense of distributions in $\Omega$ as $\epsilon\rightarrow0$. Let $\omega$ be any neighborhood of $\partial\Omega$ not containing $x_0$. By  Lemma \ref{regular}, we obtain
\begin{equation*}
\begin{split}
\big\|\|u_{\epsilon}\|^2_{\infty}v_{1,\epsilon}\big\|_{C^{1,\tilde{\alpha}}(\omega^{\prime})}&\leq C\Big[\big\|\lambda_{1,\epsilon}\|u_{\epsilon}\|^2_{\infty}C_Nf_{\epsilon}\big\|_{L^1(\Omega)}+\big\|\lambda_{1,\epsilon}\|u_{\epsilon}\|^2_{\infty}C_Nf_{\epsilon}\big\|_{L^\infty(\omega)}\Big].
\end{split}
\end{equation*}
Consequently
$$ \big\|u_{\epsilon}\big\|^2_{\infty}v_{1,\epsilon}\rightarrow\Gamma_{n}C_N\big\|W\big\|^{p}_{L^{p}(\mathbb{R}^n)}G(x,x_0)
\hspace{2mm}\mbox{in}\hspace{2mm}C^{1,\tilde{\alpha}}(\omega)\hspace{2mm}\mbox{as}\hspace{2mm}\epsilon\rightarrow0
$$
for any neighborhood $\omega$ of $\partial\Omega$, not containing $x_0$ and we conclude the proof.
\end{proof}

	\section{Estimates for the eigenvalues $\lambda_{i,\epsilon}$ for $i=2,\dots,n+1$}\label{section3-1}
This section aims to establish the estimates of the eigenvalues $\lambda_{i,\epsilon}$ for $i=2,\dots,n+1$. But first we need some technical results.
Since $x_\epsilon\to x_0$ by Theorem A, exists $\varrho >0$ s.t. $B_{2\varrho}(x_\epsilon) \subset \Omega$.  We can define
$$\psi(x)=\tilde{\psi}(x-x_{\epsilon}),$$
where $\tilde{\psi}$ is a function in $C_{0}^{\infty}(B_{2\varrho}(x_\epsilon))$ such that $\tilde{\psi}\equiv1$ in $B_{\varrho}(x_\epsilon)$, $0\leq\psi\leq1$ in $B_{2\varrho}(x_{\varepsilon})$, and
\begin{equation}\label{varepsilon}
\xi_{j,\epsilon}(x)=\psi(x)\frac{\partial u_{\epsilon}}{\partial x_j}(x),\quad j=1,\cdots,n,
\end{equation}
\begin{equation}\label{varep}
\xi_{n+1,\epsilon}(x)=\psi(x)\left((x-x_{\epsilon})\nabla u_{\epsilon}+\frac{2}{p-1-\epsilon}u_{\epsilon}\right).
\end{equation}

\begin{lem}\label{nabla}
It holds
\begin{equation}\label{dend}
			\begin{split}
-\Delta\left((x-z)\cdot\nabla u_{\varepsilon}+\frac{2}{p-1-\varepsilon}u_{\varepsilon}\right)&=(p-\varepsilon)C_N\Big(\int_{\Omega}\frac{u^{p-1-\varepsilon}\big((y-z)\cdot\nabla u_{\varepsilon}+\frac{2}{p-1-\varepsilon}u_{\varepsilon}\big)}{|x-y|^{n-2}}dy \Big)u_{\varepsilon}^{p-1-\varepsilon}(x)\\ &+(p-1-\varepsilon)C_N\Big(\int_{\Omega}\frac{u^{p-\varepsilon}}{|x-y|^{n-2}}dy \Big)u_{\varepsilon}^{p-2-\varepsilon}\big((x-z)\cdot\nabla u_{\varepsilon}+\frac{2}{p-1-\varepsilon}u_{\varepsilon}\big)
\end{split}
		\end{equation}
for any $z\in\mathbb{R}^n$.
\end{lem}
\begin{proof}
By a straightforward computation, for any $y\in\mathbb{R}^n$ and using that $u_\epsilon$ is a solution of \eqref{ele-1.1},  we get
\begin{equation}\label{comput}
\begin{split}
-\Delta\left((x-z)\cdot\nabla u_{\varepsilon}+\frac{2}{p-1-\varepsilon}u_{\varepsilon}\right)&=\frac{2(p-\varepsilon)}{p-1-\varepsilon}(-\Delta u_{\varepsilon})-\sum_{j=1}^{n}\sum_{i\neq j}^{n}(x_i-z_j)\frac{\partial^3u_{\varepsilon}}{\partial x_i\partial^2x_j}\\&
=\frac{2(p-\varepsilon)}{p-1-\varepsilon}C_N\left(\int_{\Omega}\frac{u_{\varepsilon}^{p-\varepsilon}(y)}{|x-y|^{n-2}}dy \right)u_{\varepsilon}^{p-1-\varepsilon}-\sum_{j=1}^{n}\sum_{i\neq j}^{n}(x_i-z_j)\frac{\partial^3u_{\varepsilon}}{\partial x_i\partial^2x_j}.
\end{split}
\end{equation}
Thus, since
\begin{equation}\label{yield}
\begin{split}
-\sum_{j=1}^{n}\sum_{i\neq j}^{n}(x_i-z_j)\frac{\partial^3u_{\varepsilon}}{\partial x_i\partial^2x_j}&
=C_N\Big(\int_{\Omega}\sum_{i=1}^{n}(x_i-z_i)\frac{\partial}{\partial x_{i}}\big(\frac{1}{|x-y|^{n-2}}\big)u_{\varepsilon}^{p-\varepsilon}(y)dy \Big)u_{\varepsilon}^{p-1-\varepsilon}(x)\\&\hspace{3mm}+(p-1-\varepsilon)C_Nu_{\varepsilon}^{p-2-\varepsilon}(x)(x-z)\cdot\nabla u_{\varepsilon}\left(\int_{\Omega}\frac{u_{\varepsilon}^{p-\varepsilon}(y)}{|x-y|^{n-2}}dy\right),
\end{split}
\end{equation}
and integration by parts yields,
\begin{equation*}
\begin{split}
&\Big(\int_{\Omega}\sum_{i=1}^{n}(x_i-z_i)\frac{\partial}{\partial x_{i}}\big(\frac{1}{|x-y|^{n-2}}\big)u_{\epsilon}^{p-\epsilon}(y)dy \Big)u_{\epsilon}^{p-1-\epsilon}\\
=&\Big(\int_{\Omega}\sum_{i=1}^{n}\frac{y_i-z_i}{|x-y|^{n-2}}\frac{\partial (u_{\epsilon}^{p-\epsilon}(y))}{\partial y_{i}}dy \Big)u_{\epsilon}^{p-1-\epsilon}+\Big(\int_{\Omega}\sum_{i=1}^{n}\frac{x_i-y_i}{|x-y|^{n-2}}\frac{\partial (u_{\epsilon}^{p-\epsilon}(y))}{\partial y_{i}}dy \Big)u_{\epsilon}^{p-1-\epsilon}\\
=&(p-\varepsilon)\Big(\int_{\Omega}\frac{(y-z)\cdot\nabla u_{\varepsilon}u_{\varepsilon}^{p-1-\varepsilon}(y)}{|x-y|^{n-2}}dy \Big)u_{\varepsilon}^{p-1-\varepsilon}-\Big(\int_{\Omega}\sum_{i=1}^{n}\frac{\partial}{\partial y_{i}}\left(\frac{x_i-y_i}{|x-y|^{n-2}}\right)u_{\varepsilon}^{p-\varepsilon}(y)dy \Big)u_{\varepsilon}^{p-1-\varepsilon}\\
=&(p-\epsilon)\Big(\int_{\Omega}\frac{u_{\epsilon}^{p-1-\epsilon}\big((y-z)\cdot\nabla u_{\epsilon}+\frac{2}{p-1-\epsilon}u_{\epsilon}\big)}{|x-y|^{n-2}}dy \Big)u_{\epsilon}^{p-1-\epsilon}+\left(2-\frac{2(p-\epsilon)}{p-1-\epsilon}\right)\Big(\int_{\Omega} \frac{u_{\epsilon}^{p-\epsilon}(y)}{|x-y|^{n-2}}dy\Big)u_{\epsilon}^{p-1-\epsilon}.
\end{split}
\end{equation*}
Combining this identity with \eqref{comput}-\eqref{yield} yields the conclusion since
$$\left(2-\frac{2(p-\epsilon)}{p-1-\epsilon}\right)=-\left(2-\frac{2(p-2-\epsilon)}{p-1-\epsilon}\right).$$
\end{proof}	

\begin{lem}
The functions
$$u_{\epsilon},~\xi_{1,\epsilon},\cdots,\xi_{n+1,\epsilon}$$
are linearly independent for $\epsilon$ sufficiently small.
\end{lem}
\begin{proof}
We may assume by contradiction that there exist numbers $A_{0,\epsilon}, A_{1,\varepsilon}\cdots,A_{n+1,\varepsilon}$ such that
\begin{equation}\label{A01}
A_{0,\epsilon}u_{\epsilon}+\sum_{k=1}^nA_{k,\epsilon}\xi_{k,\epsilon}+A_{n+1,\epsilon}\xi_{n+1,\epsilon}\equiv0\quad\mbox{in}\hspace{2mm}\Omega,
\end{equation}
and $\sum_{k=1}^{n+1}A^2_{k,\epsilon}\neq0$.
Without any loss of generality, we may assume that for any small $\varepsilon$, there holds
\begin{equation}\label{A2}
\sum_{k=1}^{n+1}A^2_{k,\varepsilon}=1.
\end{equation}
We aim to show that
 \begin{equation}\label{A3}
A_{0,\varepsilon}=0
\end{equation}
for any $\epsilon>0$.
Indeed, assume that \eqref{A3} does not hold true,  then by \eqref{A01} we have
\begin{equation}\label{A04}
u_{\epsilon}(x)=-\sum_{k=1}^{n}\frac{A_{k,\varepsilon}}{A_{0,\epsilon}}\xi_{k,\varepsilon}-\frac{A_{n+1,\epsilon}}{A_{0,\epsilon}}\xi_{n+1,\epsilon}=-\sum_{k=1}^{n+1}\frac{A_{k,\epsilon}}{A_{0,\epsilon}}\xi_{k,\epsilon}.
\end{equation}
Now, noticing that $\tilde{\psi}\equiv1$ in $B_{\varrho}(x_{\epsilon})$, and
$$\frac{\partial}{\partial y_j}\left(\frac{1}{|x-y|^{n-2}}\right)=-\frac{\partial}{\partial x_j}\left(\frac{1}{|x-y|^{n-2}}\right),$$
then we have
\begin{equation}\label{A5}
\begin{split}
-\Delta\xi_{k,\epsilon}&=(p-\epsilon)C_N\left(\int_{\Omega}\frac{u_{\epsilon}^{p-1-\epsilon}(y)\xi_{k,\epsilon}(y)}{|x-y|^{n-2}}dy \right)u_{\epsilon}^{p-1-\epsilon}(x)\\&+(p-1-\epsilon)C_N\left(\int_{\Omega}\frac{u_{\epsilon}^{p-\epsilon}(y)}{|x-y|^{n-2}}dy \right)u_{\epsilon}^{p-1-\epsilon}(x)\xi_{k,\epsilon}(x) \hspace{2mm}\mbox{for}\hspace{2mm}k=1,\cdots,n \hspace{2mm}\mbox{in the ball}\hspace{2mm} B_{\varrho}(x_{\epsilon}).
\end{split}
\end{equation}
On the other hand, due to Lemma \ref{nabla}, we deduce that
\begin{equation}\label{A6}
			\begin{split}
-\Delta\xi_{n+1,\epsilon}&=(p-\epsilon)C_N\Big(\int_{\Omega}\frac{u_{\epsilon}^{p-1-\epsilon}(y)\xi_{n+1,\epsilon}(y)}{|x-y|^{n-2}}dy \Big)u_{\epsilon}^{p-1-\epsilon}(x)\\ &\hspace{3mm}+(p-1-\epsilon)C_N\Big(\int_{\Omega}\frac{u_{\epsilon}^{p-1-\epsilon}}{|x-y|^{n-2}}dy \Big)u_{\epsilon}^{p-2-\epsilon}(x)\xi_{n+1,\epsilon}(x).
\end{split}
		\end{equation}
In conclusion, combining \eqref{A5}-\eqref{A6}, \eqref{ele-1.1} and \eqref{A04}  we obtain
\begin{equation*}
\begin{split}
\Big(\int_{\Omega} \frac{u_{\epsilon}^{p-\epsilon}(y)}{|x-y|^{n-2}}dy\Big)u_{\epsilon}^{p-2-\epsilon}\Big(-\sum_{k=1}^{n+1}\frac{A_{k,\epsilon}}{A_{0,\epsilon}}\xi_{k,\epsilon}\Big)&=(p-\epsilon)\Big(\int_{\Omega}\frac{u^{p-1-\epsilon}}{|x-y|^{n-2}}\Big(-\sum_{k=1}^{n+1}\frac{A_{k,\epsilon}}{A_{0,\epsilon}}\xi_{k,\epsilon}\Big)dy \Big)u_{\epsilon}^{p-1-\epsilon}(x)\\ &+(p-1-\epsilon)\Big(\int_{\Omega}\frac{u_{\epsilon}^{p-\epsilon}}{|x-y|^{n-2}}dy \Big)u_{\epsilon}^{p-2-\epsilon}\Big(-\sum_{k=1}^{n+1}\frac{A_{k,\epsilon}}{A_{0,\epsilon}}\xi_{k,\epsilon}\Big).
\end{split}
\end{equation*}
in the ball $B_{\varrho}(x_{\epsilon})$. This implies
$2(p-1-\epsilon)=0$, that is a contradiction. Hence  we have proved the \eqref{A3}.\\ \noindent
Now we can pass to show that $A_{n+1,\epsilon}=0$. Given $\epsilon>0$, by \eqref{A01}, \eqref{varepsilon}-\eqref{varep} and since  $x_{\varepsilon}$ is a maximum point of $u_{\epsilon}$, we have

\begin{equation}\label{alafa}
\frac{2A_{n+1,\epsilon}}{p-1-\epsilon}u_{\epsilon}(x_{\epsilon})=0,
\end{equation} and so the thesis follows from \eqref{alafa}. As a consequence we have
\begin{equation}\label{afak}
\begin{split}
\sum_{k=1}^nA_{k,\epsilon}\frac{\partial u_{\epsilon}}{\partial x_j}(x)=0
\hspace{2mm}\mbox{in the ball}\hspace{2mm} B_{\varrho}(x_{\epsilon}).
\end{split}
\end{equation}
We note that the authors of \cite{DAIQIN, DY19,GHPS19} independently classified the extremal functions of \eqref{mini-p} as the bubbles $W[\xi,\lambda]$ (as defined in \eqref{defU}). These bubbles are the positive solutions to \eqref{ele-1.1}. Given this, and considering that $\|\tilde{u}_{\epsilon}(x)\|_{\infty}$ is uniformly bounded by some constant $M$, elliptic regularity implies that
$$\|\tilde{u}_{\epsilon}(x)\|_{C^{2+\alpha}(\overline{\Omega})}\leq M\hspace{2mm}\mbox{with}\hspace{2mm}\alpha\in(0,1).$$
Therefore, it is not hard to see that
$\tilde{u}_{\epsilon}(x)\rightarrow W[0,1](x)$ in $C_{loc}^2(\mathbb{R}^n)$
by combining the elliptic interior estimates and where $W(x)$ are the only positive solutions of the equation \eqref{wx}.
It is immediate to see that
\begin{equation*}
\sum_{k=1}^nA_{k}\frac{\partial u_{\epsilon}}{\partial x_j}(x)=0
\hspace{2mm}\mbox{in}\hspace{2mm} \mathbb{R}^n,
\end{equation*}
where $A_{k}$ is the limit of $A_{k,\epsilon}$  with $\sum_{k=1}^{n+1}A^2_{k}=1.$ On the other hand, one can easily verify that $A_{k}=0$ since $\partial_{x_1}W(x),\cdots,\partial_{x_n}W(x)$ are linearly independent.  Whence we get the contradiction with the initial assumption $\sum_{k=1}^{n+1}A^2_{k}=1.$
\end{proof}

	\begin{lem}\label{baowenbei}
		For $i=2,\dots,n+1$, we have
		\begin{equation}\label{baowen1}
\lambda_{i,\epsilon}\leq1+O(\epsilon^{\frac{n}{n-2}}).
		\end{equation}
\end{lem}
	\begin{proof}
We define a linear space $\mathcal{V}$ spanned by $$\{u_{\epsilon}\}\cup\{\xi_{j,\epsilon}: 1\leq j\leq i-1\},$$
so that any nonzero function $f\in\mathcal{V}\setminus\{0\}$ can be write as
\begin{equation}\label{nonzero}
f=a_0u_{\epsilon}+\sum_{j=1}^{i-1}a_j\xi_{j,\epsilon}.
\end{equation}
By the variational
characterization of the eigenvalue $\lambda_{i,\epsilon}$, we have
		\begin{equation}\label{minmax}
\begin{split}
			\lambda_{i,\epsilon}&=\min\limits_{\substack{\mathcal{V}\subset H_{0}^1(\Omega),\\ dim\mathcal{V}=i}}\max\limits_{f\in\mathcal{V}\setminus\{0\}}\frac{\int_{\Omega}|\nabla f(x)|^2dx}{p_{\epsilon}^{(1)}\int_{\Omega}\Big(|x|^{-(n-2)} \ast \big(u_{\epsilon}^{p-1-\epsilon}f\big)\Big)u_{\epsilon}^{p-1-\epsilon}fdx
		+p_{\varepsilon}^{(2)}\int_{\Omega}\Big(|x|^{-(n-2)} \ast u_{\epsilon}^{p-\epsilon}\Big)u_{\epsilon}^{p-2-\epsilon}f^2dx}			\\&\leq\max\limits_{f\in\mathcal{V}\setminus\{0\}}\frac{\int_{\Omega}|\nabla f(x)|^2dx}{p_{\epsilon}^{(1)}\int_{\Omega}\Big(|x|^{-(n-2)} \ast \big(u_{\epsilon}^{p-1-\epsilon}f\big)\Big)u_{\epsilon}^{p-1-\epsilon}fdx
		+p_{\epsilon}^{(2)}\int_{\Omega}\Big(|x|^{-(n-2)} \ast u_{\epsilon}^{p-\epsilon}\Big)u_{\epsilon}^{p-2-\epsilon}f^2dx}\\&
:=\max\limits_{f\in\mathcal{V}\setminus\{0\}}\mathcal{P}_{\epsilon,1}
\end{split}
		\end{equation}
with $p_{\epsilon}^{(1)}=(p-\epsilon),~ p_{\epsilon}^{(2)}=(p-1-\epsilon).$
For the sake of notational simplicity, we write $z_{\epsilon}=\sum_{j=1}^{i-1}a_j\frac{\partial u_{\epsilon}}{\partial x_j}$, so we have $f=a_0u_{\epsilon}+\psi z_{\epsilon}$.
Multiplying \eqref{ele-1.1} by $\psi z_{\epsilon}$ and integrating, we have
		\begin{equation}\label{kuangquanshui}
			\begin{split}
			\int_{\Omega}\nabla(a_0u_{\epsilon})\cdot\nabla(\psi z_{\epsilon})dx=a_0\int_{\Omega}\Big(|x|^{-(n-2)} \ast u_{\epsilon}^{p-\epsilon}\Big)u_{\epsilon}^{p-1-\epsilon}\psi z_{\epsilon}dx.
			\end{split}
		\end{equation}
Moreover, testing
$$-\Delta z_{\epsilon}=(p-\epsilon)\Big(|x|^{-(n-2)} \ast \big(u_{\epsilon}^{p-1-\epsilon}z_{\epsilon}\big)\Big)u_{\epsilon}^{p-1-\epsilon}
		+(p-1-\epsilon)\Big(|x|^{-(n-2)} \ast u_{\epsilon}^{p-\epsilon}\Big)u_{\epsilon}^{p-2-\epsilon}z_{\epsilon}$$
with $\psi^2 z_{\epsilon}$, one has
		\begin{equation}\label{psi-z}
\begin{split}
			\int_{\Omega}|\nabla(\psi z_{\epsilon})|^2dx&=\int_{\Omega}|\nabla\psi|^2z^2_{\epsilon}dx+(p-\epsilon)\int_{\Omega}\Big(|x|^{-(n-2)} \ast \big(u_{\epsilon}^{p-1-\epsilon}z_{\epsilon}\big)\Big)u_{\epsilon}^{p-1-\epsilon}\psi^2 z_{\varepsilon}dx\\&
		\hspace{3mm}+(p-1-\epsilon)\int_{\Omega}\Big(|x|^{-(n-2)} \ast u_{\epsilon}^{p-\epsilon}\Big)u_{\epsilon}^{p-2-\epsilon}\psi^2 z^2_{\epsilon}dx.
		\end{split}
\end{equation}
Then by \eqref{kuangquanshui}-\eqref{psi-z}, we can see $\mathcal{P}_{\epsilon,1}=1+\mathcal{P}_{\epsilon,2}/\mathcal{P}_{\epsilon,3}$ where
\begin{equation*}
\begin{split}			\mathcal{P}_{\epsilon,2}&=-(p_{\epsilon}^{(1)}+p_{\epsilon}^{(2)}-1)a_0^2\int_{\Omega}\Big(|x|^{-(n-2)} \ast u_{\epsilon}^{p-\epsilon}\Big)u_{\epsilon}^{p-\epsilon}dx
-p_{\epsilon}^{(1)}a_0\int_{\Omega}\Big(|x|^{-(n-2)} \ast \big(u_{\epsilon}^{p-1-\epsilon}\psi z_{\epsilon}\big)\Big)u_{\epsilon}^{p-\epsilon}dx\\&\hspace{3mm}
-[p_{\epsilon}^{(1)}+2p_{\epsilon}^{(2)}-1]a_0\int_{\Omega}\Big(|x|^{-(n-2)} \ast u_{\epsilon}^{p-\epsilon}\Big)u_{\epsilon}^{p-1-\epsilon}\psi z_{\epsilon}dx+\int_{\Omega}|\nabla\psi|^2z^2_{\epsilon}dx
\\&\hspace{3mm}+p_{\epsilon}^{(1)}\Big[\int_{\Omega}\Big(|x|^{-(n-2)} \ast \big(u_{\epsilon}^{p-1-\epsilon} z_{\epsilon}\big)\Big)u_{\epsilon}^{p-1-\epsilon}\psi^2 z_{\epsilon}dx-\int_{\Omega}\Big(|x|^{-(n-2)} \ast \big(u_{\epsilon}^{p-1-\epsilon}\psi z_{\epsilon}\big)\Big)u_{\epsilon}^{p-1-\epsilon}\psi z_{\epsilon}dx\Big].
\end{split}
		\end{equation*}
		and
		\begin{equation*}
		\begin{split}	\mathcal{P}_{\epsilon,3}&=(p_{\epsilon}^{(1)}+p_{\epsilon}^{(2)})a_0^2\int_{\Omega}\Big(|x|^{-(n-2)} \ast u_{\epsilon}^{p-\epsilon}\Big)u_{\epsilon}^{p-\epsilon}dx
		+(p_{\epsilon}^{(1)}+2p_{\epsilon}^{(2)})a_0\int_{\Omega}\Big(|x|^{-(n-2)} \ast u_{\epsilon}^{p-\epsilon}\Big)u_{\epsilon}^{p-1-\epsilon}\psi z_{\epsilon}dx\\&\hspace{3mm}+p_{\epsilon}^{(1)}\Big[a_0\int_{\Omega}\Big(|x|^{-(n-2)} \ast \big(u_{\epsilon}^{p-1-\epsilon}\psi z_{\epsilon}\big)\Big)u_{\epsilon}^{p-\epsilon}dx+\int_{\Omega}\Big(|x|^{-(n-2)} \ast \big(u_{\epsilon}^{p-1-\epsilon}\psi z_{\epsilon}\big)\Big)u_{\epsilon}^{p-1-\epsilon}\psi z_{\epsilon}dx\Big]\\&\hspace{3mm}+p_{\epsilon}^{(2)}\int_{\Omega}\Big(|x|^{-(n-2)} \ast u_{\epsilon}^{p-\epsilon}\Big)u_{\epsilon}^{p-2-\epsilon}\psi^2 z_{\epsilon}^2dx.
		\end{split}
\end{equation*}

	Our aim is to find an upper bound of $\mathcal{P}_{\epsilon,2}$ and  $\mathcal{P}_{\epsilon,3}$. 	
		Let's analyze each piece.\\ \noindent
        By an integration by parts and direct calculation we have
		\begin{equation}\label{first  piece}
			\begin{split}				
&\int_{\Omega}\Big(|x|^{-(n-2)} \ast u_{\epsilon}^{p-\epsilon}\Big)u_{\epsilon}^{p-1-\epsilon}\psi z_{\epsilon}dx=\int_{\Omega}\int_{\Omega}\frac{u_{\epsilon}^{p-1-\epsilon}(x)\psi(x) u_{\epsilon}^{p-\epsilon}(y)}{|x-y|^{n-2}}\sum_{j=1}^{i-1}a_j\frac{\partial u_{\epsilon}}{\partial x_j}(x) dxdy\\
=&-\frac{1}{p-\epsilon}\int_{\Omega}\int_{\Omega}\frac{u_{\epsilon}^{p-\epsilon}(x) u_{\epsilon}^{p-\epsilon}(y)}{|x-y|^{n-2}}\sum_{j=1}^{i-1}a_j\frac{\partial \psi}{\partial x_j} dxdy-\frac{1}{p-\epsilon}\int_{\Omega}\int_{\Omega}\sum_{j=1}^{i-1}a_j\frac{\partial}{\partial x_j}(\frac{1}{|x-y|^{n-2}})\psi u_{\epsilon}^{p-\epsilon}(x) u_{\epsilon}^{p-\epsilon}(y) dxdy\\
=&O\Big(\frac{1}{\|u_{\epsilon}\|_{\infty}^{p-\epsilon}}\Big)\frac{\|u_{\epsilon}\|_{\infty}^{\frac{(n-2)(p-1-\epsilon)(p-\epsilon)}{4}}}{\|u_{\epsilon}\|_{\infty}^{\frac{(p-1-\epsilon)(n+2)}{4}}}\Big\|W[0,1]\Big\|^{p-\epsilon}_{L^{\frac{2n}{n+2}(p-\epsilon)}}
\Big[\int_{\Omega\cap\{|x-x_{\epsilon}|\geq\varrho\}}\Big|\sum_{j=1}^{i-1}a_j\frac{\partial \psi}{\partial x_j}(\|u_{\epsilon}\|_{\infty}u_{\epsilon})^{p-\epsilon}\Big|^{\frac{2n}{n+2}}dx\Big]^{\frac{n+2}{2n}}\\&+
\frac{n-2}{p-\epsilon}\int_{\Omega}\int_{\Omega}\sum_{j=1}^{i-1}a_j \frac{(x_j-y_j)\psi u_{\epsilon}^{p-\epsilon}(x) u_{\epsilon}^{p-\epsilon}(y) }{|x-y|^{n}}dxdy.
\end{split}
\end{equation}
For the last term of \eqref{first  piece}, we have the following estimates.
		\begin{equation*}
\begin{split}
			&\int_{\Omega}\int_{\Omega}\sum_{j=1}^{i-1}a_j \frac{(x_j-y_j)\psi(x) u_{\epsilon}^{p-\epsilon}(x) u_{\varepsilon}^{p-\epsilon}(y) }{|x-y|^{n}}dxdy=\Big(\int_{\{|x-x_{\epsilon}|\leq\varrho\}}\int_{\{|x-x_{\epsilon}|\leq\varrho\}}
+\int_{\Omega\setminus\{|x-x_{\epsilon}|\leq\varrho\}}\int_{\{|x-x_{\epsilon}|\leq\varrho\}}\\+&\int_{\{|x-x_{\epsilon}|\leq\varrho\}}\int_{\Omega\setminus\{|x-x_{\epsilon}|\leq\varrho\}}+\int_{\Omega\setminus\{|x-x_{\epsilon}|\leq\varrho\}}\int_{\Omega\setminus\{|x-x_{\epsilon}|\leq\varrho\}}\Big)\sum_{j=1}^{i-1}a_j \frac{(x_j-y_j)\psi(x)u_{\epsilon}^{p-\epsilon}(x) u_{\epsilon}^{p-\epsilon}(y) }{|x-y|^{n}}dxdy\\=&I_1+I_2+I_3+I_4.
\end{split}
		\end{equation*}
We have $I_1=0$ by the symmetry, and
		\begin{equation*}
\begin{split}
I_2\leq&C\big\|u_{\epsilon}\big\|_{\infty}^{\frac{(n-2)(p-1-\epsilon)(p-\epsilon)}{2}}\int_{\Omega\setminus\{|y-x_{\epsilon}|\leq\varrho\}}\int_{\{|x-x_{\epsilon}|\leq\varrho\}}
\Big(\frac{1}{1+\|u_{\epsilon}\|_{\infty}^{p-1-\epsilon}|y-x_{\epsilon}|^2}\Big)^{\frac{(n-2)(p-\epsilon)}{2}}\\&\times\frac{1}{|x-y|^{n-1}}\Big(\frac{1}{1+\|u_{\epsilon}\|_{\infty}^{p-1-\epsilon}|x-x_{\epsilon}|^2}\Big)^{\frac{(n-2)(p-\epsilon)}{2}}dxdy\\
\leq&C\frac{\|u_{\epsilon}\|_{\infty}^{\frac{(n-2)(p-1-\epsilon)(p-\epsilon)}{2}}}{\|u_{\epsilon}\|_{\infty}^{\frac{(p-1-\epsilon)(n+1)}{2}}}
\Big[\int_{\Omega\setminus\{|y|\leq\varrho\|u_{\epsilon}\|_{\infty}^{\frac{p-1-\epsilon}{2}}\}}\frac{1}{(1+|y|^2)^{(n-2)(p-\epsilon)}}dy\Big]^{\frac{1}{2}}\\&\times\Big[\int_{\{|x|\leq\varrho\|u_{\epsilon}\|_{\infty}^{\frac{p-1-\epsilon}{2}}\}}\frac{1}{(1+|x|^2)^{\frac{(n-2)n(p-\epsilon)}{n+2}}}dx\Big]^{\frac{n+2}{2n}}
\leq C\frac{1}{\|u_{\epsilon}\|_{\infty}^{\frac{(p-1-\epsilon)(n+2)}{4}}}=O\Big(\big\|u_{\epsilon}\big\|_{\infty}^{\frac{\epsilon(n+2)}{4}-p}\Big).
\end{split}
		\end{equation*}

by Hardy-Littlewood-Sobolev inequality, $u_{\epsilon}\leq W[x_{\epsilon},\mu_{\epsilon}]$, \eqref{miu}. Similarly, we have
	\begin{equation*}
\begin{split}
I_3=O\Big(\big\|u_{\epsilon}\big\|_{\infty}^{\frac{\epsilon(n+2)}{4}-p}\Big),\quad I_4=O\Big(\big\|u_{\epsilon}\big\|_{\infty}^{\frac{\epsilon(n+2)}{4}-p}\Big).
\end{split}
		\end{equation*}

Finally, combining \eqref{first  piece} with  these inequalities,  Hardy-Littlewood-Sobolev inequality, the condition $u_{\epsilon}\leq W[x_{\epsilon},\mu_{\epsilon}]$ and \eqref{miu}, we have
\begin{equation*}
    \begin{split}
        &\int_{\Omega}\Big(|x|^{-(n-2)} \ast u_{\epsilon}^{p-\epsilon}\Big)u_{\epsilon}^{p-1-\epsilon}\psi z_{\epsilon}dx\leq O\Big(\big\|u_{\epsilon}\big\|_{\infty}^{\epsilon-p}\Big)+O\Big(\big\|u_{\epsilon}\big\|_{\infty}^{\frac{\epsilon(n+2)}{4}-p}\Big)
    \end{split}
\end{equation*}

Thus, similar to the above argument, we obtain	
		\begin{equation}\label{dx11}
			\begin{split}
				\int_{\Omega}\Big(|x|^{-(n-2)} \ast \big(u_{\epsilon}^{p-1-\epsilon}\psi z_{\epsilon}\big)\Big)u_{\epsilon}^{p-\epsilon}dx=O\Big(\big\|u_{\epsilon}\big\|_{L^{\infty}(\Omega)}^{\epsilon-p}\Big)+O\Big(\big\|u_{\epsilon}\big\|_{\infty}^{\frac{\epsilon(n+2)}{4}-p}\Big).
			\end{split}
		\end{equation}
Next by symmetry and integration by parts, we find
		\begin{equation}\label{JO}
		\begin{split}
			J_0:&=\int_{\Omega}\Big(|x|^{-(n-2)} \ast \big(u_{\epsilon}^{p-1-\epsilon} z_{\epsilon}\big)\Big)u_{\epsilon}^{p-1-\epsilon}\psi^2 z_{\epsilon}dx-\int_{\Omega}\Big(|x|^{-(n-2)} \ast \big(u_{\epsilon}^{p-1-\epsilon}\psi z_{\epsilon}\big)\Big)u_{\epsilon}^{p-1-\epsilon}\psi z_{\epsilon}dx\\&
=\frac{1}{2}\int_{\Omega}\int_{\Omega}\frac{(\psi(x)-\psi(y))^2u_{\epsilon}^{p-1-\epsilon}(x) u_{\epsilon}^{p-1-\epsilon}(y)}{|x-y|^{n-2}}\sum_{j=1}^{i-1}a_j\frac{\partial u_{\epsilon}}{\partial x_j}(x) \sum_{k=1}^{i-1}a_k\frac{\partial u_{\epsilon}}{\partial y_k}(y) dxdy\\&
=\frac{1}{2(p-\epsilon)^2}\int_{\Omega}\int_{\Omega}\sum_{j,k=1}^{i-1}a_ja_k\frac{\partial^2}{\partial x_j\partial y_k}\Big(\frac{(\psi(x)-\psi(y))^2}{|x-y|^{n-2}}\Big)u_{\epsilon}^{p-\epsilon}(x) u_{\epsilon}^{p-\epsilon}(y)dxdy.
			\end{split}
		\end{equation}
It is noticing that		
		\begin{equation}\label{4J}
			\begin{split}
				\frac{\partial^2}{\partial x_j\partial y_k}\Big(\frac{(\psi(x)-\psi(y))^2}{|x-y|^{n-2}}\Big)&=\frac{\partial^2}{\partial x_j\partial y_k}\Big(\frac{1}{|x-y|^{n-2}}\Big)(\psi(x)-\psi(y))^2+\frac{\partial}{\partial x_j}\Big(\frac{1}{|x-y|^{n-2}}\Big)\frac{\partial}{\partial y_k}(\psi(x)-\psi(y))^2\\&+\frac{\partial}{\partial y_k}\Big(\frac{1}{|x-y|^{n-2}}\Big)\frac{\partial}{\partial x_j}(\psi(x)-\psi(y))^2+\frac{\partial^2}{\partial x_j\partial y_k}\Big((\psi(x)-\psi(y))^2\Big)\frac{1}{|x-y|^{n-2}}\\&
:=J_1+J_2+J_3+J_4.
			\end{split}
		\end{equation}
Therefore, by Hardy Littlewood Sobolev inequality and definition of $\psi$, similar to estimate of $I_1-I_4$, we deduce that
\begin{equation*}
			\begin{split}
&\int_{\Omega}\int_{\Omega}\sum_{j,k=1}^{i-1}a_ja_ku_{\epsilon}^{p-\epsilon}(x) u_{\epsilon}^{p-\epsilon}(y)dxdy\leq C\Big(\int_{\{|x-x_{\epsilon}|\leq\varrho\}}\int_{\{|x-x_{\epsilon}|\leq\varrho\}}
+2\int_{\Omega\setminus\{|x-x_{\epsilon}|\leq\varrho\}}\int_{\{|x-x_{\epsilon}|\leq\varrho\}}\\+&\int_{\Omega\setminus\{|x-x_{\epsilon}|\leq\varrho\}}\int_{\Omega\setminus\{|x-x_{\epsilon}|\leq\varrho\}}\Big)
\sum_{j,k=1}^{i-1}a_ja_k\frac{W^{p-\epsilon}[\mu_{\epsilon},\lambda_{\epsilon}](x) W^{p-\epsilon}[\mu_{\epsilon},\lambda_{\epsilon}](y)}{|x-y|^n}dxdy = O\Big(\big\|u_{\epsilon}\big\|_{\infty}^{\frac{\epsilon(n+2)}{4}-p}\Big).
\end{split}
		\end{equation*}
Similar computations and gives that
\begin{equation*}
			\begin{split}
\int_{\Omega}\int_{\Omega}\sum_{j,k=1}^{i-1}a_ja_kJ_2u_{\epsilon}^{p-\epsilon}(x) u_{\epsilon}^{p-\epsilon}(y)dxdy=O\Big(\big\|u_{\epsilon}\big\|_{\infty}^{\frac{\epsilon(n+2)}{4}-p}\Big).
\end{split}
		\end{equation*}
and
\begin{equation*}
			\begin{split}
\int_{\Omega}\int_{\Omega}\sum_{j,k=1}^{i-1}a_ja_kJ_3u_{\epsilon}^{p-\epsilon}(x) u_{\epsilon}^{p-\epsilon}(y)dxdy=O\Big(\|u_{\epsilon}\|_{\infty}^{\frac{\epsilon(n+2)}{4}-p}\Big).
\end{split}
		\end{equation*}
Finally, we compute
\begin{equation*}
			\begin{split}
\int_{\Omega}\int_{\Omega}\sum_{j,k=1}^{i-1}a_ja_kJ_4u_{\epsilon}^{p-\epsilon}(x) u_{\epsilon}^{p-\epsilon}(y)dxdy&\leq C\Big[\int_{\Omega}\Big|\sum_{j=1}^{i-1}a_j\frac{\partial\psi(x)}{\partial x_j}W^{p-\epsilon}[\mu_{\epsilon},\lambda_{\epsilon}](x)\Big|^{\frac{2n}{n+2}}dx\Big]^{\frac{n+2}{2n}}
\\&\times \Big[\int_{\Omega}\Big|\sum_{k=1}^{i-1}a_k\frac{\partial\psi(y)}{\partial y_k}W^{p-\epsilon}[\mu_{\varepsilon},\lambda_{\epsilon}](y)\Big|^{\frac{2n}{n+2}}dy\Big]^{\frac{n+2}{2n}}	=O\Big(\|u_{\epsilon}\|_{\infty}^{\frac{\epsilon(n+2)}{4}-p}\Big).
\end{split}
		\end{equation*}
Combining the previous estimates together, we get
\begin{equation}\label{dx}
J_0=O\Big(\big\|u_{\epsilon}\big\|_{\infty}^{\frac{\epsilon(n+2)}{4}-p}\Big).
\end{equation}
		Moreover, we have
		\begin{equation*}
\begin{split}
			\int_{\Omega}\Big(|x|^{-(n-2)} \ast \big(u_{\epsilon}^{p-1-\epsilon}\psi z_{\epsilon}\big)\Big)u_{\epsilon}^{p-1-\epsilon}\psi z_{\epsilon}dx=\int_{\Omega}\int_{\Omega}\frac{u_{\epsilon}^{p-1-\epsilon}(x)\psi(x)  u_{\epsilon}^{p-1-\epsilon}(y)\psi(y) }{|x-y|^{n-2}}\sum_{j=1}^{i-1}a_j\frac{\partial u_{\epsilon}}{\partial x_j}\sum_{k=1}^{i-1}a_k\frac{\partial u_{\epsilon}}{\partial y_k}dxdy.
		\end{split}
\end{equation*}
A direct computation show that
\begin{equation*}
\begin{split}
&\int_{\Omega}\int_{\Omega}\frac{u_{\epsilon}^{p-1-\epsilon}(x)\psi(x)  u_{\epsilon}^{p-1-\epsilon}(y)\psi(y) }{|x-y|^{n-2}}\frac{\partial u_{\epsilon}}{\partial x_j}(x)\frac{\partial u_{\epsilon}}{\partial y_k}(y) dxdy\\
=&\frac{\big\|u_{\epsilon}\big\|_{\infty}^{p+1+\epsilon}}{\big\|u_{\epsilon}\big\|_{\infty}^{\frac{p-1-\epsilon}{2}(n-2)}}\int_{\Omega_{\epsilon}}
\int_{\Omega_{\epsilon}}\frac{\tilde{u}_{\epsilon}^{p-1-\epsilon}(y)\tilde{u}_{\epsilon}^{p-1-\epsilon}(x)}{|x-y|^{n-2}}\psi\Big(\|u_{\epsilon}\|_{\infty}^{-\frac{4-(n-2)\epsilon}{2(n-2)}}x+x_{\epsilon}\Big)
\psi\Big(\big\|u_{\epsilon}\big\|_{\infty}^{-\frac{4-(n-2)\epsilon}{2(n-2)}}y+x_{\epsilon}\Big)\frac{\partial \tilde{u}_{\epsilon}}{\partial x_j}\frac{\partial \tilde{u}_{\epsilon}}{\partial y_k}
dxdy\\
=&
\big\|u_{\epsilon}\big\|_{\infty}^{\frac{4}{n-2}}\Big[\int_{\mathbb{R}^n}
\int_{\mathbb{R}^n}\psi^2(x_0)\frac{W^{p-1}[0,1](y)W^{p-1}[0,1](x)}{|x-y|^{n-2}}\frac{\partial W(x)}{\partial x_j}\frac{\partial W(y)}{y_k}dxdy+o(1)\Big]
\\
=&\big\|u_{\epsilon}\big\|_{L^{\infty}}^{\frac{4}{n-2}}\Big[-\frac{1}{p}\psi^2(x_0)\int_{\mathbb{R}^{n}}\frac{W^{p}[0,1](y)}{|x-y|^{n-2}}\Big(\frac{\partial}{\partial z_k}\big(W^{p}[z,1](y)\big)\Big)\Big|_{z=0}W^{p-1}[0,1](x)\frac{\partial W(x)}{\partial x_j}dxdy+o(1)\Big]\\
=&
\frac{1}{p}\big\|u_{\epsilon}\big\|_{\infty}^{\frac{4}{n-2}}\Big[\psi^2(x_0)\int_{\mathbb{R}^{n}}W^{p-1}[0,1](x)\frac{\partial W(x)}{\partial x_j}\frac{\partial W(x)}{\partial x_k}dxdy+o(1)\Big]\\
=&\frac{1}{p}\big\|u_{\epsilon}\big\|_{\infty}^{\frac{4}{n-2}}\Big[\frac{\delta_j^{k}}{n}\int_{\mathbb{R}^{n}}W^{p-1}[0,1](x)\big|\nabla W(x)\big|^2dx+o(1)\Big]
		\end{split}
\end{equation*}
since $\tilde{u}_{\epsilon}(x)\rightarrow W[0,1](x)$ in $C_{loc}^2(\mathbb{R}^n)$ and \eqref{p1-00} holds.
	Analogously, we deduce that
		\begin{equation*}
			\begin{split}
			\int_{\Omega}\Big(|x|^{-(n-2)} \ast u_{\epsilon}^{p-\epsilon}\Big)u_{\epsilon}^{p-2-\epsilon}\psi^2 z_{\epsilon}^2dx&=\int_{\Omega}\int_{\Omega}\frac{u_{\epsilon}^{p-\epsilon}(y) u_{\epsilon}^{p-2-\epsilon}(x)}{|x-y|^{n-2}}\psi^2(x)\sum_{j=1}^{n}a_j\frac{\partial u_{\epsilon}}{\partial x_j}(x)\sum_{k=1}^{n}a_k\frac{\partial u_{\epsilon}}{\partial x_k}(x) dxdy\\&
=\big\|u_{\epsilon}\big\|_{\infty}^{\frac{4}{n-2}}\Big[\sum_{j=1}^{n}a_j\frac{\delta_j^{k}}{n}\int_{\mathbb{R}^{n}}W^{2^{\ast}-2}[0,1](x)\big|\nabla W(x)\big|^2dx+o(1)\Big].
			\end{split}
		\end{equation*}
To obtain a lower bound of  $\mathcal{P}_{\epsilon,1}$, utilizing the above  estimates and  taking $(a_0,\cdots,a_{i-1})=\frac{1}{i-1}(0,1,\cdots,1)$, we deduce
\begin{equation*}
			\begin{split}
\max\limits_{f\in\mathcal{V}\setminus\{0\}}\mathcal{P}_{\epsilon,1}&\geq1+
\frac{\int_{\Omega}|\nabla\psi|^2z^2_{\epsilon}dx+(p-\epsilon)J_0}{(p-\epsilon)\int_{\Omega}\int_{\Omega}\frac{u_{\epsilon}^{p-1-\epsilon}\psi z_{\epsilon}u_{\epsilon}^{p-1-\epsilon}\psi z_{\epsilon}}{|x-y|^{n-2}} dxdy+(p-1-\epsilon)\int_{\Omega}\int_{\Omega}\frac{u_{\epsilon}^{p-\epsilon} u_{\epsilon}^{p-2-\epsilon}\psi^2 z^2_{\epsilon}}{|x-y|^{n-2}} dxdy}\\&=1+\frac{1}{\big\|u_{\epsilon}\big\|_{\infty}^{\frac{2n}{n-2}}}\bigg[\frac{n\mathcal{K}_n^2}{p}\frac{\sum_{j,k=1}^{i-1}a_ja_k\int_{\Omega}|\nabla\psi|^2\partial_{x_j}G(x,x_0)\partial_{x_k}G(x,x_0)}{\int_{\mathbb{R}^{n}}W^{p-1}[0,1](x)\big|\nabla W(x)\big|^2dx}+o(1)\bigg]
\end{split}
		\end{equation*}
by \eqref{ifini}, combined with the identity
\begin{equation}\label{m000}
\int_{\Omega}|\nabla\psi|^2z^2_{\epsilon}dx=\int_{\Omega\cap\{|x-x_{\epsilon}|\geq\varrho\}}\big|\nabla\psi\big|^2\sum_{j=1}^{i-1}a_j\frac{\partial u_{\epsilon}}{\partial x_j}(x)\sum_{k=1}^{i-1}a_k\frac{\partial u_{\epsilon}}{\partial x_k}(x)dx=O\Big(\big\|u_{\epsilon}\big\|_{\infty}^{-2}\Big).
\end{equation}
Hence we conclude that
\begin{equation}\label{mo0}
\max\limits_{f\in\mathcal{V}\setminus\{0\}}\mathcal{P}_{\epsilon,1}\geq1+\frac{M_0}{\big\|u_{\epsilon}\big\|_{\infty}^{\frac{2n}{n-2}}}
\end{equation}
for some constant $M_0>0$. On the other hand, by the boundedness of $\mathcal{P}_{\epsilon,2}$ and  $\mathcal{P}_{\epsilon,3}$, \eqref{m000}, we have
\begin{equation}\label{mo0-2}
\begin{split}
&\max\limits_{f\in\mathcal{V}\setminus\{0\}}\mathcal{P}_{\epsilon,1}=1\\+&\frac{-(2p-1-\epsilon)a_{0,\epsilon}^2\big\|u_{\epsilon}\big\|_{\infty}^{2}\int_{\Omega}\int_{\Omega}\frac{u_{\epsilon}^{p-\epsilon}u_{\epsilon}^{p-\epsilon}}{|x-y|^{n-2}} dxdy+O\Big(\big\|u_{\epsilon}\big\|_{\infty}^{\frac{\epsilon(n+2)}{4}-p+2}\Big)+O\Big(\big\|u_{\epsilon}\big\|_{\infty}^{\epsilon-p+2}\Big)+O\big(1\big)}{\big\|u_{\epsilon}\big\|_{L^{\infty}}^{\frac{2n}{n-2}}\Big[ O\Big(\big\|u_{\epsilon}\big\|_{\infty}^{-\frac{4}{n-2}}\Big)+O\Big(\big\|u_{\epsilon}\big\|_{\infty}^{\frac{\epsilon(n+2)}{4}-p-\frac{4}{n-2}}\Big)+O\Big(\big\|u_{\epsilon}\big\|_{\infty}^{\epsilon-p-\frac{4}{n-2}}\Big)+A(n)+o(1)\Big]},
\end{split}
\end{equation}
where $A(n)=\frac{p}{n}\int_{\mathbb{R}^{n}}W^{p-1}[0,1](x)|\nabla W(x)|^2dx$.
\\ \noindent In order to conclude an upper bound of \eqref{mo0}, it is sufficient to show
\begin{equation}\label{Big}
a_{0,\epsilon}^2\big\|u_{\epsilon}\big\|_{\infty}^{2}=O\big(1\big).
\end{equation}
Now, assume that $\lim\limits_{\epsilon\rightarrow0}a_{0,\epsilon}^2\big\|u_{\epsilon}\big\|_{\infty}^{2}=+\infty$, by \eqref{mo0}, we have
\begin{equation}\label{Big-0}
\max\limits_{f\in\mathcal{V}\setminus\{0\}}\mathcal{P}_{\epsilon,1}\leq1-\frac{M_0}{\big\|u_{\epsilon}\big\|_{\infty}^{\frac{2n}{n-2}}}
\end{equation}
some constant $M_0$. This gives the expected contradiction to \eqref{mo0}. Thus, \eqref{baowen1} follows from \eqref{Fn}, \eqref{minmax} and \eqref{mo0-2}-\eqref{Big}. This completes the proof of the lemma.
	\end{proof}
\begin{lem}\label{limitlama}
For $i=2,\cdots,n+1$,
we have the following limit
\begin{equation}\label{baowen}
\lim\limits_{\epsilon\rightarrow0}\lambda_{i,\epsilon}=1.
		\end{equation}
	\end{lem}	
\begin{proof}
In view of \eqref{baowen1}, it is enough to prove $\lambda_{2,\epsilon}\rightarrow1$ as $\epsilon\rightarrow0$.
\\ \noindent By \eqref{ele-2}-\eqref{defanndg2}, we have that
\begin{equation}\label{e}
\left\lbrace
\begin{aligned}
    &-\Delta \tilde{v}_{2,\epsilon}=\lambda_{i,\epsilon}C_N \mathcal{G}_{\epsilon}[\tilde{v}_{2,\epsilon}] \quad \mbox{in}\quad \Omega_{\epsilon}:=\big\|u_{\varepsilon}\big\|_{\infty}^{\frac{4-(n-2)\epsilon}{2(n-2)}}(\Omega-x_{\epsilon}),\\
    &\tilde{v}_{2,\epsilon}=0\quad \mbox{on}\hspace{2.5mm}\partial\Omega_{\epsilon},\hspace{2mm}
    \|\tilde{v}_{2,\epsilon}\|_{L^{\infty}(\Omega_\epsilon)}=1,
   \end{aligned}
\right.
\end{equation}
where
	\begin{equation}\label{de}
		\mathcal{G}_{\epsilon}[\tilde{v}_{2,\epsilon}]=(p-\epsilon)\Big(|x|^{-(n-2)} \ast \big(\tilde{u}_{\epsilon}^{p-1-\epsilon}\tilde{v}_{2,\epsilon}\big)\Big)\tilde{u}_{\epsilon}^{p-1-\epsilon}
		+(p-1-\epsilon)\Big(|x|^{-(n-2)} \ast \tilde{u}_{\epsilon}^{p-\epsilon}\Big)\tilde{u}_{\epsilon}^{p-2-\epsilon}\tilde{v}_{2,\epsilon},
	\end{equation}
where $C_N$ can be found in \eqref{wx}. Then by \eqref{e}, \eqref{decay1} and \eqref{day2}, we have
\begin{equation}\label{ccinfinity}
\int_{_{\mathbb{R}^n}}|\nabla \tilde{v}_{2,\epsilon}|^2dx\leq C\int_{_{\mathbb{R}^n}}\frac{W^p(x)W^p(y)}{|x-y|^{n-2}}dxdy<+\infty,\quad\int_{_{\mathbb{R}^n}}|\tilde{v}_{2,\epsilon}|^{2^{\ast}}dx<+\infty.
		\end{equation}
 This implies that $\tilde{v}_{2,\epsilon}$ is uniformly bounded in $\mathcal{D}^{1,2}(\mathbb{R}^n)$. Then, by elliptic regularity and since $\lambda_{2,\epsilon}\rightarrow\Lambda<1$ by \eqref{baowen1}, we get that, up to subsequence,
 $$\lim\limits_{\epsilon\rightarrow0}\tilde{v}_{2,\epsilon}=v_0\quad\mbox{in}\hspace{2mm}C_{loc}^{1}(\mathbb{R}^n),$$
where $v_0$ satisfies
\begin{equation}\label{e-00}
\left\lbrace
\begin{aligned}
    &-\Delta v_{0}=pC_N\Lambda\Big(|x|^{-(n-2)} \ast \big(W^{p-1}[0,1]v_{0}\big)\Big)W^{p-1}\\&
		\quad\quad+(p-1)C_N\Lambda\Big(|x|^{-(n-2)} \ast W^{p}[0,1]\Big)W^{p-2}[0,1]v_{0} \quad \mbox{in}\quad \mathbb{R}^n,\hspace{2mm}|v_{0}|\leq1, \quad v\in\mathcal{D}^{1,2}(\mathbb{R}^n).
   \end{aligned}
\right.
\end{equation}
We now want to prove that $v_0\not\equiv0$. By contradiction, if $v=0$, we chose $y_{\epsilon}$ as a maximum point of $\tilde{v}_{2,\epsilon}$ such that
$\tilde{v}_{2,\epsilon}(y_{\epsilon})=\|\tilde{v}_{2,\epsilon}\|_{\infty}=1$ then necessary $|y_{\epsilon}|\rightarrow\infty$.
This gives the expected contradiction to \eqref{day2}. Moreover, due to Proposition \ref{propep}, we have $\Lambda=(2p-1)^{-1}$ and $v=W[0,1](x)$, and exploiting the orthogonality of $v_{2,\epsilon}$ and passing to the limit as $\epsilon\rightarrow0$ yields
\begin{equation*}
\begin{split}	
(p-\epsilon)\int_{\Omega}&\Big(|x|^{-(n-2)} \ast \big(\tilde{u}_{\epsilon}^{p-1-\epsilon}\tilde{v}_{2,\epsilon}\big)\Big)\tilde{u}_{\epsilon}^{p-\epsilon}dx		+(p-1-\epsilon)\int_{\Omega}\Big(|x|^{-(n-2)} \ast \tilde{u}_{\epsilon}^{p-\epsilon}\Big)\tilde{u}_{\epsilon}^{p-1-\epsilon}\tilde{v}_{2,\epsilon}dx=0	\\&\Rightarrow(2p-1)\int_{\mathbb{R}^n}\int_{\mathbb{R}^n}\frac{W^p[0,1](x)W^p[0,1](y)}{|x-y|^{n-2}}dxdy=0.
\end{split}
\end{equation*}
This estimate gives a contradiction. Hence $\lambda_{2,\epsilon}=1$, which finishes the proof.
\end{proof}

	\section{Asyptotic behaviour of the eigenfunctions $v_{i,\epsilon}$ for $i=2,\cdots,n+1$}\label{section4}
This section is dedicated to the proof of the Theorem \ref{Figalli}. We can start with some preliminary Lemmas.
\begin{lem}\label{identity}
For the rescaled eigenfunctions $\tilde{v}_{i,\epsilon}$ defined in \eqref{vie}, we have
\begin{equation}\label{viep}
\tilde{v}_{i,\epsilon}(x)\rightarrow-(n-2)\sum_{k=1}^{n}\frac{\alpha_k^ix_k}{(1+|x|^2)^{\frac{n}{2}}}+\beta^{i}\frac{1-|x|^2}{(1+|x|^2)^{\frac{n}{2}}}\quad\mbox{in}\hspace{2mm} C_{loc}^{1}(\mathbb{R}^n)
\end{equation}
with $\alpha^{i}=(\alpha_1^{i},\cdots,\alpha_{n}^i,\beta^{i})\neq(0,\cdots,0)$ in $\mathbb{R}^{n+1}$ for $i=2,\cdots,n+1$.
\end{lem}
\begin{proof}
		Let us recall the problem \eqref{e}-\eqref{de} on $\tilde{v}_{i,\epsilon}$ for $i=2,\cdots,n+1$.
Then by \eqref{decay1} and elliptic theory, we get that,
up to a subsequence, there exists a function $v_i$, such that
		 $$\tilde{v}_{i,\epsilon}\rightarrow v_i \quad\mbox{in}\hspace{2mm} C^1_{loc}(\mathbb{R}^n)\hspace{2mm}\mbox{as}\hspace{2mm}\epsilon\rightarrow0,$$ where $v_i$ satisfies
		\begin{equation}\label{Q1}
-\Delta v_{i}=pC_N\Big(|x|^{-(n-2)} \ast \big(W^{p-1}v_{i}\big)\Big)W^{p-1}
		+(p-1)C_N\Big(|x|^{-(n-2)} \ast W^{p}\Big)W^{p-2}v_{i}.			
		\end{equation}
Similar to \eqref{ccinfinity}, then we have $\tilde{v}_{i,\epsilon}$ is uniformly bounded in $\mathcal{D}^{1,2}(\mathbb{R}^n)$ and we know that $v_{i}\in\mathcal{D}^{1,2}(\mathbb{R}^n)$. Then, using Proposition \ref{prondgr}, we get that
\begin{equation*}
v_i(x)=(2-n)\sum_{k=1}^{n}\frac{\alpha_k^ix_k}{(1+|x|^2)^{\frac{n}{2}}}+\beta^{i}\frac{1-|x|^2}{(1+|x|^2)^{\frac{n}{2}}}.
\end{equation*}		
In order to prove \eqref{viep}, it is enough to show
\begin{equation}\label{vector666}
(\alpha_1^i,\alpha_2^i,\cdots,\alpha_n^i,\beta^{i})\neq0.
\end{equation}		
Let $y_{\varepsilon}^{i}$ is a maximum point of $\tilde{v}_{i,\epsilon}$, i.e.
 $$
 \tilde{v}_{i,\epsilon}(y_{\epsilon}^{i})=\big\|\tilde{v}_{i,\epsilon}\big\|_{\infty}=1.
 $$
 By contradiction, if \eqref{vector666} does not hold true, then necessarily $|y_{\epsilon}^{i}|\rightarrow\infty$, which contradicts \eqref{day2}.
	\end{proof}
\begin{lem}\label{B4}
	There is a constant $C>0$, such that
	$$
	\frac{1}{|x|^{n-2}}\ast \left(\frac{1}{1+\big\|u_{\epsilon}\big\|_{\infty}^{\frac{4-(n-2)\epsilon}{2(n-2)}}|x-x_{\epsilon}|}\right)^{(n-2)(p-\epsilon)}\leq C\frac{\big\|u_{\epsilon}\big\|_{\infty}^{-\frac{4-(n-2)\epsilon}{n-2}}}{\big(1+\big\|u_{\epsilon}\big\|_{\infty}^{\frac{4-(n-2)\epsilon}{2(n-2)}}|x-x_{\epsilon}|\big)^{n-2}},
	$$
where $\epsilon>0$.
\end{lem}
\begin{proof}
 The estimate can be achieved in the same way as in \cite{WY}.
\end{proof}
\begin{lem}\label{m00}
For any $i\in\{2,\dots,n+1\}$ then there exists $\beta^{i}\neq0$ such that
\begin{equation}\label{number}
\big\|u_{\epsilon}\big\|_{\infty}^2v_{i,\epsilon}(x)\rightarrow \mathcal{M}_0\beta^{i}G(x,x_0)\quad\mbox{in}\hspace{2mm}C^{1,\alpha}(\omega)
\end{equation}
where the convergence is in $C^{1,\tilde{\alpha}}(\omega)$ with $\omega$ any compact set of $\bar{\Omega}$ not containing $x_0$ the limit point of $x_\epsilon$ and $\mathcal{M}_0=-\big(\frac{4}{n(n+2)}+\frac{1}{n}\big)\Gamma_nC_N\omega_n$ where
$$\Gamma_n=I(s)S^{\frac{(2-n)}{8}}C_{n,n-2}^{\frac{2-n}{8}}[n(n-2)]^{\frac{n-2}{4}}
\hspace{2mm}\mbox{with}\hspace{2mm}I(s)=\frac{\pi^{\frac{n}{2}}\Gamma(\frac{n-2s}{2})}{\Gamma(n-s)}, \ \ \mbox{and }\Gamma(s)=\int_0^{+\infty} x^{s-1}e^{-x}\,dx,~s>0.$$
\end{lem}
\begin{proof}
Let us fix $i\in\{2,\cdots,n+1\}$, we have
\begin{equation}\label{mn}
-\Delta\big(\big\|u_{\epsilon}\big\|_{\infty}^2v_{i,\epsilon}\big)=\lambda_{i,\epsilon}C_N\big\|u_{\epsilon}\big\|_{\infty}^2 \mathcal{G}[v_{i,\epsilon}]\quad\mbox{in}\hspace{2mm}\Omega.
\end{equation}
We integrate the right-hand side of \eqref{mn}
\begin{equation*}
\begin{split}
\lambda_{i,\epsilon}C_N\big\|u_{\epsilon}\big\|_{\infty}^2 \int_{\Omega}\mathcal{G}[v_{i,\epsilon}]dx&=
\lambda_{i,\epsilon}C_N(p-\epsilon)\big\|u_{\epsilon}\big\|_{\infty}^2\int_{\Omega}\int_{\Omega}\frac{u_{\epsilon}^{p-1-\epsilon}(y)v_{i,\epsilon}(y)u_{\epsilon}^{p-1-\epsilon}(x)}{|x-y|^{n-2}}dxdy\\&
\hspace{2mm}+\lambda_{i,\epsilon}C_N(p-1-\epsilon)\big\|u_{\epsilon}\big\|_{\infty}^2\int_{\Omega}\int_{\Omega}\frac{u_{\epsilon}^{p-\epsilon}(y)u_{\epsilon}^{p-2-\epsilon}(x)v_{i,\epsilon}(x)}{|x-y|^{n-2}}dxdy.
\end{split}
\end{equation*}
 Using \eqref{bepus}, \eqref{decay1} and \eqref{day2} by dominated convergence, we deduce that
\begin{equation*}
\begin{split}
\big\|u_{\epsilon}\big\|_{\infty}^2\int_{\Omega}\int_{\Omega}\frac{u_{\epsilon}^{p-1-\epsilon}(y)v_{i,\epsilon}(y)u_{\epsilon}^{p-1-\epsilon}(x)}{|x-y|^{n-2}}dxdy
&= \frac{\big\|u_{\epsilon}\big\|_{\infty}^{2(p-\epsilon)}}{\big\|u_{\epsilon}\big\|_{\infty}^{\frac{p-1-\epsilon}{2}(n+2)}}\int_{\Omega_{\epsilon}}\int_{\Omega_{\epsilon}}\frac{\tilde{u}_{\epsilon}^{p-1-\epsilon}(y)\tilde{v}_{i,\epsilon}(y)\tilde{u}_{\epsilon}^{p-1-\epsilon}(x)}{|x-y|^{n-2}}dxdy\\&\leq M\int_{\mathbb{R}^n}\int_{\mathbb{R}^n}\frac{W^{p}(y)W^{p-1}(x)}{|x-y|^{n-2}}dxdy<\infty
\end{split}
\end{equation*}
as $\epsilon\rightarrow0$ and for some $M>0$. Similarly, we have
\begin{equation*}
\begin{split}
\big\|u_{\epsilon}\big\|_{\infty}^2\int_{\Omega}\int_{\Omega}\frac{u_{\epsilon}^{p-\epsilon}(y)u_{\epsilon}^{p-2-\epsilon}(x)v_{i,\epsilon}(x)}{|x-y|^{n-2}}dxdy\leq M\int_{\mathbb{R}^n}\int_{\mathbb{R}^n}\frac{W^{p}(y)W^{p-1}(x)}{|x-y|^{n-2}}dxdy<\infty
\end{split}
\end{equation*}
as $\epsilon\rightarrow0$ and for some $M>0$.
For $\psi\in C_{0}^{\infty}(\Omega)$, multiplying \eqref{mn} by $\psi$ and integrating we get
\begin{equation*}
\begin{split}
&\lambda_{i,\epsilon}C_N(p-\epsilon)\big\|u_{\epsilon}\big\|_{\infty}^2\int_{\Omega}\int_{\Omega}\frac{u_{\epsilon}^{p-1-\epsilon}(y)v_{i,\epsilon}(y)u_{\epsilon}^{p-1-\epsilon}(x)\psi(x)}{|x-y|^{n-2}}dxdy\\&
=\lambda_{i,\epsilon}C_N(p-\epsilon)\frac{\big\|u_{\epsilon}\big\|_{\infty}^{2(p-\epsilon)}}{\big\|u_{\epsilon}\big\|_{\infty}^{\frac{p-1-\epsilon}{2}(n-2)}}\int_{\Omega_{\epsilon}}\int_{\Omega_{\epsilon}}\frac{\tilde{u}_{\epsilon}^{p-1-\epsilon}(y)\tilde{v}_{i,\epsilon}(y)\tilde{u}_{\epsilon}^{p-1-\epsilon}(x)}{|x-y|^{n-2}}\psi\Big(\big\|u_{\epsilon}\big\|_{\infty}^{-\frac{4-(n-2)\epsilon}{2(n-2)}}x+x_{\epsilon}\Big)
dxdy\\&
=C_Np\int_{\mathbb{R}^n}\int_{\mathbb{R}^n}\frac{W^{p-1}(y)W^{p-1}(x)}{|x-y|^{n-2}}\Big((2-n)\sum_{k=1}^{n}\frac{\alpha_k^iy_k}{(1+|y|^2)^{\frac{n}{2}}}+\beta^{i}\frac{1-|y|^2}{(1+|y|^2)^{\frac{n}{2}}}
\Big)\psi\big(x_{0}\big)dxdy+o(1),
\end{split}
\end{equation*}
where we exploited \eqref{bepus}, \eqref{decay1}, \eqref{day2}, \eqref{baowen} and \eqref{viep} in the last equality. On account of the oddness of the integrand, we have that
\begin{equation}\label{W-0}
\int_{\mathbb{R}^n}\int_{\mathbb{R}^n}\frac{W^{p-1}(y)W^{p-1}(x)}{|x-y|^{n-2}}\frac{y_k}{(1+|y|^2)^{\frac{n}{2}}}dxdy=
\int_{\mathbb{R}^n}\int_{\mathbb{R}^n}\frac{y_k}{(1+|y|^2)^{\frac{n}{2}+2}}\frac{1}{|x-y|^{n-2}}\frac{1}{(1+|x|^2)^{2}}dxdy=0
\end{equation}
for any $k=1,\cdots,n$. By \eqref{p1-00},
we have
\begin{equation}\label{thatwwe}
\begin{split}
\int_{\mathbb{R}^n}\frac{W^{p-1}(y)}{|x-y|^{n-2}}\frac{1-|y|^2}{(1+|y|^2)^{\frac{n}{2}}}dy&=\frac{\partial}{\partial\lambda}\Big(\int_{\mathbb{R}^n}\frac{W^{p}[0,\mu](y)}{|x-y|^{n-2}}dy\Big)\Big|_{\mu=1}
=I(\frac{n-2}{2})S^{\frac{(2-n)}{8}}C_{n,n-2}^{\frac{2-n}{8}}[n(n-2)]^{\frac{n-2}{4}}\Big(\frac{\partial W}{\partial\mu}\Big|_{\mu=1}\Big)\\&=I(\frac{n-2}{2})S^{\frac{(2-n)}{8}}C_{n,n-2}^{\frac{2-n}{8}}[n(n-2)]^{\frac{n-2}{4}}
\frac{1-|x|^2}{(1+|x|^2)^{\frac{n}{2}}}.
\end{split}
\end{equation}
Considering the above equality, we calculate
\begin{equation*}
\begin{split}
&\lambda_{i,\epsilon}C_N(p-\epsilon)\big\|u_{\epsilon}\big\|_{\infty}^2\int_{\Omega}\int_{\Omega}\frac{u_{\epsilon}^{p-1-\epsilon}(y)v_{i,\epsilon}(y)u_{\epsilon}^{p-1-\epsilon}(x)\psi(x)}{|x-y|^{n-2}}dxdy\\&
=\Gamma_nC_Np\beta^{i}\psi\big(x_{0}\big)\int_{\mathbb{R}^n}W^{p-1}(x)\frac{1-|x|^2}{(1+|x|^2)^{\frac{n}{2}}}
dx+o(1)=-\frac{1}{n}\Gamma_nC_N\omega_n\beta^{i}\psi\big(x_{0}\big)+o(1),
\end{split}
\end{equation*}
where $\omega_n$ is the measure of the surface  the unit ball in $\mathbb{R}^n$. In the last inequality, we have applied the  computation
\begin{equation}\label{wn}
\int_{\mathbb{R}^n}W^{p-1}(x)\frac{1-|x|^2}{(1+|x|^2)^{\frac{n}{2}}}
dx=\frac{\omega_n}{2}\Big(\frac{\Gamma(\frac{n}{2})\Gamma(2)}{\Gamma(\frac{n}{2}+2)}-\frac{\Gamma(\frac{n}{2}+1)\Gamma(1)}{\Gamma(\frac{n}{2}+2)}\Big)=-\frac{n-2}{n(n+2)}\omega_n.
\end{equation}
Similarly, we find that
\begin{equation*}
\begin{split}
&\lambda_{i,\epsilon}C_N(p-1-\epsilon)\big\|u_{\epsilon}\big\|_{\infty}^2\int_{\Omega}\int_{\Omega}\frac{u_{\epsilon}^{p-\epsilon}(y)u_{\epsilon}^{p-2-\epsilon}(x)v_{i,\epsilon}(x)}{|x-y|^{n-2}}\psi(x)dxdy\\&
=C_N(p-1)\int_{\mathbb{R}^n}\int_{\mathbb{R}^n}\frac{W^{p}(y)W^{p-2}(x)}{|x-y|^{n-2}}\Big((2-n)\sum_{k=1}^{n}\frac{\alpha_k^ix_k}{(1+|x|^2)^{\frac{n}{2}}}+\beta^{i}\frac{1-|x|^2}{(1+|x|^2)^{\frac{n}{2}}}
\Big)\psi\big(x_{0}\big)dxdy+o(1).
\end{split}
\end{equation*}
Therefore, from \eqref{p1-00} we conclude that
\begin{equation*}
\begin{split}
&\lambda_{i,\epsilon}C_N(p-1-\epsilon)\big\|u_{\epsilon}\big\|_{\infty}^2\int_{\Omega}\int_{\Omega}\frac{u_{\epsilon}^{p-\epsilon}(y)u_{\epsilon}^{p-2-\epsilon}(x)v_{i,\epsilon}(x)}{|x-y|^{n-2}}\psi(x)dxdy\\&
=\Gamma_nC_N(p-1)\beta^{i}\psi\big(x_{0}\big)\int_{\mathbb{R}^n}W^{2^{\ast}-2}(x)\frac{1-|x|^2}{(1+|x|^2)^{\frac{n}{2}}}
dx+o(1)=-\frac{4}{n(n+2)}\Gamma_nC_N\omega_n\beta^{i}\psi\big(x_{0}\big)+o(1).
\end{split}
\end{equation*}
As a result,
\begin{equation*}
\lambda_{i,\epsilon}C_N\big\|u_{\epsilon}\big\|_{\infty}^2 \int_{\Omega}\mathcal{G}[v_{i,\epsilon}]\psi(x)dx=-\big(\frac{4}{n(n+2)}+\frac{1}{n}\big)\Gamma_nC_N\omega_n\beta^{i}\psi\big(x_{0}\big)+o(1).
\end{equation*}
On the other hand,
\begin{equation*}
\begin{split}
\lambda_{i,\epsilon}C_N\big\|u_{\epsilon}\big\|_{\infty}^2 \mathcal{G}[v_{i,\epsilon}]&=\lambda_{i,\epsilon}C_N
\big\|u_{\epsilon}\big\|_{\infty}^2\Big(|x|^{-(n-2)} \ast \big(u_{\epsilon}^{p-1-\epsilon}v_{i,\epsilon}\big)\Big)u_{\epsilon}^{p-1-\epsilon}
\\&\hspace{3mm}+\lambda_{i,\epsilon}C_N\big\|u_{\epsilon}\big\|_{\infty}^2\Big(|x|^{-(n-2)} \ast u_{\epsilon}^{p-\epsilon}\Big)u_{\epsilon}^{p-2-\epsilon}v_{i,\epsilon}=:\mathcal{I}_{\epsilon,1}+\mathcal{I}_{\epsilon,2}.
\end{split}
\end{equation*}
Hence, using Lemma \ref{B4}, \eqref{dus} and \eqref{decay22} in Theorem A with $x_{\epsilon}\rightarrow x_0\in\Omega$ as $\epsilon\rightarrow0$, we find
\begin{equation*}
\begin{split}
\mathcal{I}_{\epsilon,1}&\leq M_0\big\|u_{\epsilon}\big\|_{\infty}^{2(p-\epsilon)}\Big(|x|^{-(n-2)} \ast W^{p-\epsilon}\big(\big\|u_{\epsilon}\big\|_{\infty}^{\frac{4-(n-2)\epsilon}{2(n-2)}}(x-x_{\epsilon})\big)\Big)W^{p-1-\epsilon}\Big(\big\|u_{\epsilon}\big\|_{\infty}^{\frac{4-(n-2)\epsilon}{2(n-2)}}(x-x_{\epsilon})\Big)
\\&\leq M_0\frac{\big\|u_{\epsilon}\big\|_{\infty}^{2(p-\epsilon)}\cdot\big\|u_{\epsilon}\big\|_{\infty}^{-\frac{4-(n-2)\epsilon}{n-2}}}{\big\|u_{\epsilon}\big\|_{\infty}^{\frac{4-(n-2)\epsilon}{2}(p-\epsilon)}|x-x_{0}|^{(n-2)(p-\epsilon)}}\leq M_0\frac{1}{\big\|u_{\epsilon}\big\|_{\infty}^{\frac{4}{n-2}}},
\end{split}
\end{equation*}
for $x\neq x_0$ and some positive constant $M_{0}$ not depending on $\epsilon$. Then
$$\lambda_{i,\epsilon}C_N
\big\|u_{\epsilon}\big\|_{\infty}^2\Big(|x|^{-(n-2)} \ast \big(u_{\epsilon}^{p-1-\epsilon}v_{i,\epsilon}\big)\Big)u_{\epsilon}^{p-1-\varepsilon}
\rightarrow0\quad \mbox{for}\hspace{2mm}x\neq x_0.$$
Also we have
\begin{equation*}
\begin{split}
\mathcal{I}_{\epsilon,2}\leq M_0\frac{\big\|u_{\epsilon}\big\|_{\infty}^{2(p-\epsilon)}\cdot\big\|u_{\epsilon}\big\|_{\infty}^{-\frac{4-(n-2)\epsilon}{n-2}}}{\big\|u_{\epsilon}\big\|_{\infty}^{\frac{4-(n-2)\epsilon}{2}(p-\epsilon)}|x-x_{\epsilon}|^{(n-2)(p-\epsilon)}}\leq M_0\frac{1}{\big\|u_{\epsilon}\big\|_{\infty}^{\frac{4}{n-2}}}\hspace{2mm} \hspace{2mm}\mbox{in}\hspace{2mm} \bar{\Omega}\setminus\{x_0\}
\end{split}
\end{equation*}
for some constant $M_{0}>0$ not depending on $\epsilon$.
Thus
$$\lambda_{i,\epsilon}C_N\big\|u_{\epsilon}\big\|_{\infty}^2\Big(|x|^{-(n-2)} \ast u_{\epsilon}^{p-\epsilon}\Big)u_{\epsilon}^{p-2-\epsilon}v_{i,\epsilon}\rightarrow0\quad\mbox{for}\hspace{2mm}x\neq x_0.$$
From the above estimates we have
$$
-\Delta\big(\big\|u_{\epsilon}\big\|_{\infty}^2v_{i,\epsilon}\big)\rightarrow-\big(\frac{4}{n(n+2)}+\frac{1}{n}\big)\Gamma_nC_N\omega_n\beta^{i}\delta_{x=x_0}
$$
in the sense of distribution in $\Omega$. \\ \noindent Let $\omega$ be any compact set in $\bar{\Omega}$ not containing $x_0$. By  Lemma \ref{regular}, we obtain
$$
\big\|\big\|u_{\epsilon}\big\|_{\infty}^2v_{i,\epsilon}\big\|_{C^{1,\tilde{\alpha}}(\omega)}\leq
C\Big(\big\|\lambda_{i,\epsilon}C_N\big\|u_{\epsilon}\big\|_{\infty}^2 \mathcal{G}[v_{i,\epsilon}]\big\|_{L^{1}(\Omega)}+\big\|\lambda_{i,\epsilon}C_N\big\|u_{\epsilon}\big\|_{\infty}^2 \mathcal{G}[v_{i,\epsilon}]\big\|_{L^{\infty}(\omega)}\Big).
$$
Thus,
$$\big\|u_{\epsilon}\big\|_{\infty}^2v_{i,\epsilon}\rightarrow-\big(\frac{4}{n(n+2)}+\frac{1}{n}\big)\Gamma_nC_N\omega_n\beta^{i}G(x,x_0)\quad\mbox{in}\hspace{2mm}C^{1,\tilde{\alpha}}(\omega)$$
as $\epsilon\rightarrow0$ and the conclusion follows.
 \end{proof}	
As a consequence of Lemma \ref{nabla}, we find the following Corollary.
\begin{cor}\label{conseq}
It holds that
\begin{equation}\label{conse1}
\begin{split}
&\int_{\partial\Omega}((x-z),\nu)\frac{\partial u_{\epsilon}}{\partial\nu}\frac{\partial v_{i,\epsilon}}{\partial\nu}dS_x\\&=
(p-\epsilon)C_N(1-\lambda_{i,\epsilon})\int_{\Omega}\int_{\Omega}\frac{u_{\epsilon}^{p-1-\epsilon}v_{i,\epsilon}}{|x-y|^{n-2}}u_{\epsilon}^{p-1-\epsilon}\big((y-z)\cdot\nabla u_{\epsilon}+\frac{2}{p-1-\epsilon}u_{\epsilon}\big)dxdy\\ &+(p-1-\epsilon)C_N(1-\lambda_{i,\epsilon})\int_{\Omega}\int_{\Omega}\frac{u_{\epsilon}^{p-\epsilon}(y)}{|x-y|^{n-2}} u_{\epsilon}^{p-2-\epsilon}(x)v_{i,\epsilon}(x)\big((x-z)\cdot\nabla u_{\epsilon}+\frac{2}{p-1-\epsilon}u_{\epsilon}\big)dxdy
\end{split}
\end{equation}
for any $z\in\mathbb{R}^n$ and for $i=2,\cdots,n+2$,
where $\nu=\nu(x)$ denotes the unit outward normal to the boundary $\partial \Omega$.
\end{cor}
\begin{proof}
Recalling the identity \eqref{dend}, and multiplying this identity by $v_{i,\epsilon}$ and integrating by parts, we obtain
\begin{equation}\label{de11}
			\begin{split}
\int_{\Omega}&\nabla\Big((x-z)\cdot\nabla u_{\epsilon}+\frac{2}{p-1-\epsilon}u_{\epsilon}\Big)\cdot\nabla v_{i,\epsilon}dx\\&=(p-\epsilon)C_N\int_{\Omega}\Big(\int_{\Omega}\frac{u^{p-1-\epsilon}\big((y-z)\cdot\nabla u_{\epsilon}+\frac{2}{p-1-\epsilon}u_{\epsilon}\big)}{|x-y|^{n-2}}dy \Big)u_{\epsilon}^{p-1-\epsilon}v_{i,\epsilon}dx\\ &\hspace{3mm}+(p-1-\epsilon)C_N\int_{\Omega}\Big(\int_{\Omega}\frac{u^{p-\epsilon}}{|x-y|^{n-2}}dy \Big)u_{\epsilon}^{p-2-\epsilon}v_{i,\epsilon}\big((y-z)\cdot\nabla u_{\epsilon}+\frac{2}{p-1-\epsilon}u_{\epsilon}\big)dx.
\end{split}
		\end{equation}
Similarly, by \eqref{ele-2}-\eqref{defanndg2}, we have
\begin{equation}\label{d00}
			\begin{split}
\int_{\Omega}&\nabla\Big((x-z)\cdot\nabla u_{\epsilon}+\frac{2}{p-1-\epsilon}u_{\epsilon}\Big)\cdot\nabla v_{i,\epsilon}dx-\int_{\partial\Omega}\nabla\Big((x-z)\cdot\nabla u_{\epsilon}\Big)\frac{\partial v_{i,\epsilon}}{\partial\nu}dS_x \\&=(p-\epsilon)C_N\lambda_{i,\epsilon}\int_{\Omega}\Big(\int_{\Omega}\frac{u_{\epsilon}^{p-1-\epsilon}v_{i,\epsilon}}{|x-y|^{n-2}}dy \Big)u_{\epsilon}^{p-1-\epsilon}\big((x-z)\cdot\nabla u_{\epsilon}+\frac{2}{p-1-\epsilon}u_{\epsilon}\big)dx\\ &\hspace{3mm}+(p-1-\epsilon)C_N\lambda_{i,\epsilon}\int_{\Omega}\Big(\int_{\Omega}\frac{u_{\epsilon}^{p-\epsilon}}{|x-y|^{n-2}}dy \Big)u_{\varepsilon}^{p-2-\epsilon}v_{i,\epsilon}\big((x-z)\cdot\nabla u_{\epsilon}+\frac{2}{p-1-\epsilon}u_{\epsilon}\big)dx.
\end{split}
		\end{equation}
Hence, combining \eqref{de11}-\eqref{d00} and the  symmetry of the integrals give \eqref{conse1}.
\end{proof}

  We are now ready to prove Theorem \ref{Figalli}.
\begin{proof}[\textbf{Proof of Theorem \ref{Figalli} (for \eqref{vie-1})}]
In order to pove \eqref{vie-1}, it is  sufficient to get that $\beta^{i}=0$ in \eqref{identity} for any $i=\{2,\cdots,n+1\}$. To show this,
we use \eqref{ifini} and \eqref{number} to write
\begin{equation}\label{write }	\frac{\partial}{\partial\nu}\big(\big\|u_{\epsilon}\big\|_{\infty}u_{\epsilon}\big)\frac{\partial}{\partial\nu}\big(\big\|u_{\epsilon}\big\|_{\infty}^2v_{i,\epsilon}\big) \rightarrow\mathcal{K}_n \mathcal{M}_0\beta^{i}\big(\frac{\partial G}{\partial\nu}(x,x_{0})\big)^2
\end{equation}
uniformly on $\partial\Omega$. Furthermore, since we have the following identity (see \cite{BP})
\begin{equation}\label{GX0}
\int_{\partial\Omega}\frac{\partial G(x,x_0)}{\partial \nu}\frac{\partial G(x,x_0)}{\partial \nu}(\nu,x-x_0)dS_{x}=-(n-2)H(x_0,x_0).
\end{equation}
Therefore, choosing $z=x_{\epsilon}$ in the left hand side of \eqref{conse1}, we can write
\begin{equation}\label{can}
\begin{split}
\int_{\partial\Omega}((x-x_{\epsilon}),\nu)\frac{\partial u_{\epsilon}}{\partial\nu}\frac{\partial v_{i,\epsilon}}{\partial\nu}dS_{x}&=
\frac{\mathcal{K}_n \mathcal{M}_0\beta^{i}}{\big\|u_{\epsilon}\big\|^{3}_{\infty}}\Big(\int_{\Omega\partial}((x-x_{0}),\nu) \big(\frac{\partial G}{\partial\nu}(x,x_{0})\big)^2+o(1)\Big)\\&
=-\frac{(n-2)\mathcal{K}_n \mathcal{M}_0\beta^{i}}{\big\|u_{\epsilon}\big\|^{3}_{\infty}}\big(H(x_0,x_0)+o(1)\big)
\end{split}
\end{equation}
as $\epsilon\rightarrow0$. On the other hand, taking $z=x_{\epsilon}$ in the right hand side of \eqref{conse1},
\begin{equation}\label{hand}
\begin{split}
&\int_{\Omega}\int_{\Omega}\frac{u_{\epsilon}^{p-1-\epsilon}(y)\big((y-x_{\epsilon})\cdot\nabla u_{\epsilon}+\frac{2}{p-1-\epsilon}u_{\epsilon}(y)\big)}{|x-y|^{n-2}}u_{\epsilon}^{p-1-\epsilon}(x)v_{i,\epsilon}(x)dxdy\\&
=\frac{\big\|u_{\epsilon}\big\|_{\infty}^{\frac{4-(n-2)\epsilon}{2}}\cdot\big\|u_{\epsilon}\big\|_{\infty}^{2(p-1-\epsilon)+1}}{\big\|u_{\epsilon}\big\|_{\infty}^{\frac{4-(n-2)\epsilon}{n-2}n}}
\int_{\Omega_{\epsilon}}\int_{\Omega_{\epsilon}}\frac{\tilde{u}_{\epsilon}^{p-1-\epsilon}(y)\big[y_{\epsilon}\cdot\nabla \tilde{u}_{\epsilon}+\frac{2}{p-1-\epsilon}\tilde{u}_{\epsilon}(y)\big]\tilde{u}_{\epsilon}^{p-1-\epsilon}(x)\tilde{v}_{i,\epsilon}(x)}{|x-y|^{n-2}}
dxdy\\&
=\frac{1}{\big\|u_{\epsilon}\big\|_{\infty}^{\frac{2-(n-2)\epsilon}{2}}}\bigg[\int_{\mathbb{R}^n}\int_{\mathbb{R}^n}\frac{W^{p-1}(y)W^{p-1}(x)}{|x-y|^{n-2}}\frac{1-|y|^2}{(1+|y|^2)^{\frac{n}{2}}}\Big((2-n)\sum_{k=1}^{n}\frac{\alpha_k^ix_k}{(1+|x|^2)^{\frac{n}{2}}}+\beta^{i}\frac{1-|x|^2}{(1+|x|^2)^{\frac{n}{2}}}
\Big)dxdy\\&\hspace{3mm}+o(1)\bigg]
=\frac{1}{\big\|u_{\epsilon}\big\|_{\infty}^{\frac{2-(n-2)\epsilon}{2}}}\bigg[\beta^{i}\int_{\mathbb{R}^n}\int_{\mathbb{R}^n}\frac{W^{p-1}(y)W^{p-1}(x)}{|x-y|^{n-2}}\frac{1-|y|^2}{(1+|y|^2)^{\frac{n}{2}}}\frac{1-|x|^2}{(1+|x|^2)^{\frac{n}{2}}}
dxdy+o(1)\bigg]
\end{split}
\end{equation}
as $\epsilon\rightarrow0$, where in the last integral we used the oddness of the integrand. Combining \eqref{p1-00} and \eqref{thatwwe} yields
\begin{equation*}
\begin{split}
\int_{\mathbb{R}^n}\int_{\mathbb{R}^n}\frac{W^{p-1}(y)W^{p-1}(x)}{|x-y|^{n-2}}\frac{1-|y|^2}{(1+|y|^2)^{\frac{n}{2}}}\frac{1-|x|^2}{(1+|x|^2)^{\frac{n}{2}}}
dxdy=\Gamma_n\int_{\mathbb{R}^n}W^{p-1}(x)\Big(\frac{1-|x|^2}{(1+|x|^2)^{\frac{n}{2}}}\Big)^2dx.
\end{split}
\end{equation*}
As a consequence, we obtain that
\begin{equation}\label{hand23}
\begin{split}
&\int_{\Omega}\int_{\Omega}\frac{u_{\epsilon}^{p-1-\epsilon}(y)\big((y-x_{\epsilon})\cdot\nabla u_{\epsilon}+\frac{2}{p-1-\epsilon}u_{\epsilon}(y)\big)}{|x-y|^{n-2}}u_{\epsilon}^{p-1-\epsilon}(x)v_{i,\epsilon}(x)dxdy\\&
=\frac{1}{\big\|u_{\epsilon}\big\|_{\infty}^{\frac{2-(n-2)\epsilon}{2}}}\bigg[\beta^{i}
\Gamma_n\int_{\mathbb{R}^n}W^{p-1}(x)\Big(\frac{1-|x|^2}{(1+|x|^2)^{\frac{n}{2}}}\Big)^2dx+o(1)\bigg]
\end{split}
\end{equation}
as $\epsilon\rightarrow0$. Similarly, we have
\begin{equation}\label{h23}
\begin{split}
&\int_{\Omega}\int_{\Omega}\frac{u_{\epsilon}^{p-\epsilon}(y)}{|x-y|^{n-2}} u_{\epsilon}^{p-2-\epsilon}(x)v_{i,\epsilon}(x)\big((x-x_{\epsilon})\cdot\nabla u_{\epsilon}+\frac{2}{p-1-\epsilon}u_{\epsilon}\big)dxdy\\&
=\frac{1}{\big\|u_{\epsilon}\big\|_{\infty}^{\frac{2-(n-2)\epsilon}{2}}}
\int_{\Omega_{\epsilon}}\int_{\Omega_{\epsilon}}\frac{\tilde{u}_{\epsilon}^{p-\epsilon}(z)\tilde{u}_{\epsilon}^{p-2-\epsilon}(y)\tilde{v}_{i,\epsilon}(y)}{|y-z|^{n-2}}
\big[y_{\epsilon}\cdot\nabla \tilde{u}_{\epsilon}+\frac{2}{p-1-\epsilon}\tilde{u}_{\epsilon}(y)\big]dydz\\&
=\frac{1}{\big\|u_{\epsilon}\big\|_{\infty}^{\frac{2-(n-2)\epsilon}{2}}}\bigg[\beta^{i}
\Gamma_n\int_{\mathbb{R}^n}W^{2^{\ast}-2}(y)\Big(\frac{1-|y|^2}{(1+|y|^2)^{\frac{n}{2}}}\Big)^2dy+o(1)\bigg]
\end{split}
\end{equation}
as $\epsilon\rightarrow0$. Combining \eqref{conse1}, \eqref{can}, \eqref{hand23} and \eqref{h23}, we deduce that
\begin{equation*}
\begin{split}
-\frac{(n-2)\mathcal{K}_n \mathcal{M}_0\beta^{i}}{\big\|u_{\epsilon}\big\|^{3}_{\infty}}\big(H(x_0,x_0)+o(1)\big)=
\frac{(2p-1)C_N(1-\lambda_{i,\epsilon})}{\big\|u_{\epsilon}\big\|_{\infty}^{\frac{2-(n-2)\epsilon}{2}}}\bigg[\beta^{i}
\Gamma_n\int_{\mathbb{R}^n}W^{p-1}(x)\Big(\frac{1-|x|^2}{(1+|x|^2)^{\frac{n}{2}}}\Big)^2dx+o(1)\bigg]
\end{split}
\end{equation*}
as $\epsilon\rightarrow0$. Therefore if $\beta^{i}\neq0$ we have
\begin{equation}\label{refor}
1-\lambda_{i,\varepsilon}=\frac{1}{\big\|u_{\epsilon}\big\|_{\infty}^{\frac{4+(n-2)\epsilon}{2}}}\bigg[
-\frac{(n-2)\mathcal{K}_n \mathcal{M}_0}{(2p-1)\Gamma_nC_N}\Big(\int_{\mathbb{R}^n}W^{p-1}(x)\big(\frac{1-|x|^2}{(1+|x|^2)^{\frac{n}{2}}}\big)^2dx\Big)^{-1}H(x_0,x_0)+o(1)\bigg]
\end{equation}
as $\varepsilon\rightarrow0$. In view of \eqref{ifini}, Lemma \ref{m00} and the fact that $H(x,x)<0$ by the maximum principle, simple computations give
\begin{equation}\label{xiaoyu}
-\frac{(n-2)\mathcal{K}_n \mathcal{M}_0}{(2p-1)\Gamma_nC_N}\Big(\int_{\mathbb{R}^n}W^{p-1}(x)\big(\frac{1-|x|^2}{(1+|x|^2)^{\frac{n}{2}}}\big)^2dx\Big)^{-1}H(x_0,x_0)<0
\end{equation}
by
$$\Big(\int_{\mathbb{R}^n}W^{p-1}(x)\big(\frac{1-|x|^2}{(1+|x|^2)^{\frac{n}{2}}}\big)^2dx\Big)^{-1}=\frac{2^{n+1}}{\pi^{\frac{n}{2}}}>0.$$
It follows from \eqref{bepus} and \eqref{refor} that
\begin{equation}\label{yu}
1-\lambda_{i,\epsilon}\leq-\frac{1}{2\big\|u_{\epsilon}\big\|_{\infty}^{2}}
\frac{(n-2)\mathcal{K}_n \mathcal{M}_0}{(2p-1)\Gamma_nC_N}\Big(\int_{\mathbb{R}^n}W^{p-1}(x)\big(\frac{1-|x|^2}{(1+|x|^2)^{\frac{n}{2}}}\big)^2dx\Big)^{-1}H(x_0,x_0)
\end{equation}
for $\epsilon$ sufficiently small. On the other hand, owing to  \eqref{Fn} and \eqref{baowen},
$$
1-\lambda_{i,\epsilon}\geq\frac{\widehat{M}_0}{\big\|u_{\epsilon}\big\|_{\infty}^{\frac{2n}{n-2}}}\quad \mbox{for}\hspace{2mm}\epsilon\hspace{2mm}\mbox{sufficiently small},
$$
and some constant $\widehat{M}_0>0$. Combining the above estimates and \eqref{yu}, we get
$$
\frac{2\widehat{M}_0}{\big\|u_{\epsilon}\big\|_{\infty}^{\frac{4}{n-2}}}\leq1-\lambda_{i,\epsilon}\leq-\frac{(n-2)\mathcal{K}_n \mathcal{M}_0}{(2p-1)\Gamma_nC_N}\Big(\int_{\mathbb{R}^n}W^{p-1}(x)\big(\frac{1-|x|^2}{(1+|x|^2)^{\frac{n}{2}}}\big)^2dx\Big)^{-1}H(x_0,x_0).
$$
Hence we clearly get a contradiction with \eqref{xiaoyu}, which concludes the proof of \eqref{vie-1}.
\end{proof}

Before starting the proof of \eqref{vie-2} in Theorem \ref{Figalli}, we would like to note that Hartree type nonlinearity in problem \eqref{defanndg2} is nonlocal. In particular, we have to borrow some new ideas to overcome the difficulties caused by nonlocal cross terms, not exploiting the methods of constructing orthogonal matrix by Gross-Pacella in \cite{GP05}.
Then following the work of Takahashi \cite{T} for unique solvability result, we can introduce the following initial value problem of the linear first order PDE.
\begin{lem}\label{Ve}
Let some vectors $\alpha^{i}=(\alpha_1^{i},\cdots,\alpha_{n}^i)\neq0$ in $\mathbb{R}^n$ and $\Gamma_{\alpha}=\{x\in\mathbb{R}^n:\alpha\cdot z=0\}$. Then there exists a unique solution $\tilde{V}_{\epsilon}$ of the following initial value problem
\begin{equation}\label{SOLV}
\left\lbrace
\begin{aligned}
    &\sum_{k=1}^{n}\alpha_{k}^{i}\frac{\partial\tilde{V}_{\epsilon}(w)}{\partial w_k}=\int_{\Omega_\epsilon}\frac{\tilde{u}_{\epsilon}^{p-1-\epsilon}(z)\tilde{v}_{i,\epsilon}(z)\tilde{u}_{\epsilon}^{p-1-\epsilon}(w)}{|w-z|^{n-2}}dz  \quad \mbox{in}\quad \Omega_{\epsilon}:=\|u_{\epsilon}\|_{L^{\infty}(\Omega)}^{\frac{4-(n-2)\epsilon}{2(n-2)}}(\Omega-x_{\epsilon}),\\
    &\tilde{V}_{\epsilon}(w)\big|_{\Gamma_{\alpha}}=\frac{\Gamma_n}{p^2}W^{p}[0,1](w).
   \end{aligned}
\right.
\end{equation}
Moreover, we have
\begin{equation}\label{esti}
\tilde{V}_{\epsilon}(w)\leq C\frac{1}{|w|^{n+1}}\quad\mbox{as}\hspace{2mm}|w|\rightarrow\infty.
\end{equation}
\end{lem}
\begin{proof}
The unique solvability result follows from Lemma 2.4 in \cite{T} in same way. Estimate \eqref{esti} can be achieved by \eqref{decay1}, \eqref{day2} and \eqref{p1-00}.
\end{proof}
\begin{lem}
For all $x\neq y\in\Omega$, there exists a positive constant $c$ such that
\begin{equation}\label{Gx}
0<G(x,y)<\frac{c_n}{|x-y|^{n-2}}\quad\mbox{and}\quad|\nabla_{x}G(x,y)|\leq\frac{c}{|x-y|^{n-1}}.
\end{equation}
\end{lem}
\begin{proof}
For the derivation of \eqref{Gx}, see the proof of Lemma A.1 of \cite{Ackerman}.
\end{proof}
We now resume the proof of \eqref{vie-2} in Theorem \ref{Figalli}.
\begin{proof}[\textbf{Proof of Theorem \ref{Figalli} (for \eqref{vie-2})}]
By the Green's representation, we have
\begin{equation}\label{115if}
\begin{split}
v_{i,\epsilon}(x)&=\lambda_{i,\epsilon}C_N \int_{\Omega}G(x,w)\mathcal{G}[v_{i,\epsilon}]dw
=\lambda_{i,\epsilon}C_N(p-\epsilon)\int_{\Omega}\int_{\Omega}\frac{u_{\epsilon}^{p-1-\epsilon}(z)v_{i,\epsilon}(z)u_{\epsilon}^{p-1-\epsilon}(w)}{|w-z|^{n-2}} G(x,w)dwdz	\\&+\lambda_{i,\epsilon}C_N(p-1-\epsilon)\int_{\Omega}\int_{\Omega}\frac{u_{\epsilon}^{p-\epsilon}(z)u_{\epsilon}^{p-2-\epsilon}(w)v_{i,\epsilon}(w)}{|w-z|^{n-2}} G(x,w)dwdz=:P_1+P_2.
\end{split}
\end{equation}
Applying \eqref{bepus} and \eqref{baowen}, we deduce that
\begin{equation}\label{and}
\begin{split}
P_1&=\frac{\lambda_{i,\epsilon}C_N(p-\epsilon)\big\|u_{\epsilon}\big\|_{\infty}^{\frac{4-(n-2)\epsilon}{2}}\cdot\big\|u_{\epsilon}\big\|_{\infty}^{2(p-1-\epsilon)}}{\big\|u_{\epsilon}\big\|_{\infty}^{\frac{4-(n-2)\epsilon}{n-2}n}}
\int_{\Omega_\epsilon}\int_{\Omega_\epsilon}\frac{\tilde{u}_{\epsilon}^{p-1-\epsilon}(z)\tilde{v}_{i,\epsilon}(z)\tilde{u}_{\epsilon}^{p-1-\epsilon}(w)}{|w-z|^{n-2}} \\&\hspace{3mm}\times G\left(x,\frac{w}{\big\|u_{\epsilon}\big\|_{\infty}^{\frac{4-(n-2)\epsilon}{2(n-2)}}}+x_{\epsilon}\right)dwdz\\&	=\frac{C_Np+o(1)}{\big\|u_{\epsilon}\big\|_{\infty}^{\frac{4-(n-2)\epsilon}{2}}}
\int_{\Omega_\epsilon}\int_{\Omega_\epsilon}\frac{\tilde{u}_{\epsilon}^{p-1-\epsilon}(z)\tilde{v}_{i,\epsilon}(z)\tilde{u}_{\epsilon}^{p-1-\epsilon}(w)}{|w-z|^{n-2}} G\left(x,\frac{w}{\big\|u_{\epsilon}\big\|_{\infty}^{\frac{4-(n-2)\epsilon}{2(n-2)}}}+x_{\epsilon}\right)dwdz. \end{split}
\end{equation}
Analogously, we can infer that
\begin{equation}\label{and11}
P_2=\frac{C_N(p-1)+o(1)}{\big\|u_{\epsilon}\big\|_{\infty}^{\frac{4-(n-2)\epsilon}{2}}}
\int_{\Omega_\epsilon}\int_{\Omega_\epsilon}\frac{\tilde{u}_{\epsilon}^{p-\epsilon}(z)\tilde{u}_{\epsilon}^{p-2-\epsilon}(w)\tilde{v}_{i,\epsilon}(w)}{|w-z|^{n-2}} G\left(x,\frac{w}{\big\|u_{\epsilon}\big\|_{\infty}^{\frac{4-(n-2)\epsilon}{2(n-2)}}}+x_{\epsilon}\right)dwdz.
\end{equation}
Since
\begin{equation}\label{viepu}
\tilde{v}_{i,\epsilon}(z)\rightarrow-(n-2)\sum_{k=1}^{n}\frac{\alpha_k^ix_k}{(1+|x|^2)^{\frac{n}{2}}}
=\sum_{k=1}^{n}\alpha_k^i\frac{\partial}{\partial z_k}\big(W[0,1](z)\big),
\end{equation}
and whence we infer from \eqref{decay1}, \eqref{day2} and \eqref{p1-00} that
\begin{equation}\label{whence}
\begin{split}
\int_{\Omega_\epsilon}\frac{\tilde{u}_{\epsilon}^{p-1-\epsilon}(z)\tilde{v}_{i,\epsilon}(z)\tilde{u}_{\varepsilon}^{p-1-\varepsilon}(w)}{|w-z|^{n-2}}dz
&\rightarrow \sum_{k=1}^{n}\alpha_k^i\int_{\mathbb{R}^{n}}\frac{W^{p-1}[0,1](z)\frac{\partial}{\partial z_k}\big(W[0,1](z)\big)W^{p-1}[0,1](w)}{|w-z|^{n-2}}dz\\&
=-\frac{1}{p}\sum_{k=1}^{n}\alpha_k^i\int_{\mathbb{R}^{n}}\frac{W^{p-1}[0,1](w)}{|w-z|^{n-2}}\Big(\frac{\partial}{\partial y_k}\big(W^{p}[y,1](z)\big)\Big)\Big|_{y=0}dz\\&
=-\frac{1}{p}\sum_{k=1}^{n}\alpha_k^i\Big(\frac{\partial}{\partial y_k}\Big(\int_{\mathbb{R}^{n}}\frac{W^{p}[y,1](z)}{|w-z|^{n-2}}dz\Big)\Big|_{y=0}\Big)W^{p-1}[0,1](w)
\\&=\frac{\Gamma_n}{p^2}\sum_{k=1}^{n}\alpha_k^i\Big(\frac{\partial}{\partial w_{k}}W^p[0,1](w)\Big).
\end{split}
\end{equation}
As a consequence, from Lemma \ref{Ve}, \eqref{and}, \eqref{whence} and an integration by parts yields
\begin{equation}\label{whence11}
\begin{split}
P_1&=\frac{C_Np+o(1)}{\big\|u_{\epsilon}\big\|_{\infty}^{\frac{4-(n-2)\epsilon}{2}}}
\int_{\Omega_\epsilon}\int_{\Omega_\epsilon}\frac{\tilde{u}_{\epsilon}^{p-1-\epsilon}(z)\tilde{v}_{i,\epsilon}(z)\tilde{u}_{\epsilon}^{p-1-\epsilon}(w)}{|w-z|^{n-2}} G\left(x,\frac{w}{\big\|u_{\epsilon}\big\|_{\infty}^{\frac{4-(n-2)\epsilon}{2(n-2)}}}+x_{\epsilon}\right)dwdz\\& =\frac{C_Np+o(1)}{\big\|u_{\epsilon}\big\|_{\infty}^{\frac{4-(n-2)\epsilon}{2}}}
\int_{\Omega_\epsilon}\sum_{k=1}^{n}\alpha_{k}^{i}\frac{\partial\tilde{V}_{\epsilon}(w)}{\partial w_k}G\left(x,\frac{w}{\big\|u_{\epsilon}\big\|_{\infty}^{\frac{4-(n-2)\epsilon}{2(n-2)}}}+x_{\epsilon}\right)dw
\\&
=-\frac{C_Np+o(1)}{\big\|u_{\epsilon}\big\|_{\infty}^{\frac{(n-1)[4-(n-2)\epsilon]}{2(n-2)}}}
\int_{\Omega_\epsilon}\tilde{V}_{\epsilon}(w)\left(\sum_{k=1}^{n}\alpha_{k}^{i}\frac{\partial}{\partial w_k}G\left(x,\frac{w}{\big\|u_{\epsilon}\big\|_{\infty}^{\frac{4-(n-2)\epsilon}{2(n-2)}}}+x_{\epsilon}\right)\right)dw.
\end{split}
\end{equation}
Passing to the limit as $\epsilon\rightarrow0$ in \eqref{whence11} and combining \eqref{Gx} with
$$
\int_{\Omega_\epsilon}\tilde{V}_{\epsilon}(w)dw\rightarrow\frac{\Gamma_n}{p^2}\int_{\mathbb{R}^n}W^{p}[0,1](w)dw,
$$
we would get
\begin{equation}\label{whence11}
P_1=-\frac{C_N\Gamma_n}{p\big\|u_{\epsilon}\big\|_{\infty}^{\frac{(n-1)[4-(n-2)\epsilon]}{2(n-2)}}}
\left(\sum_{k=1}^{n}\alpha_{k}^{i}\frac{\partial}{\partial w_k}G\left(x,x_{0}\right)\right)\int_{\mathbb{R}^n}W^{p}[0,1](w)dw
+o(1).
\end{equation}
For $P_2$, similarly to the argument of $P_1$, we have
\begin{equation}\label{argu}
P_2=\frac{C_N\Gamma_n(p-1)}{(2^{\ast}-1)\big\|u_{\epsilon}\big\|_{\infty}^{\frac{(n-1)[4-(n-2)\epsilon]}{2(n-2)}}}
\left(\sum_{k=1}^{n}\alpha_{k}^{i}\frac{\partial}{\partial w_k}G\left(x,x_{0}\right)\right)\int_{\mathbb{R}^n}W^{2^{\ast}-1}[0,1](w)dw
+o(1)
\end{equation}
by \eqref{decay1}, \eqref{day2}, \eqref{p1-00}, \eqref{and11}, \eqref{viepu}. Recalling \eqref{Fn} and \eqref{bepus}, and combining them with \eqref{whence11}, \eqref{argu} and regularity theory, we have
$$
\frac{v_{i,\epsilon}}{\epsilon^{\frac{n-1}{n-2}}}=A_0\left(\sum_{k=1}^{n}\alpha_{k}^{i}\frac{\partial}{\partial w_k}G\left(x,x_{0}\right)\right)
+o(1)\quad\mbox{in}\hspace{2mm}C_{loc}^1\big(\overline{\Omega}\setminus\{x_0\}\big)
$$
as $\epsilon$ sufficiently small  with
$$
A_0=\frac{C_N\Gamma_n(p-2)}{pF_n^{\frac{n-1}{n-2}}}\int_{\mathbb{R}^n}W^{p}[0,1](w)dw,
$$
where $F_n$ is defined in \eqref{Fn},
which concludes the proof.
\end{proof}

	\section{Asymptotic behaviour of the eigenvalues $\lambda_{i,\epsilon}$ for $i=2,\cdots,n+1$}	\label{sangshen}
The objective of this section is to give the proof of Theorem \ref{remainder terms}. The following Lemmas will be used to obtain the asymptotic behaviour of the eigenvalues $\lambda_{i,\epsilon}$ for $i=2,\cdots,n+1$.	
	\begin{lem}\label{qiegao}
		Let $\{u_\epsilon\}_{\epsilon>0}$ be a least energy solutions of  \eqref{ele-1.1}-\eqref{minimi}. Suppose that $v_{i,\varepsilon}$ is a solution \eqref{ele-2}-\eqref{defanndg2}. Then
		\begin{equation}\label{qiegao-1}
\begin{split}
			\int_{\partial\Omega}\frac{\partial u_{\epsilon}}{\partial x_j}\frac{\partial v_{i,\epsilon}}{\partial\nu}dS_x&=
(p-\epsilon)C_N(1-\lambda_{i,\epsilon})\int_{\Omega}\int_{\Omega}\frac{u_{\epsilon}^{p-1-\epsilon}(y)\frac{\partial u_{\epsilon}}{\partial y_j}(y)u_{\epsilon}^{p-1-\epsilon}(x)v_{i,\epsilon}(x)}{|x-y|^{n-2}}dxdy\\ &+(p-1-\epsilon)C_N(1-\lambda_{i,\epsilon})\int_{\Omega}\int_{\Omega}\frac{u_{\epsilon}^{p-\epsilon}(y)u_{\epsilon}^{p-2-\epsilon}(x)v_{i,\epsilon}(x)\frac{\partial u_{\epsilon}(x)}{\partial x_j}}{|x-y|^{n-2}} dxdy
		\end{split}
\end{equation}
		for any $j=1,\cdots,n$.
	\end{lem}
		\begin{proof}
The conclusion can be derived by simple computations similar to \eqref{de11}-\eqref{d00} and we omit it.	
	\end{proof}
	
	\begin{lem}\label{wanfan}
		Let $v_{l,\epsilon}$ and  $v_{m,\epsilon}$ are two orthogonal eigenfunctions of \eqref{ele-2}-\eqref{defanndg2} in $H_{0}^{1}(\Omega)$ corresponding to eigenvalues $\lambda_{l,\epsilon}$ and $\lambda_{m,\epsilon}$ , then there holds	
$$
(\alpha^{l},\alpha^{m})=0,
		$$
where the corresponding vector $\alpha^{l}$ and $\alpha^{m}$ defined in Theorem \ref{Figalli}.
	\end{lem}
	\begin{proof}
		By \eqref{ele-2}-\eqref{defanndg2} and orthogonality, we obtain
		\begin{equation}\label{pingjiehuoguo}
		\begin{split}	(p-\epsilon)\int_{\Omega}\int_{\Omega}&\frac{u_{\epsilon}^{p-1-\epsilon}(y)v_{l,\epsilon}(y)u_{\epsilon}^{p-1-\epsilon}(x)v_{m,\epsilon}(x)}{|x-y|^{n-2}}dxdy \\&+(p-1-\epsilon)\int_{\Omega}\int_{\Omega}\frac{u_{\varepsilon}^{p-\epsilon}(y)u_{\epsilon}^{p-2-\epsilon}(x)v_{l,\epsilon}(x)v_{m,\epsilon}(x)}{|x-y|^{n-2}} dxdy=0.
\end{split}
		\end{equation}
Coupling \eqref{vie-1}, \eqref{decay1}, \eqref{day2} and \eqref{pingjiehuoguo} yields
		\begin{equation*}
\begin{split}		
p\sum_{k,h=1}^{n}\int_{\mathbb{R}^n}\int_{\mathbb{R}^n}&\frac{W^{p-1}(y)W^{p-1}(x)}{|x-y|^{n-2}}
\Big(\frac{\alpha_k^ly_k}{(1+|y|^2)^{\frac{n}{2}}}\Big)\Big(\frac{\alpha_h^mx_h}{(1+|x|^2)^{\frac{n}{2}}}\Big)dxdy\\&+	(p-1)\sum_{k,h=1}^{n}\int_{\mathbb{R}^n}\int_{\mathbb{R}^n}\frac{W^{p}(y)W^{p-2}(x)}{|x-y|^{n-2}}
\Big(\frac{\alpha_k^lx_k\alpha_h^mx_h}{(1+|x|^2)^{n}}\Big)dxdy=0.	\end{split}
\end{equation*}
Furthermore, we have that
$$
\int_{\mathbb{R}^n}\frac{W^{p-1}(y)}{|x-y|^{n-2}}
\frac{y_k}{(1+|y|^2)^{\frac{n}{2}}}dy=\int_{\mathbb{R}^n}\frac{y_k}{(1+|y|^2)^{\frac{n}{2}+2}}\frac{1}{|x-y|^{n-2}}dy=0,
$$
which implies
$$\int_{\mathbb{R}^n}\int_{\mathbb{R}^n}\frac{W^{p-1}(y)W^{p-1}(x)}{|x-y|^{n-2}}
\Big(\frac{y_k}{(1+|y|^2)^{\frac{n}{2}}}\Big)\Big(\frac{x_h}{(1+|x|^2)^{\frac{n}{2}}}\Big)dxdy=0.$$
As a result,
$$\sum_{k,h=1}^{n}\int_{\mathbb{R}^n}\int_{\mathbb{R}^n}\frac{W^{p}(y)W^{p-2}(x)}{|x-y|^{n-2}}
\Big(\frac{\alpha_k^lx_k\alpha_h^mx_h}{(1+|x|^2)^{n}}\Big)dxdy=0,$$
so that we eventually get
$$		\sum_{k=1}^{n}\frac{\alpha_k^l\alpha_k^m}{n}\int_{\mathbb{R}^n}\int_{\mathbb{R}^n}\frac{W^{p}(y)W^{p-2}(x)}{|x-y|^{n-2}}
\Big(\frac{|x|^2}{(1+|x|^2)^{n}}\Big)dxdy=0.$$
The desired identity is obtained.
	\end{proof}

We next give the following proof.
\begin{proof}[\textbf{Proof of Theorem \ref{remainder terms}}]
For each $i=2,\cdots,n+1$, we define $L_i^{\epsilon}$ and $R_{i}^{\epsilon}$ as the left-hand and right-hand sides of identity \eqref{qiegao-1}. We will then proceed to estimate their values.

First we estimate $L_i^{\epsilon}$.
It follows directly from \eqref{vie-2} and \eqref{Fn}-\eqref{ifini} that
		\begin{equation}\label{caI}
	\begin{split}		L_i^{\epsilon}&=\big\|u_{\epsilon}\big\|_{\infty}^{-\frac{3n-4}{n-2}}\int_{\partial\Omega}\frac{\partial }{\partial x_j}\big(\big\|u_{\epsilon}\big\|_{\infty}u_{i,\epsilon}\big)\frac{\partial }{\partial\nu}\big(\big\|u_{\epsilon}\big\|_{\infty}^{\frac{2(n-1)}{n-2}}v_{i,\epsilon}\big)dS_x\\&
=F_n^{\frac{n-1}{n-2}}\mathcal{K}_nA_0\big\|u_{\epsilon}\big\|_{\infty}^{-\frac{3n-4}{n-2}}\Big[A_0\sum_{k=1}^{n}\alpha_k^i\int_{\partial\Omega}\frac{\partial}{\partial\nu}\Big(\frac{\partial G}{\partial w_k}(x,x_0)\Big)\frac{\partial G}{\partial x_j}(x,x_0)+o(1)\Big],
		\end{split}
\end{equation}
		where $A_0$, $F_n$ and $\mathcal{K}_n$ are defined in Theorem \ref{Figalli} and Theorem A. Furthermore, we observe that similar to the computation of \eqref{GX0}, we have
 \begin{equation}\label{GX0-1}
\int_{\partial\Omega}\frac{\partial G(x,w)}{\partial x_j}\frac{\partial G(x,w)}{\partial w_k}\Big(\frac{\partial G}{\partial\nu}(x,w)\Big)dS_{x}=\frac{1}{2}\frac{\partial^2}{\partial x_j\partial x_k}H(w,w)\quad\mbox{for all}\hspace{2mm}j, k=1,\cdots,n.
\end{equation}
Combining \eqref{caI} and \eqref{GX0-1} yields
		\begin{equation*}		L_i^{\epsilon}=\|u_{\epsilon}\|_{L^{\infty}(\Omega)}^{-\frac{3n-4}{n-2}}\Big[\frac{F_n^{\frac{n-1}{n-2}}\mathcal{K}_nA_0}{2}\sum_{k=1}^{n}\alpha_k^i\frac{\partial^2}{\partial x_j\partial x_k}H(x_0,x_0)+o(1)\Big].
		\end{equation*}
Next we estimate $R_{i}^{\epsilon}$. By \eqref{vie} and \eqref{miu-1}, we have
		\begin{equation}\label{AIC}
\begin{split}
\int_{\Omega}\int_{\Omega}&\frac{u_{\epsilon}^{p-1-\epsilon}(y)\frac{\partial u_{\epsilon}}{\partial y_j}(y)u_{\epsilon}^{p-1-\epsilon}(x)v_{i,\epsilon}(x)}{|x-y|^{n-2}}dxdy\\&=
\frac{\big\|u_{\epsilon}\big\|_{\infty}^{\frac{4-(n-2)\epsilon}{2}}\cdot\big\|u_{\epsilon}\big\|_{\infty}^{2(p-1-\epsilon)}}{\big\|u_{\epsilon}\big\|_{\infty}^{\frac{4-(n-2)\epsilon}{n-2}n}}
\big\|u_{\epsilon}\big\|_{\infty}^{\frac{4-(n-2)\epsilon}{2(n-2)}+1}\int_{\Omega_\epsilon}\int_{\Omega_\epsilon}\frac{\tilde{u}_{\epsilon}^{p-1-\epsilon}(y)\frac{\partial\tilde{u}_{\epsilon}}{\partial y_j}(y)\tilde{v}_{i,\epsilon}(x)\tilde{u}_{\epsilon}^{p-1-\epsilon}(x)}{|x-y|^{n-2}}dxdy.
\end{split}
\end{equation}
Moreover, by Hardy-Littlewood-Sobolev inequality, combining \eqref{vie-1}, \eqref{decay1} and \eqref{day2} and dominated convergence, we deduce that
\begin{equation}\label{A1-1}
\begin{split}
\int_{\Omega_\epsilon}\int_{\Omega_\epsilon}&\frac{\tilde{u}_{\epsilon}^{p-1-\epsilon}(y)\frac{\partial\tilde{u}_{\epsilon}}{\partial y_j}(y)\tilde{v}_{i,\epsilon}(x)\tilde{u}_{\epsilon}^{p-1-\epsilon}(x)}{|x-y|^{n-2}}dxdy
\\&=-(n-2)\sum_{k=1}^{n}\alpha_k^i\int_{\mathbb{R}^n}\int_{\mathbb{R}^n}\frac{W^{p-1}(y)W^{p-1}(x)}{|x-y|^{n-2}}\frac{\partial W(y)}{\partial y_j}\frac{x_k}{(1+|x|^2)^{\frac{n}{2}}}dxdy+o(1)\\&
=(n-2)\frac{\Gamma_n}{p}\sum_{k=1}^{n}\alpha_k^i\int_{\mathbb{R}^n}W^{p-1}\frac{\partial W(x)}{\partial x_k}\frac{\partial W(x)}{\partial x_j}dx+o(1)\\&
=(n-2)\frac{\Gamma_n}{pn}\alpha_j^i\int_{\mathbb{R}^n}W^{p-1}|\nabla W|^2dx+o(1),
\end{split}
\end{equation}
where we used similarly to argument of \eqref{whence} and the fact that $\int_{\mathbb{R}^n}\nabla(\frac{\partial W}{\partial x_k})\nabla(\frac{\partial W}{\partial x_j})dx=0$ if $k\neq j$. Therefore, by \eqref{AIC} and \eqref{A1-1} we have
\begin{equation*}
\begin{split} \int_{\Omega}\int_{\Omega}\frac{u_{\epsilon}^{p-1-\epsilon}(y)\frac{\partial u_{\epsilon}}{\partial y_j}(y)u_{\epsilon}^{p-1-\epsilon}(x)v_{i,\epsilon}(x)}{|x-y|^{n-2}}dxdy
=\frac{(n-2)}{\big\|u_{\epsilon}\big\|_{\infty}^{\frac{n-4}{(n-2)}}}\frac{\Gamma_n}{pn}\Big[\alpha_j^i\int_{\mathbb{R}^n}W^{p-1}|\nabla W|^2dx+o(1)\Big]
\end{split}
\end{equation*}
by \eqref{bepus}. Analogously, we also have
 $$
\int_{\Omega}\int_{\Omega}\frac{u_{\epsilon}^{p-\epsilon}(y)u_{\varepsilon}^{p-2-\epsilon}(x)v_{i,\epsilon}(x)\frac{\partial u_{\epsilon}(x)}{\partial x_j}}{|x-y|^{n-2}} dxdy=-\frac{(n-2)}{\big\|u_{\epsilon}\big\|_{\infty}^{\frac{n-4}{n-2}}}\frac{\Gamma_n}{n}\Big[\alpha_j^i\int_{\mathbb{R}^n}W^{2^{\ast}-2}|\nabla W|^2dx+o(1)\Big]
$$
Consequently, we get that for each $j$,
$$R_i^{\epsilon}=-\frac{(n-2)}{\big\|u_{\epsilon}\big\|_{\infty}^{\frac{n-4}{n-2}}}\frac{\Gamma_nC_N}{n}(p-2)(1-\lambda_{i,\epsilon})\Big[\alpha_j^i\int_{\mathbb{R}^n}W^{p-1}|\nabla W|^2dx+o(1)\Big].$$
Combining the estimates for $L_{i}^{\epsilon}$ and $R_i^{\epsilon}$ together, we obtain for each $j=1,\cdots,n$,
\begin{equation}\label{FR}
\begin{split}
&\Big[\frac{F_n^{\frac{n-1}{n-2}}\mathcal{K}_nA_0}{2}\sum_{k=1}^{n}\alpha_k^i\frac{\partial^2}{\partial x_j\partial x_k}H(x_0,x_0)+o(1)\Big]\\=&-\big\|u_{\epsilon}\big\|_{\infty}^{\frac{2n}{n-2}}(n-2)\frac{\Gamma_nC_N}{n}(p-2)(1-\lambda_{i,\epsilon})\Big[\alpha_j^i\int_{\mathbb{R}^n}W^{p-1}|\nabla W|^2dx+o(1)\Big].
\end{split}
\end{equation}
Owing to $\alpha_k^i\neq0$ in Theorem \ref{Figalli},
\begin{equation}\label{FR-1}
\frac{\lambda_{i,\epsilon}-1}{\epsilon^{\frac{n}{n+2}}}\rightarrow\widetilde{\mathcal{H}}\frac{\sum_{k=1}^{n}\alpha_k^i\frac{\partial^2}{\partial x_j\partial x_k}H(x_0,x_0)}{\alpha_j^i}=:\widetilde{\mathcal{H}}\xi_i,
\end{equation}
with
$$\widetilde{\mathcal{H}}:=\frac{pn\tilde{\alpha}_n\phi(x_0)C_{HL}^{-\frac{n+2}{4}}}{2(n-2)(p-2)\Gamma_n\tilde{c}_{n,\mu}^{2(p-2)}}\frac{(\int_{\mathbb{R}^n}W(x)dx)^3}{\int_{\mathbb{R}^n}W^{p-1}|\nabla W|^2dx}.$$	
Thus, thanks to \eqref{FR} and \eqref{FR-1}, there holds
\begin{equation}\label{FR-2}
\sum_{k=1}^{n}\alpha_k^i\frac{\partial^2}{\partial x_j\partial x_k}H(x_0,x_0)=\xi_i\alpha_j^i.
\end{equation}	
This implies that $\xi_i$ is an eigenvalue of the Hessian matrix $D^{2}H(x_0,x_0)$ with $\alpha^i$ as its corresponding eigenvector. In virtue of Lemma \ref{wanfan}, we readily establish that $\xi_i$ for $i=2,\cdots,n+1$ are precisely the $n$ eigenvalues $\nu_1,\cdots,\nu_n$ of the Hessian matrix $D^{2}H(x_0,x_0)$. Furthermore, it is important to note that the ordering $\lambda_{2,\epsilon}\leq\cdots\leq\lambda_{n+1,\epsilon}$, when combined with \eqref{FR-1}, implies $\xi_i=\nu_{i-1}$ for $i=2,\cdots,n+1$. This concludes the proof.

	\end{proof}

\section{Qualitative characteristics for $(n+2)$-th eigenvalue and corresponding to eigenfunction}\label{section7}
In this section, we are devoted to show that Theorem \ref{thmprtb} and Corollary \ref{emm-1}. 	
We first need estimate for $\lambda_{n+2,\epsilon}$, which is similar to the treatment in Lemmas \ref{baowenbei} and \ref{limitlama}.
\begin{lem}
There holds
\begin{equation}\label{ba}
\lim\limits_{\epsilon\rightarrow0}\lambda_{n+2,\epsilon}=1.
		\end{equation}
\end{lem}
\begin{proof}
By virtue of Theorem \ref{remainder terms} and Lemma \ref{limitlama},
it is enough to show that $$\limsup_{\epsilon\rightarrow0}\lambda_{n+2,\epsilon}\leq1.$$
Referring to \eqref{varep}, we let $\mathcal{W}$ be a vector space whose basis is $$\{u_{\epsilon}\}\cup\{\xi_{j,\epsilon}: 1\leq j\leq n+1\},$$
so that any function $\tilde{f}\in\mathcal{W}\setminus\{0\}$ can be write as
\begin{equation}\label{nonzero-00}
\tilde{f}=a_0u_{\epsilon}+\psi\left(\sum_{j=1}^{n}a_j\frac{\partial u_{\epsilon}}{\partial x_j}+d\Big((x-x_{\epsilon})\cdot\nabla u_{\epsilon}+\frac{2}{p-1-\epsilon}u_{\epsilon}\Big)\right):=a_0u_{\epsilon}+\psi\tilde{z}_{\epsilon},
\end{equation}
for some $a_0,~a_j,~d\in\mathbb{R}$.
Then we have
		\begin{equation}\label{minmax-00}
\begin{split}
			\lambda_{n+2,\epsilon}&=\min\limits_{\substack{\mathcal{W}\subset H_{0}^1(\Omega),\\ dim\mathcal{W}=n+2}}\max\limits_{\tilde{f}\in\mathcal{W}\setminus\{0\}}\frac{\int_{\Omega}|\nabla \tilde{f}(x)|^2dx}{p_{\epsilon}^{(1)}\int_{\Omega}\Big(|x|^{-(n-2)} \ast \big(u_{\epsilon}^{p-1-\epsilon}\tilde{f}\big)\Big)u_{\epsilon}^{p-1-\epsilon}\tilde{f}dx
		+p_{\varepsilon}^{(2)}\int_{\Omega}\Big(|x|^{-(n-2)} \ast u_{\epsilon}^{p-\epsilon}\Big)u_{\epsilon}^{p-2-\epsilon}\tilde{f}^2dx}			\\&\leq\max\limits_{\tilde{f}\in\mathcal{W}\setminus\{0\}}\frac{\int_{\Omega}|\nabla \tilde{f}(x)|^2dx}{p_{\epsilon}^{(1)}\int_{\Omega}\Big(|x|^{-(n-2)} \ast \big(u_{\epsilon}^{p-1-\epsilon}\tilde{f}\big)\Big)u_{\epsilon}^{p-1-\epsilon}\tilde{f}dx
		+p_{\epsilon}^{(2)}\int_{\Omega}\Big(|x|^{-(n-2)} \ast u_{\epsilon}^{p-\epsilon}\Big)u_{\epsilon}^{p-2-\epsilon}\tilde{f}^2dx}\\&
:=\max\limits_{\tilde{f}\in\mathcal{W}\setminus\{0\}}\mathcal{B}_{\epsilon,1}
\end{split}
		\end{equation}
with $p_{\epsilon}^{(1)}=(p-\epsilon),~ p_{\epsilon}^{(2)}=(p-1-\epsilon)$. Together with \eqref{dend}, we find that $\tilde{z}_{\epsilon}$ solves \eqref{A6}. Consequently, by replacing $z_{\epsilon}$ with $\tilde{z}_{\epsilon}$, we have that \eqref{kuangquanshui}-\eqref{psi-z} hold and yield $\mathcal{B}_{\epsilon,1}=1+\mathcal{B}_{\epsilon,2}/\mathcal{B}_{\epsilon,3}$. Here, $\mathcal{B}_{\epsilon,2}$ and $\mathcal{B}_{\epsilon,3}$ are defined as $\mathcal{P}_{\epsilon,2}$ and $\mathcal{P}_{\epsilon,3}$ were in the proof of Lemma \ref{baowenbei}, with $z_{\epsilon}$ replaced by $\tilde{z}_{\epsilon}$. Our next step is to estimate each term of $\mathcal{B}_{\epsilon,2}$ and $\mathcal{B}_{\epsilon,3}$. To begin, from \eqref{m000} we find that
\begin{equation}\label{partial}
\begin{split}
&\int_{\Omega}|\nabla\psi|^2\tilde{z}^2_{\epsilon}dx=\int_{\Omega}\big|\nabla\psi\big|^2\sum_{j=1}^{n}a_j\frac{\partial u_{\epsilon}}{\partial x_j}(x)\sum_{k=1}^{n}a_k\frac{\partial u_{\epsilon}}{\partial x_k}(x)dx\\&
=\int_{\Omega}\big|\nabla\psi\big|^2\Big[2d\sum_{j=1}^{n}a_j\frac{\partial u_{\epsilon}}{\partial x_j}+\Big((x-x_{\epsilon})\cdot\nabla u_{\epsilon}+\frac{2}{p-1-\epsilon}u_{\epsilon}\Big)\Big]\Big((x-x_{\epsilon})\cdot\nabla u_{\epsilon}+\frac{2}{p-1-\epsilon}u_{\epsilon}\Big)dx\\&
=O\Big(\big\|u_{\epsilon}\big\|_{L^{\infty}}^{-2}\Big).
\end{split}
\end{equation}
We decompose
\begin{equation*}
			\begin{split}
&\int_{\Omega}\Big(|x|^{-(n-2)} \ast u_{\epsilon}^{p-\epsilon}\Big)u_{\epsilon}^{p-1-\epsilon}\psi \tilde{z}_{\epsilon}dx=\int_{\Omega}\int_{\Omega}\frac{u_{\epsilon}^{p-1-\epsilon}(x)\psi(x) u_{\epsilon}^{p-\epsilon}(y)}{|x-y|^{n-2}}\sum_{j=1}^{i-1}a_j\frac{\partial u_{\epsilon}}{\partial x_j}(x) dxdy
\\&+\int_{\Omega}\Big(|x|^{-(n-2)} \ast u_{\epsilon}^{p-\epsilon}\Big)u_{\epsilon}^{p-1-\epsilon}\psi \big[(x-x_{\epsilon})\cdot\nabla u_{\epsilon}+\frac{2}{p-1-\epsilon}u_{\epsilon}\big]dx:=\mathcal{I}_1+\mathcal{I}_2.
\end{split}
\end{equation*}
Furthermore, we have
\begin{equation*}
			\begin{split}
\mathcal{I}_2&=\Big[\int_{\{|x-x_{\epsilon}|\leq\varrho\}}\int_{\{|x-x_{\epsilon}|\leq\varrho\}}
+\int_{\Omega\setminus\{|x-x_{\epsilon}|\leq\varrho\}}\int_{\{|x-x_{\epsilon}|\leq\varrho\}}+\int_{\{|x-x_{\epsilon}|\leq\varrho\}}\int_{\Omega\setminus\{|x-x_{\epsilon}|\leq\varrho\}}\\&+\int_{\Omega\setminus\{|x-x_{\epsilon}|\leq\varrho\}}\int_{\Omega\setminus\{|x-x_{\epsilon}|\leq\varrho\}}\Big]
\frac{u_{\epsilon}^{p-\epsilon}(x) u_{\epsilon}^{p-1-\epsilon}(y)\psi}{|x-y|^{n-2}}\big[(x-x_{\epsilon})\cdot\nabla u_{\epsilon}+\frac{2}{p-1-\epsilon}u_{\epsilon}\big]dxdy\\&=:\mathcal{E}_{1}+\mathcal{E}_{2}+\mathcal{E}_{3}+\mathcal{E}_{4}.
\end{split}
\end{equation*}
An integration by parts yields
\begin{equation*}
			\begin{split}				
\mathcal{E}_{1}&=\frac{1}{p-\epsilon}\int_{\{|x-x_{\epsilon}|\leq\varrho\}}\int_{\{|x-x_{\epsilon}|\leq\varrho\}}
\frac{u_{\epsilon}^{p-\epsilon}(y) }{|x-y|^{n-2}}\sum_{i=1}^{n}\frac{\partial}{\partial x_i}\big[(x_i-x_{i,\epsilon}) u_{\epsilon}^{p-\epsilon}\big]\psi dxdy\\&\hspace{3mm}-\big(\frac{n}{p-\epsilon}-\frac{2}{p-1-\epsilon}\big)\int_{\{|x-x_{\epsilon}|\leq\varrho\}}\int_{\{|x-x_{\epsilon}|\leq\varrho\}}
\frac{u_{\epsilon}^{p-\epsilon}(x) u_{\epsilon}^{p-\epsilon}(y)\psi}{|x-y|^{n-2}}dxdy\\
&=\frac{n-2}{p-\epsilon}\int_{\{|x-x_{\epsilon}|\leq\varrho\}}\int_{\{|x-x_{\epsilon}|\leq\varrho\}}
\frac{(x-y)(x-x_{\epsilon})u_{\epsilon}^{p-\epsilon}(x) u_{\epsilon}^{p-\epsilon}(y)\psi}{|x-y|^{n}} dxdy\\&\hspace{3mm}+\frac{1}{p-\epsilon}\int_{\{|x-x_{\epsilon}|=\varrho\}}\int_{\{|x-x_{\epsilon}|\leq\varrho\}}
\frac{u_{\epsilon}^{p-\epsilon}(x) u_{\epsilon}^{p-\epsilon}(y)\psi}{|x-y|^{n-2}}(x-x_{\epsilon})\cdot\nu dxdy\\&\hspace{3mm}-\big(\frac{n}{p-\epsilon}-\frac{2}{p-1-\epsilon}\big)\int_{\{|x-x_{\epsilon}|\leq\varrho\}}\int_{\{|x-x_{\epsilon}|\leq\varrho\}}\frac{u_{\epsilon}^{p-\epsilon}(x) u_{\epsilon}^{p-\epsilon}(y) \psi}{|x-y|^{n-2}}dxdy\\&=:\frac{n-2}{p-\epsilon}\widetilde{\mathcal{E}}_1+\frac{1}{p-\epsilon}\widetilde{\mathcal{E}}_2+\big(-\frac{n}{p-\epsilon}+\frac{2}{p-1-\epsilon}\big)\widetilde{\mathcal{E}}_3.
\end{split}
\end{equation*}
Similarly to estimates of $I_1-I_4$ in Lemma \eqref{baowenbei}, we have
\begin{equation*}
			\begin{split}	
\widetilde{\mathcal{E}}_2=O\Big(\big\|u_{\epsilon}\big\|_{\infty}^{\frac{\epsilon(n+2)}{4}-p}\Big).
\end{split}
\end{equation*}
Let us now estimate $\widetilde{\mathcal{E}}_1$.
Exploiting the definition of  $\psi$ and symmetry, we get that
\begin{equation*}
\begin{split}
\int_{\{|x-x_{\epsilon}|\leq\varrho\}}\int_{\{|x-x_{\epsilon}|\leq\varrho\}}
&\frac{(x-y)(x-x_{\epsilon})u_{\epsilon}^{p-\epsilon}(x) u_{\epsilon}^{p-\epsilon}(y)}{|x-y|^{n}}dxdy\\&
=-\int_{\{|x-x_{\epsilon}|\leq\varrho\}}\int_{\{|x-x_{\epsilon}|\leq\varrho\}}
\frac{(y-x)(x_{\epsilon}-y)u_{\epsilon}^{p-\epsilon}(x) u_{\epsilon}^{p-\epsilon}(y)}{|x-y|^{n}}dxdy,
\end{split}
\end{equation*}
then since $\psi=1$ in $B_{\varrho}(x_{\epsilon})$ we have
\begin{equation*}
\frac{n-2}{p-\epsilon}\widetilde{\mathcal{E}}_{1}=\frac{n-2}{2\cdot(p-\epsilon)}\int_{\{|x-x_{\epsilon}|\leq\varrho\}}\int_{\{|x-x_{\epsilon}|\leq\varrho\}}
\frac{u_{\epsilon}^{p-\epsilon}(x) u_{\epsilon}^{p-\epsilon}(y) \psi}{|x-y|^{n-2}}dxdy.
\end{equation*}
This implies that
\begin{equation*}
\frac{n-2}{p-\epsilon}\widetilde{\mathcal{E}}_{1}+\big(\frac{2}{p-1-\epsilon}-\frac{n}{p-\epsilon}\big)\widetilde{\mathcal{E}}_{3}=O(\epsilon)=O\Big(\big\|u_{\epsilon}\big\|_{\infty}^{-2}\Big),
\end{equation*}
by \eqref{Fn}. Combing these estimates, we obtain
\begin{equation}\label{e31}
\mathcal{E}_{1}=O\Big(\big\|u_{\epsilon}\big\|_{\infty}^{-2}\Big)+O\Big(\big\|u_{\epsilon}\big\|_{\infty}^{\frac{\epsilon(n+2)}{4}-p}\Big).
\end{equation}
By Hardy-Littlewood-Sobolev inequality, Holder inequality and Sobolev inequality we also deduce that
\begin{equation}\label{e32}
\begin{split}
\mathcal{E}_{2}\leq&C\big\|u_{\epsilon}\big\|_{\infty}^{p-\epsilon}\Big[\int_{\Omega\setminus\{|x-x_{\epsilon}|\leq\varrho\}}
\Big(\frac{1}{1+\|u_{\epsilon}\|_{\infty}^{p-1-\epsilon}|x-x_{\epsilon}|^2}\Big)^{\frac{n(p-\epsilon)}{p}}dx\Big]^{\frac{n+2}{2n}}\Big[\int_{\{|x-x_{\epsilon}|\leq\varrho\}}u_{\epsilon}^{\frac{2n}{n-2}-\frac{n\epsilon}{2}}dx\Big]^{\frac{2}{n}}
\\ \times&\Big[\int_{\Omega\setminus\{|x-x_{\epsilon}|\leq\varrho\}}\big|(x-x_{\epsilon})\cdot\nabla u_{\epsilon}+\frac{2}{p-1-\epsilon}u_{\epsilon}\big|^{\frac{2n}{n-2}}dx\Big]^{\frac{n-2}{2n}}
\\
\leq&C\frac{\|u_{\epsilon}\|_{\infty}^{p-\epsilon}}{\|u_{\epsilon}\|_{\infty}^{\frac{(p-1-\epsilon)(n+2)}{4}}}
\Big[\int_{\Omega\setminus\{|x|\leq\varrho\|u_{\epsilon}\|_{\infty}^{\frac{p-1-\epsilon}{2}}\}}\frac{1}{(1+|x|^2)^{\frac{n(p-\epsilon)}{p}}}dy\Big]^{\frac{n+2}{2n}}=O\Big(\frac{1}{\|u_{\epsilon}\|_{\infty}^{p-[1+\frac{\epsilon(n-2)(p-\epsilon)}{2}]\epsilon}}\Big),
\end{split}
		\end{equation}
since \eqref{dus}.
By Hardy-Littlewood-Sobolev inequality, \eqref{decay1} and $\frac{\partial }{\partial\mu}W[0,\mu](x)|_{\mu=1}=\frac{n-2}{2}\frac{1-|x|^2}{(1+|x|^2)^{n/2}}=O(W[0,1])$,
we have
\begin{equation}\label{e33}
\begin{split}
\mathcal{E}_{3}\leq&
 \frac{\|u_{\epsilon}\|_{\infty}^{p-\epsilon}}{\|u_{\epsilon}\|_{\infty}^{\frac{(p-1-\epsilon)(n+2)}{4}}}\Big[\int_{\{|x-x_{\epsilon}|\leq\varrho\}}u_{\epsilon}^{\frac{2n}{n-2}-\frac{2n\epsilon}{n+2}}dx\Big]^{\frac{n+2}{2n}}
 \\& \times\Big[\int_{\Omega_{\epsilon}\setminus\{|x|\leq\varrho\|u_{\epsilon}\|_{\infty}^{\frac{p-1-\epsilon}{2}}\}}\big|\tilde{u}_{\epsilon}^{p-1-\epsilon}[y_{\epsilon}\cdot\nabla \tilde{u}_{\epsilon}+\frac{2}{p-1-\epsilon}\tilde{u}_{\epsilon}]\big|^{\frac{2n}{n+2}}dx\Big]^{\frac{n+2}{2n}}
 \\
\leq&C\frac{\|u_{\epsilon}\|_{\infty}^{p-\epsilon}}{\|u_{\epsilon}\|_{\infty}^{\frac{(p-1-\epsilon)(n+2)}{4}}}
\Big[\int_{\Omega_{\epsilon}\setminus\{|y|\leq\varrho\|u_{\epsilon}\|_{\infty}^{\frac{p-1-\epsilon}{2}}\}}\frac{1}{(1+|y|^2)^{\frac{n(p-\epsilon)}{p}}}dy\Big]^{\frac{n+2}{2n}}=O\Big(\frac{1}{\|u_{\epsilon}\|_{\infty}^{p-[1+\frac{\epsilon(n-2)(p-\epsilon)}{2}]\epsilon}}\Big).
\end{split}
		\end{equation}
Similarly, we can obtain an estimate for $\mathcal{E}_{4}$:
\begin{equation}\label{e34}
\mathcal{E}_{4}=O\Big(\frac{1}{\|u_{\epsilon}\|_{\infty}^{2p-[2+\epsilon(n-2)(p-\epsilon)]\epsilon}}\Big).
\end{equation}
Collecting all the previous estimates \eqref{e31}-\eqref{e34} for $\mathcal{I}_{2}$ and the estimate of $\mathcal{I}_{1}$ from the proof of Lemma \ref{baowenbei}, we find that
\begin{equation}\label{e35}
\begin{split}
\int_{\Omega}\Big(|x|^{-(n-2)} \ast u_{\epsilon}^{p-\epsilon}\Big)u_{\epsilon}^{p-1-\epsilon}\psi \tilde{z}_{\epsilon}dx
=O\Big(\big\|u_{\epsilon}\big\|_{\infty}^{-2}\Big) +O\Big(\big\|u_{\epsilon}\big\|_{\infty}^{\frac{\epsilon(n+2)}{4}-p}\Big)
+O\Big(\big\|u_{\epsilon}\big\|_{\infty}^{[1+\frac{\epsilon(n-2)(p-\epsilon)}{2}]\epsilon-p}\Big).
\end{split}
\end{equation}
Arguing in the same way, we deduce
\begin{equation}\label{e36}
				\int_{\Omega}\Big(|x|^{-(n-2)} \ast \big(u_{\epsilon}^{p-1-\epsilon}\psi \tilde{z}_{\epsilon}\big)\Big)u_{\epsilon}^{p-\epsilon}dx=O\Big(\big\|u_{\epsilon}\big\|_{\infty}^{-2}\Big)+O\Big(\big\|u_{\epsilon}\big\|_{\infty}^{\frac{\epsilon(n+2)}{4}-p}\Big)
+O\Big(\big\|u_{\epsilon}\big\|_{\infty}^{[1+\frac{\epsilon(n-2)(p-\epsilon)}{2}]\epsilon-p}\Big).
		\end{equation}
Furthermore, thanks to symmetry, one has that
\begin{equation*}
\begin{split}
&\int_{\Omega}\Big(|x|^{-(n-2)} \ast \big(u_{\epsilon}^{p-1-\epsilon} \tilde{z}_{\epsilon}\big)\Big)u_{\epsilon}^{p-1-\epsilon}\psi^2 \tilde{z}_{\epsilon}dx-\int_{\Omega}\Big(|x|^{-(n-2)} \ast \big(u_{\epsilon}^{p-1-\epsilon}\psi \tilde{z}_{\epsilon}\big)\Big)u_{\epsilon}^{p-1-\epsilon}\psi \tilde{z}_{\epsilon}dx\\&
=\frac{1}{2}\int_{\Omega}\int_{\Omega}\frac{(\psi(x)-\psi(y))^2u_{\epsilon}^{p-1-\epsilon}(x) u_{\epsilon}^{p-1-\epsilon}(y)}{|x-y|^{n-2}}\sum_{j=1}^{n}a_j\frac{\partial u_{\epsilon}}{\partial x_j}\sum_{k=1}^{n}a_k\frac{\partial u_{\epsilon}}{\partial y_k}dxdy\\&
+\frac{d}{2}\int_{\Omega}\int_{\Omega}\frac{(\psi(x)-\psi(y))^2u_{\epsilon}^{p-1-\epsilon}(x) u_{\epsilon}^{p-1-\epsilon}(y)}{|x-y|^{n-2}}\sum_{j=1}^{n}a_j\frac{\partial u_{\epsilon}}{\partial x_j}\big[(y-x_{\epsilon})\cdot\nabla u_{\epsilon}+\frac{2}{p-1-\epsilon}u_{\epsilon}\big]dxdy\\&
+\frac{d}{2}\int_{\Omega}\int_{\Omega}\frac{(\psi(x)-\psi(y))^2u_{\epsilon}^{p-1-\epsilon}(x) u_{\epsilon}^{p-1-\epsilon}(y)}{|x-y|^{n-2}}\sum_{k=1}^{n}a_k\frac{\partial u_{\epsilon}}{\partial y_k}\big[(x-x_{\epsilon})\cdot\nabla u_{\epsilon}+\frac{2}{p-1-\epsilon}u_{\epsilon}\big]dxdy\\&
+\frac{d^2}{2}\int_{\Omega}\int_{\Omega}\frac{(\psi(x)-\psi(y))^2u_{\epsilon}^{p-1-\epsilon} u_{\epsilon}^{p-1-\epsilon}}{|x-y|^{n-2}}\big[(x-x_{\epsilon})\cdot\nabla u_{\epsilon}+\frac{2}{p-1-\epsilon}u_{\epsilon}\big]\big[(y-x_{\epsilon})\cdot\nabla u_{\epsilon}+\frac{2}{p-1-\epsilon}u_{\epsilon}\big]dxdy\\&
:=\mathcal{F}_1+\mathcal{F}_2+\mathcal{F}_3+\mathcal{F}_4.
\end{split}
\end{equation*}
A direct computation shows
\begin{equation*}
\begin{split}
\mathcal{F}_2&=\frac{d}{2(p-\epsilon)^2}\int_{\Omega}\int_{\Omega}\frac{(\psi(x)-\psi(y))^2 }{|x-y|^{n-2}}\sum_{j=1}^{n}a_j\frac{\partial }{\partial x_j}u_{\epsilon}^{p-\epsilon}(x) \sum_{k=1}^{n}\frac{\partial}{{\partial y_k}}\big[(y_k-x_{\epsilon,k}) u_{\epsilon}^{p-\epsilon}\big]dxdy\\&
-\frac{[(p-1-\epsilon)n-2]d}{2(p-1-\epsilon)(p-\epsilon)^2}\int_{\Omega}\int_{\Omega}\frac{(\psi(x)-\psi(y))^2 u_{\epsilon}^{p-\epsilon}(y)}{|x-y|^{n-2}}\sum_{j=1}^{n}a_j\frac{\partial }{\partial x_j}u_{\epsilon}^{p-\epsilon}(x)dxdy\\&
:=\mathcal{F}_{21}+\mathcal{F}_{22}.
\end{split}
\end{equation*}
By an integration by parts yields, definition of $\psi$ and Hardy-Littlewood-Sobolev inequality, we get
\begin{equation*}
\begin{split}
\mathcal{F}_{22}&=\frac{[(p-1-\epsilon)n-2]d}{2(p-1-\epsilon)(p-\epsilon)^2}\sum_{j=1}^{n}a_j\int_{\Omega}\int_{\Omega}\frac{\partial}{\partial x_j}\Big(\frac{(\psi(x)-\psi(y))^2}{|x-y|^{n-2}}\Big)u_{\epsilon}^{p-\epsilon}(x)u_{\epsilon}^{p-\epsilon}(y)dxdy
\\&=\frac{[(p-1-\epsilon)n-2]d}{2(p-1-\epsilon)(p-\epsilon)^2}\sum_{j=1}^{n}a_j\int_{\Omega}\int_{\Omega}\Big[\frac{\partial}{\partial x_j}\Big(\frac{1}{|x-y|^{n-2}}\Big)(\psi(x)-\psi(y))^2+\frac{\partial_{x_j}\psi(x)}{|x-y|^{n-2}}\Big]u_{\epsilon}^{p-\epsilon}(x)u_{\epsilon}^{p-\epsilon}(y)dxdy
\\&=O\Big(\big\|u_{\epsilon}\big\|_{\infty}^{\frac{\epsilon(n+2)}{4}-p}\Big).
\end{split}
\end{equation*}
Similar to the estimates of each term of $J_0$ in \eqref{JO}, due to \eqref{4J} and integrate by parts on $\Omega$, by exploiting Hardy Littlewood Sobolev inequality and definition of $\psi$, we obtain
\begin{equation*}
\begin{split}
\mathcal{F}_{21}&=\frac{d}{2(p-\epsilon)^2}\sum_{j=1}^{n}\sum_{k=1}^{n}a_j\int_{\Omega}\int_{\Omega}\frac{\partial^2 }{\partial y_k\partial x_j}\Big(\frac{(\psi(x)-\psi(y))^2 }{|x-y|^{n-2}}\Big)(y_k-x_{\epsilon,k})u_{\epsilon}^{p-\epsilon}(y)u_{\epsilon}^{p-\epsilon}(x)dxdy\\&
=\frac{d}{2(p-\epsilon)^2}\sum_{j=1}^{n}\sum_{k=1}^{n}a_j\int_{\Omega}\int_{\Omega}(y_k-x_{\epsilon,k})u_{\epsilon}^{p-\epsilon}(y)u_{\epsilon}^{p-\epsilon}(x)[J_1+J_2+J_3+J_4]dxdy\\&
=O\Big(\big\|u_{\epsilon}\big\|_{\infty}^{\frac{\epsilon(n+2)}{4}-p}\Big).
\end{split}
\end{equation*}
Hence, combining the above estimates entails that
\begin{equation}\label{F22}
\mathcal{F}_{2}=O\Big(\big\|u_{\epsilon}\big\|_{\infty}^{\frac{\epsilon(n+2)}{4}-p}\Big).
\end{equation}
Analogously, we have
\begin{equation}\label{F33}
\mathcal{F}_{3}=O\Big(\big\|u_{\epsilon}\big\|_{\infty}^{\frac{\epsilon(n+2)}{4}-p}\Big).
\end{equation}
Next we estimate $\mathcal{F}_{4}$. Simple computations give
\begin{equation*}
\begin{split}
\mathcal{F}_4&=\frac{d^2}{2(p-\epsilon)^2}\int_{\Omega}\int_{\Omega}\frac{(\psi(x)-\psi(y))^2 }{|x-y|^{n-2}}\sum_{j=1}^{n}\frac{\partial }{\partial x_j}\big[(x_j-x_{\epsilon,j}) u_{\epsilon}^{p-\epsilon}\big]\sum_{k=1}^{n}\frac{\partial}{{\partial y_k}}\big[(y_k-x_{\epsilon,k}) u_{\epsilon}^{p-\epsilon}\big]dxdy\\&
-\frac{[(p-1-\epsilon)n-2]d^2}{2(p-1-\epsilon)(p-\epsilon)^2}\int_{\Omega}\int_{\Omega}\frac{(\psi(x)-\psi(y))^2 u_{\epsilon}^{p-\epsilon}(y)}{|x-y|^{n-2}}\sum_{j=1}^{n}\frac{\partial }{\partial x_j}\big[(x_j-x_{\epsilon,j})u_{\epsilon}^{p-\epsilon}\big]dxdy\\&
-\frac{[(p-1-\epsilon)n-2]d^2}{2(p-1-\epsilon)(p-\epsilon)^2}\int_{\Omega}\int_{\Omega}\frac{(\psi(x)-\psi(y))^2 u_{\epsilon}^{p-\epsilon}(x)}{|x-y|^{n-2}}\sum_{k=1}^{n}\frac{\partial}{{\partial y_k}}\big[(y_k-x_{\epsilon,k}) u_{\epsilon}^{p-\epsilon}\big]dxdy\\&
:=\mathcal{F}_{41}+\mathcal{F}_{42}+\mathcal{F}_{43}.
\end{split}
\end{equation*}
Using Hardy Littlewood Sobolev inequality and definition of $\psi$, we compute the following estimate
\begin{equation*}
\begin{split}
\mathcal{F}_{41}&=\frac{d^2}{2(p-\epsilon)^2}\sum_{j=1}^{n}\sum_{k=1}^{n}\int_{\Omega}\int_{\Omega}\frac{\partial^2}{\partial y_k\partial x_j}\Big(\frac{(\psi(x)-\psi(y))^2 }{|x-y|^{n-2}}\Big)(x_j-x_{\epsilon,j}) u_{\epsilon}^{p-\epsilon}(y_k-x_{\epsilon,k}) u_{\epsilon}^{p-\epsilon}dxdy\\&
=\frac{d^2}{2(p-\epsilon)^2}\sum_{j=1}^{n}\sum_{k=1}^{n}\int_{\Omega}\int_{\Omega}(x_j-x_{\epsilon,j}) u_{\epsilon}^{p-\epsilon}(y_k-x_{\epsilon,k}) u_{\epsilon}^{p-\epsilon}[J_1+J_2+J_3+J_4]dxdy\\&
=O\Big(\big\|u_{\epsilon}\big\|_{\infty}^{\frac{\epsilon(n+2)}{4}-p}\Big).
\end{split}
\end{equation*}
Similarly, we have
\begin{equation*}
\mathcal{F}_{42}
=O\Big(\big\|u_{\epsilon}\big\|_{\infty}^{\frac{\epsilon(n+2)}{4}-p}\Big),\hspace{2mm}\mathcal{F}_{43}
=O\Big(\big\|u_{\epsilon}\big\|_{\infty}^{\frac{\epsilon(n+2)}{4}-p}\Big).
\end{equation*}
Therefore, we deduce that
\begin{equation}\label{AF4}
\mathcal{F}_{4}
=O\Big(\big\|u_{\epsilon}\big\|_{\infty}^{\frac{\epsilon(n+2)}{4}-p}\Big).
\end{equation}
Combining \eqref{F22}-\eqref{AF4} and the estimate of $\mathcal{F}_1$ in \eqref{JO} and \eqref{dx} tell us that
\begin{equation}\label{Omega}
\begin{split}
&\int_{\Omega}\Big(|x|^{-(n-2)} \ast \big(u_{\epsilon}^{p-1-\epsilon} \tilde{z}_{\epsilon}\big)\Big)u_{\epsilon}^{p-1-\epsilon}\psi^2 \tilde{z}_{\epsilon}dx-\int_{\Omega}\Big(|x|^{-(n-2)} \ast \big(u_{\epsilon}^{p-1-\epsilon}\psi \tilde{z}_{\epsilon}\big)\Big)u_{\epsilon}^{p-1-\epsilon}\psi \tilde{z}_{\epsilon}dx\\&
=O\Big(\big\|u_{\epsilon}\big\|_{\infty}^{\frac{\epsilon(n+2)}{4}-p}\Big).
\end{split}
\end{equation}
We have the following identity
\begin{equation*}
\begin{split}
&\int_{\Omega}\Big(|x|^{-(n-2)} \ast \big(u_{\epsilon}^{p-1-\epsilon}\psi \tilde{z}_{\epsilon}\big)\Big)u_{\epsilon}^{p-1-\epsilon}\psi \tilde{z}_{\epsilon}dx=\int_{\Omega}\int_{\Omega}\frac{u_{\epsilon}^{p-1-\epsilon}(x)\psi(x)  u_{\epsilon}^{p-1-\epsilon}(y)\psi(y) }{|x-y|^{n-2}}\sum_{j=1}^{n}a_j\frac{\partial u_{\epsilon}}{\partial x_j}\sum_{k=1}^{n}a_k\frac{\partial u_{\epsilon}}{\partial y_k}dxdy\\&
+d^2\int_{\Omega}\int_{\Omega}\frac{u_{\epsilon}^{p-1-\epsilon}(x)\psi(x)  u_{\epsilon}^{p-1-\epsilon}(y)\psi(y) }{|x-y|^{n-2}}\big[(x-x_{\epsilon})\cdot\nabla u_{\epsilon}+\frac{2}{p-1-\epsilon}u_{\epsilon}\big]\big[(y-x_{\epsilon})\cdot\nabla u_{\epsilon}+\frac{2}{p-1-\epsilon}u_{\epsilon}\big]dxdy
\\&
+d\int_{\Omega}\int_{\Omega}\frac{u_{\epsilon}^{p-1-\epsilon}(x)\psi(x)  u_{\epsilon}^{p-1-\epsilon}(y)\psi(y) }{|x-y|^{n-2}}\sum_{k=1}^{n}a_k\frac{\partial u_{\epsilon}(y)}{\partial y_k}\big[(x-x_{\epsilon})\cdot\nabla u_{\epsilon}(x)+\frac{2}{p-1-\epsilon}u_{\epsilon}(x)\big]dxdy\\&+
d\int_{\Omega}\int_{\Omega}\frac{u_{\epsilon}^{p-1-\epsilon}(x)\psi(x)  u_{\epsilon}^{p-1-\epsilon}(y)\psi(y) }{|x-y|^{n-2}}\big[(y-x_{\epsilon})\cdot\nabla u_{\epsilon}+\frac{2}{p-1-\epsilon}u_{\epsilon}\big]\sum_{k=1}^{i-1}a_j\frac{\partial u_{\epsilon}}{\partial x_j}dxdy:=\mathcal{G}_1+\mathcal{G}_2+\mathcal{G}_3+\mathcal{G}_4.
\end{split}
\end{equation*}
Similarly to the argument of \eqref{hand23}, thanks to \eqref{p1-00}-\eqref{thatwwe} and the change of variable $x=\|u_{\epsilon}\|_{\infty}^{-\frac{4-(n-2)\varepsilon}{2(n-2)}}y+x_{\epsilon}$, and $y=\|u_{\epsilon}\|_{\infty}^{-\frac{4-(n-2)\varepsilon}{2(n-2)}}z+x_{\epsilon}$,  we deduce
\begin{equation*}
\begin{split}
\mathcal{G}_2=&d^2\big\|u_{\epsilon}\big\|_{\infty}^{p+1-\epsilon}\int_{\Omega_{\epsilon}}\int_{\Omega_{\epsilon}}\frac{\tilde{u}_{\epsilon}^{p-1-\epsilon}(z)\psi(\big\|u_{\epsilon}\big\|_{\infty}^{-\frac{4-(n-2)\varepsilon}{2(n-2)}}y+x_{\epsilon})  \tilde{u}_{\epsilon}^{p-1-\epsilon}(y)\psi(\big\|u_{\epsilon}\big\|_{\infty}^{-\frac{4-(n-2)\varepsilon}{2(n-2)}}z+x_{\epsilon})}{|y-z|^{n-2}}\\&
\times\big[y_{\epsilon}\cdot\nabla \tilde{u}_{\epsilon}+\frac{2}{p-1-\epsilon}\tilde{u}_{\epsilon}\big]\big[z_{\epsilon}\cdot\nabla \tilde{u}_{\epsilon}+\frac{2}{p-1-\epsilon}\tilde{u}_{\epsilon}\big]dydz\\
=&d^2\big\|u_{\epsilon}\big\|_{\infty}^{p+1-\epsilon}\bigg[\int_{\mathbb{R}^n}\int_{\mathbb{R}^n}\frac{W^{p-1}(y)W^{p-1}(z)}{|y-z|^{n-2}}\frac{1-|z|^2}{(1+|z|^2)^{\frac{n}{2}}}\frac{1-|y|^2}{(1+|y|^2)^{\frac{n}{2}}}\psi^2\big(x_{0}\big)dydz+o(1)\bigg]\\
=&d^2\big\|u_{\epsilon}\big\|_{\infty}^{p+1-\epsilon}\bigg[\Gamma_n\frac{\delta_j^{k}}{n}\int_{\mathbb{R}^n}W^{p-1}(x)\Big(\frac{1-|y|^2}{(1+|y|^2)^{\frac{n}{2}}}\Big)^2dy+o(1)\bigg].
\end{split}
\end{equation*}
Furthermore, by changing variable and recalling the definition of $\tilde{u}_{\epsilon}$, we have
\begin{equation*}
\begin{split}
\mathcal{G}_3=&d\big\|u_{\epsilon}\big\|_{\infty}^{\frac{3(p+1-\epsilon)}{2}+2}\int_{\Omega_{\epsilon}}\int_{\Omega_{\epsilon}}\frac{\tilde{u}_{\epsilon}^{p-1-\epsilon}(y)\psi(\|u_{\epsilon}\|_{\infty}^{-\frac{4-(n-2)\varepsilon}{2(n-2)}}y+x_{\epsilon})  \tilde{u}_{\epsilon}^{p-1-\epsilon}(z)\psi(\|u_{\epsilon}\|_{\infty}^{-\frac{4-(n-2)\varepsilon}{2(n-2)}}z+x_{\epsilon})}{|y-z|^{n-2}}\\&
\times\sum_{k=1}^{n}a_k\frac{\partial \tilde{u}_{\epsilon}(z)}{\partial z_k}\big[y_{\epsilon}\cdot\nabla \tilde{u}_{\epsilon}+\frac{2}{p-1-\epsilon}\tilde{u}_{\epsilon}\big]dydz\\
=&d\big\|u_{\epsilon}\big\|_{\infty}^{\frac{3(p+1-\epsilon)}{2}+2}\bigg[\int_{\mathbb{R}^n}\int_{\mathbb{R}^n}\frac{W^{p-1}(z)W^{p-1}(y)}{|y-z|^{n-2}}\sum_{k=1}^{n}a_k\frac{\partial W(z)}{\partial z_k}\frac{1-|y|^2}{(1+|y|^2)^{\frac{n}{2}}}\psi^2\big(x_{0}\big)dydz+o(1)\bigg]\\
=&d\big\|u_{\epsilon}\big\|_{\infty}^{\frac{3(p+1-\epsilon)}{2}+2}\bigg[(2-n)\sum_{k=1}^{n}a_k\frac{\Gamma_n}{p}\frac{\delta_j^{k}}{n}\int_{\mathbb{R}^n}W^{p-1}(x)\frac{y_k}{(1+|y|^2)^{\frac{n}{2}}}\frac{1-|y|^2}{(1+|y|^2)^{\frac{n}{2}}}dy+o(1)\bigg]\\
=&o\big(\big\|u_{\epsilon}\big\|_{\infty}^{\frac{3(p+1-\epsilon)}{2}+2}\big).
\end{split}
\end{equation*}
Arguing in the same way to $\mathcal{G}_3$, we are able to deduce
\begin{equation*}
\begin{split}
\mathcal{G}_4=o\big(\big\|u_{\epsilon}\big\|_{\infty}^{\frac{3(p+1-\epsilon)}{2}+2}\big).
\end{split}
\end{equation*}
Summing the estimate of $\mathcal{G}_1$ in the proof of Lemma \ref{baowenbei}, and estimates of $\mathcal{G}_2$-$\mathcal{G}_4$, we deduce that
\begin{equation}\label{Omega-0}
\begin{split}
\int_{\Omega}\Big(|x|^{-(n-2)} \ast \big(u_{\epsilon}^{p-1-\epsilon}\psi \tilde{z}_{\epsilon}\big)\Big)u_{\epsilon}^{p-1-\epsilon}\psi \tilde{z}_{\epsilon}dx=&\frac{1}{p}\big\|u_{\epsilon}\big\|_{L^{\infty}}^{\frac{4}{n-2}}\sum_{j=1}^{n}a_j\frac{\delta_j^{k}}{n}\int_{\mathbb{R}^{n}}W^{p-1}(x)\big|\nabla W(x)\big|^2dx
\\&
+d^2\big\|u_{\epsilon}\big\|_{\infty}^{p+1-\epsilon}\Gamma_n\frac{\delta_j^{k}}{n}\int_{\mathbb{R}^n}W^{p-1}(x)\Big(\frac{1-|x|^2}{(1+|x|^2)^{\frac{n}{2}}}\Big)^2dx
\\&+o\big(\big\|u_{\epsilon}\big\|_{\infty}^{\frac{3(p+1-\epsilon)}{2}+2}\big).
\end{split}
\end{equation}
Finally, we also have
\begin{equation*}
\begin{split}
&\int_{\Omega}\Big(|x|^{-(n-2)} \ast u_{\epsilon}^{p-\epsilon}\Big)u_{\epsilon}^{p-2-\epsilon}\psi^2 \tilde{z}_{\epsilon}^2dx=\int_{\Omega}\int_{\Omega}\frac{u_{\epsilon}^{p-\epsilon}(y) u_{\epsilon}^{p-2-\epsilon}(x)\psi^2(x) }{|x-y|^{n-2}}\sum_{j=1}^{n}a_j\frac{\partial u_{\epsilon}}{\partial x_j}\sum_{k=1}^{n}a_k\frac{\partial u_{\epsilon}}{\partial x_k}dxdy\\&
+2d\int_{\Omega}\int_{\Omega}\frac{u_{\epsilon}^{p-\epsilon}(y)  u_{\epsilon}^{p-2-\epsilon}(x)\psi^2(x) }{|x-y|^{n-2}}\sum_{j=1}^{n}a_j\frac{\partial u_{\epsilon}}{\partial x_j}\big[(x-x_{\epsilon})\cdot\nabla u_{\epsilon}(x)+\frac{2}{p-1-\epsilon}u_{\epsilon}(x)\big]dxdy\\&
+d^2\int_{\Omega}\int_{\Omega}\frac{u_{\epsilon}^{p-\epsilon}(y)  u_{\epsilon}^{p-2-\epsilon}(x)\psi^2(x) }{|x-y|^{n-2}}\big[(x-x_{\epsilon})\cdot\nabla u_{\epsilon}(x)+\frac{2}{p-1-\epsilon}u_{\epsilon}(x)\big]^2dxdy:=\mathcal{H}_1+\mathcal{H}_2+\mathcal{H}_3.
\end{split}
\end{equation*}
For $\mathcal{H}_2$, similarly to the estimates of $\mathcal{G}_3-\mathcal{G}_4$, we have
\begin{equation*}
\mathcal{H}_2=o\big(\big\|u_{\epsilon}\big\|_{\infty}^{\frac{3(p+1-\epsilon)}{2}+2}\big).
\end{equation*}
From \eqref{p1-00}-\eqref{thatwwe} and the change of variable $x=\|u_{\epsilon}\|_{\infty}^{-\frac{4-(n-2)\varepsilon}{2(n-2)}}y+x_{\epsilon}$, we obtain that
\begin{equation*}
\begin{split}
\mathcal{H}_3=&d^2\big\|u_{\epsilon}\big\|_{\infty}^{p+1-\epsilon}\int_{\Omega_{\epsilon}}\int_{\Omega_{\epsilon}}\frac{\tilde{u}_{\epsilon}^{p-\epsilon}(z)\psi^2(\|u_{\epsilon}\|_{\infty}^{-\frac{4-(n-2)\varepsilon}{2(n-2)}}y+x_{\epsilon})  \tilde{u}_{\epsilon}^{p-2-\epsilon}(y)\big[y_{\epsilon}\cdot\nabla \tilde{u}_{\epsilon}+\frac{2}{p-1-\epsilon}\tilde{u}_{\epsilon}\big]^2}{|y-z|^{n-2}}dydz\\
=&d^2\big\|u_{\epsilon}\big\|_{\infty}^{p+1-\epsilon}\bigg[\Gamma_n\frac{\delta_j^{k}}{n}\int_{\mathbb{R}^n}W^{2^\ast-2}(y)\Big(\frac{1-|y|^2}{(1+|y|^2)^{\frac{n}{2}}}\Big)^2dy+o(1)\bigg].
\end{split}
\end{equation*}
Utilizing these estimates and estimate of $\mathcal{H}_1$, we deduce that
\begin{equation}\label{Omega-1}
\begin{split}
\int_{\Omega}\Big(|x|^{-(n-2)} \ast u_{\epsilon}^{p-\epsilon}\Big)u_{\epsilon}^{p-2-\epsilon}\psi^2 \tilde{z}_{\epsilon}^2dx&=\big\|u_{\epsilon}\big\|_{L^{\infty}}^{\frac{4}{n-2}}\sum_{j=1}^{n}a_j\frac{\delta_j^{k}}{n}\int_{\mathbb{R}^{n}}W^{2^{\ast}-2}(x)\big|\nabla W(x)\big|^2dx\\&
+d^2\big\|u_{\epsilon}\big\|_{\infty}^{p+1-\epsilon}\Gamma_n\frac{\delta_j^{k}}{n}\int_{\mathbb{R}^n}W^{2^\ast-2}(x)\Big(\frac{1-|x|^2}{(1+|x|^2)^{\frac{n}{2}}}\Big)^2dx
\\&+o\big(\big\|u_{\epsilon}\big\|_{\infty}^{\frac{3(p+1-\epsilon)}{2}+2}\big).
\end{split}
\end{equation}
Hence, by \eqref{partial}, \eqref{e35}, \eqref{e36}, \eqref{Omega} and \eqref{Omega-1}, we conclude that
\begin{equation}\label{617}
\begin{split}
\mathcal{B}_{\epsilon,2}&=-(p_{\epsilon}^{(1)}+p_{\epsilon}^{(2)}-1)a_0^2\int_{\Omega}\Big(|x|^{-(n-2)} \ast u_{\epsilon}^{p-\epsilon}\Big)u_{\epsilon}^{p-\epsilon}dx
+O\Big(\big\|u_{\epsilon}\big\|_{\infty}^{-2}\Big)\\&
\leq O\Big(\big\|u_{\epsilon}\big\|_{\infty}^{-2}\Big),
\end{split}
\end{equation}
and
\begin{equation}\label{618}
\begin{split}
\mathcal{B}_{\epsilon,3}&\geq
\big(\frac{p_{\epsilon}^{(1)}}{p}+p_{\epsilon}^{(2)}\big)\big\|u_{\epsilon}\big\|_{\infty}^{\frac{4}{n-2}}\sum_{j=1}^{n}a_j\frac{\delta_j^{k}}{n}\int_{\mathbb{R}^{n}}W^{p-1}(x)\big|\nabla W(x)\big|^2dx
\\&
+d^2\big(p_{\epsilon}^{(1)}+p_{\epsilon}^{(2)}\big)\big\|u_{\epsilon}\big\|_{\infty}^{p+1-\epsilon}\Gamma_n\frac{\delta_j^{k}}{n}\int_{\mathbb{R}^n}W^{p-1}(x)\Big(\frac{1-|x|^2}{(1+|x|^2)^{\frac{n}{2}}}\Big)^2dx
\\&+o\big(\big\|u_{\epsilon}\big\|_{\infty}^{\frac{3(p+1-\epsilon)}{2}+2}\big)
+
O\Big(\big\|u_{\epsilon}\big\|_{\infty}^{-2}\Big)\\&
\geq\delta>0
\end{split}
\end{equation}
for some constant $\delta$.
It follows from \eqref{minmax-00} and \eqref{617}-\eqref{618} that
\begin{equation*}
\limsup_{\epsilon\rightarrow0}\lambda_{n+2,\epsilon}\leq \limsup_{\epsilon\rightarrow0}\max\limits_{\tilde{f}\in\mathcal{W}\setminus\{0\}}\big(1+\mathcal{B}_{\epsilon,2}/\mathcal{B}_{\epsilon,3}\big)\leq1,
\end{equation*}
and the conclusion follows.
\end{proof}

 Now we are ready to prove Theorem \ref{thmprtb}.
\begin{proof}[\textbf{Proof of Theorem \ref{thmprtb}}]	
We now proceed similarly to Lemma \ref{identity}. For the rescaled eigenfunctions $\tilde{v}_{n+2,\epsilon}$, we deduce that
\begin{equation}\label{viep-1}
\tilde{v}_{n+2,\epsilon}(x)\rightarrow-(n-2)\sum_{k=1}^{n}\frac{\alpha_kx_k}{(1+|x|^2)^{\frac{n}{2}}}+\beta\frac{1-|x|^2}{(1+|x|^2)^{\frac{n}{2}}}\quad\mbox{in}\hspace{2mm} C_{loc}^{1}(\mathbb{R}^n)
\end{equation}
with $(\alpha_1,\cdots,\alpha_{n},\beta)\neq(0,\cdots,0)$ in $\mathbb{R}^{n+1}$. We claim that $(\alpha_1,\cdots,\alpha_{n})=(0,\cdots,0)$. Note that once this claim is established, \eqref{rtbdy46} will follow directly. Indeed, arguing in the same way as in Lemma \ref{wanfan}, we can show that the eigenfunctions $v_{n+2,\epsilon}$ and $v_{m,\epsilon}$ are orthogonal in $H_{0}^{1}(\Omega)$ for $m=2,\cdots,n+1$.  Then we have
\begin{equation}\label{ping-1}
		\begin{split}	(p-\epsilon)\int_{\Omega}\int_{\Omega}&\frac{u_{\epsilon}^{p-1-\epsilon}(y)v_{n+2,\epsilon}(y)u_{\epsilon}^{p-1-\epsilon}(x)v_{m,\epsilon}(x)}{|x-y|^{n-2}}dxdy \\&+(p-1-\epsilon)\int_{\Omega}\int_{\Omega}\frac{u_{\varepsilon}^{p-\epsilon}(y)u_{\epsilon}^{p-2-\epsilon}(x)v_{n+2,\epsilon}(x)v_{m,\epsilon}(x)}{|x-y|^{n-2}} dxdy=0.
\end{split}
		\end{equation}
It follows from \eqref{vie-1} and \eqref{viep-1}-\eqref{ping-1} that
\begin{equation*}
\begin{split}		
&p\sum_{k,h=1}^{n}\int_{\mathbb{R}^n}\int_{\mathbb{R}^n}\frac{W^{p-1}(y)W^{p-1}(x)}{|x-y|^{n-2}}
\Big[-(n-2)\sum_{k=1}^{n}\frac{\alpha_ky_k}{(1+|y|^2)^{\frac{n}{2}}}+\beta\frac{1-|y|^2}{(1+|y|^2)^{\frac{n}{2}}}\Big]\Big(\frac{\alpha_h^mx_h}{(1+|x|^2)^{\frac{n}{2}}}\Big)dxdy\\&+	(p-1)\sum_{k,h=1}^{n}\int_{\mathbb{R}^n}\int_{\mathbb{R}^n}\frac{W^{p}(y)W^{p-2}(x)}{|x-y|^{n-2}}
\Big[-(n-2)\sum_{k=1}^{n}\frac{\alpha_kx_k}{(1+|x|^2)^{\frac{n}{2}}}+\beta\frac{1-|x|^2}{(1+|x|^2)^{\frac{n}{2}}}\Big]\frac{\alpha_h^mx_h}{(1+|x|^2)^{\frac{n}{2}}}dxdy=0.	\end{split}
\end{equation*}
Combining \eqref{W-0} and \eqref{p1-00}-\eqref{thatwwe} yields
\begin{equation*}
\begin{split}
&\big[I(\frac{n-2}{2})S^{\frac{(2-n)}{8}}C_{n,n-2}^{\frac{2-n}{8}}[n(n-2)]^{\frac{n-2}{4}}
p+p-1\big]\beta\sum_{k,h=1}^{n}\int_{\mathbb{R}^n}W^{p-1}(x)\frac{1-|x|^2}{(1+|x|^2)^{\frac{n}{2}}}\frac{\alpha_h^mx_h}{(1+|x|^2)^{\frac{n}{2}}}dx\\&
-4I(\frac{n-2}{2})S^{\frac{(2-n)}{8}}C_{n,n-2}^{\frac{2-n}{8}}[n(n-2)]^{\frac{n-2}{4}}
\sum_{k,h=1}^{n}\int_{\mathbb{R}^n}W^{p-1}(x)\frac{\alpha_kx_k}{(1+|x|^2)^{\frac{n}{2}}}\frac{\alpha_h^mx_h}{(1+|x|^2)^{\frac{n}{2}}}dx\\&
=-4I(\frac{n-2}{2})S^{\frac{(2-n)}{8}}C_{n,n-2}^{\frac{2-n}{8}}[n(n-2)]^{\frac{n-2}{4}}
\sum_{k,h=1}^{n}\alpha_h\alpha_h^m\int_{\mathbb{R}^n}W^{p-1}(x)\frac{x_h^2}{(1+|x|^2)^{n}}dx
=0,
\end{split}
\end{equation*}
which implies $(a,a^{(m)})=0$ for $m=2,\cdots, n+1$, and so the claim follows from Lemma \ref{wanfan}.

To show \eqref{rt}, note that from \eqref{bepus}, \eqref{dus}, \eqref{decay1}, \eqref{day2}, Lemma \ref{regular} and Lemma \ref{ba}, similarly to Lemma \ref{m00}, we have
\begin{equation}\label{number-1}
\big\|u_{\epsilon}\big\|_{\infty}^2v_{n+2,\epsilon}(x)\rightarrow \mathcal{M}_0\beta G(x,x_0)\quad\mbox{in}\hspace{2mm}C^{1,\alpha}(\omega),
\end{equation}
where the convergence is in $C^{1,\alpha}(\omega)$ with $\omega$ any compact set of $\bar{\Omega}$ not containing $x_0$,  the limit point of $x_\epsilon$, and $\mathcal{M}_0=-\big(\frac{4}{n(n+2)}+\frac{1}{n}\big)\Gamma_nC_N\omega_n$.
Thus, thanks to $\beta\neq0$,
\begin{equation*}
\begin{split}
1-\lambda_{n+2,\varepsilon}&=\frac{1}{\|u_{\epsilon}\|_{\infty}^{\frac{4+(n-2)\epsilon}{2}}}\bigg[
-\frac{(n-2)\mathcal{K}_n \mathcal{M}_0}{(2p-1)\Gamma_nC_N}\Big(\int_{\mathbb{R}^n}W^{p-1}(x)\big(\frac{1-|x|^2}{(1+|x|^2)^{\frac{n}{2}}}\big)^2dx\Big)^{-1}H(x_0,x_0)+o(1)\bigg]
\\&=\epsilon(\mathcal{C}_0+o(1))
\end{split}
\end{equation*}
with
$$\mathcal{C}_0=-\frac{(n-2)\mathcal{K}_n \mathcal{M}_0F_n}{(2p-1)\Gamma_nC_N}\Big(\int_{\mathbb{R}^n}W^{p-1}(x)\big(\frac{1-|x|^2}{(1+|x|^2)^{\frac{n}{2}}}\big)^2dx\Big)^{-1}H(x_0,x_0).$$
Then \eqref{rt} is established.

We can now prove that $\lambda_{n+2,\epsilon}$ is simple. Then we may assume by contradiction that there exist at least two eigenfunctions $v_{n+2,\epsilon}^{(1)}$ and $v_{n+2,\epsilon}^{(2)}$ corresponding to $\lambda_{n+2,\epsilon}$ orthogonal in the space $H_{0}^{1}(\Omega)$, and so that
\begin{equation*}
		\begin{split}	(p-\epsilon)\int_{\Omega}\int_{\Omega}&\frac{u_{\epsilon}^{p-1-\epsilon}(y)v_{n+2,\epsilon}^{(1)}(y)u_{\epsilon}^{p-1-\epsilon}(x)v_{n+2,\epsilon}^{(2)}(x)}{|x-y|^{n-2}}dxdy \\&+(p-1-\epsilon)\int_{\Omega}\int_{\Omega}\frac{u_{\varepsilon}^{p-\epsilon}(y)u_{\epsilon}^{p-2-\epsilon}(x)v_{n+2,\epsilon}^{(1)}(x)v_{n+2,\epsilon}^{(2)}(x)}{|x-y|^{n-2}} dxdy=0.
\end{split}
\end{equation*}
Thus, combing \eqref{rtbdy46} and \eqref{p1-00}-\eqref{thatwwe} we have
\begin{equation*}
\begin{split}
&\int_{\mathbb{R}^n}\int_{\mathbb{R}^n}\Big[p\frac{W^{p-1}(y)W^{p-1}(x)}{|x-y|^{n-2}}\frac{1-|y|^2}{(1+|y|^2)^{\frac{n}{2}}}
+(p-1)\frac{W^{p}(y)W^{p-2}(x)}{|x-y|^{n-2}}\frac{1-|x|^2}{(1+|x|^2)^{\frac{n}{2}}}\Big]\frac{1-|x|^2}{(1+|x|^2)^{\frac{n}{2}}}
dxdy\\&
=(2p-1)\Gamma_n\beta_1\beta_2\int_{\mathbb{R}^n}W^{p-1}(x)\Big(\frac{1-|x|^2}{(1+|x|^2)^{\frac{n}{2}}}\Big)^2dx=0,
\end{split}
\end{equation*}
where $\beta_1$ and $\beta_2$ are the convergence coefficients of $v_{n+2,\epsilon}^{(1)}$ and $v_{n+2,\epsilon}^{(2)}$ in \eqref{rtbdy46}, respectively. This gives the expected contradiction and the result follows.

We next want to show that $v_{n+2,\epsilon}$ has only two nodal regions.
For this, we first define
\begin{equation}\label{gama}
\gamma(x)=\frac{1-|x|^2}{(1+|x|^2)^{\frac{n}{2}}}>0\hspace{2mm}\mbox{in the ball}\hspace{2mm}B_1(0);\hspace{2mm}\gamma(x)<0\hspace{2mm}\mbox{in}\hspace{2mm}\mathbb{R}^n\setminus B_1(0).
\end{equation}
Without loss of generality, we may assume that $\beta>0$ in \eqref{rtbdy46}, then by the $C_{loc}^{1}$ convergence  we have $\tilde{v}_{n+2,\epsilon}>0$ in $B_{1/2}(0)$ for $\epsilon>0$ small enough, which means that
$$v_{n+2,\epsilon}>0\hspace{2mm}\mbox{in}\hspace{2mm}\big\{x:|x-x_{\epsilon}|<\frac{1}{2}\|u_{\epsilon}\|_{\infty}^{-\frac{4-(n-2)\epsilon}{2(n-2)}}\big\}.$$
In the same way, in view of $\gamma<0$ on $D_{R}=\{x\in\mathbb{R}^n,~|x|=R\}$ for any $R>1$, we find $v_{n+2,\epsilon}<0$ on $\Omega^{\prime}:=\{x:~|x-x_{\epsilon}|=r\}$ for the radius $r=\frac{1}{2}\|u_{\epsilon}\|_{\infty}^{-\frac{4-(n-2)\epsilon}{2(n-2)}}R$.
Thus, as $R$ large enough, the boundary of $\Omega^{\prime\prime}:=\Omega\setminus\{x,~|x-x_{\epsilon}|<r\}$ satisfies the negative constraint and together with ellipticity and non-positivity of zero-order terms of operators $L_{\epsilon}\phi=-\Delta \phi-\lambda_{n+2,\epsilon} \mathcal{G}_{\epsilon}[\phi]$ hold, and thanks to the maximum principle and the fact $v_{n+2,\epsilon}<0$ on $\Omega^{\prime\prime}$,
\begin{equation}\label{LL}
v_{n+2,\epsilon}<0\hspace{2mm}\mbox{in}\hspace{2mm}\Omega^{\prime\prime}.
\end{equation}
Therefore, we argue by contradiction if $v_{n+2,\epsilon}$ has more than two nodal regions, then these nodal regions must be in domain $D_{\epsilon}:=\big\{x,~|x-x_{\epsilon}|<r\big\}\setminus\big\{x:|x-x_{\epsilon}|<\frac{1}{2}\|u_{\epsilon}\|_{\infty}^{-\frac{4-(n-2)\epsilon}{2(n-2)}}\big\}$.
We may denote one of them as $C_{\epsilon}$, such that $\overline{C}_{\epsilon}\subset D_{\epsilon}$. We consider two cases, depending wether $v_{n+2,\epsilon}<0$ or not.\\
\bigskip

\textbf{Case 1.} If $v_{n+2,\epsilon}>0$, then taking $y_{\epsilon}\in D_{\epsilon}$ as its maximum, we have $v_{n+2,\epsilon}(y_{\epsilon})>0$ and $\nabla v_{n+2,\epsilon}(y_{\epsilon})=0$.
Passing to the limit as $\epsilon\rightarrow0$ with the change of variable $\tilde{y}_{\epsilon}=\|u_{\epsilon}\|_{\infty}^{-\frac{4-(n-2)\epsilon}{2(n-2)}}y_{\epsilon}+x_{\epsilon}$,
we would get $\gamma(\tilde{y}_{0})>0$, $|\tilde{y}_{0}|>\frac{1}{2}$, and $\nabla \gamma(\tilde{y}_{0})=0$, which gives a contradiction to the rigidity of function $\gamma(x)$.\\
\bigskip

\textbf{Case 2.} If $v_{n+2,\epsilon}<0$. Similarly to case 1 and assume that $y_{\epsilon}\in D_{\epsilon}$ is a minimum point of $v_{n+2,\epsilon}$, then we have $v_{n+2,\epsilon}(y_{\epsilon})<0$, $v_{n+2,\epsilon}\equiv0$ on $\partial C_{\epsilon}$ and $\nabla v_{n+2,\epsilon}(y_{\epsilon})=0$, and on passing to the limit as $\epsilon$ small enough so that $\gamma(\tilde{y}_{0})=0$, $|\tilde{y}_{0}|=1$, and $\nabla \gamma(\tilde{y}_{0})=0$, a contradiction. Hence the conclusion follows.

Finally, due to \eqref{LL}, we obtain that $v_{n+2,\epsilon}$ is always negative, and then the closure of the nodal set $\mathcal{N}_{n+2,\epsilon}$ follows. This concludes the proof.
\end{proof}	

We conclude the section by proving the Corollary \ref{emm-1} by exploiting Theorem \ref{Figalli-1}, Theorem \ref{remainder terms} and \eqref{rt}.
\begin{proof}[\textbf{Proof of Corollary \ref{emm-1}}]
By virtue of Theorem \ref{remainder terms} we have that any $\nu_{i-1}<0$ for $i=2,\cdots,n+1$ so that  $\lambda_{1,\epsilon}<1$, and $\lambda_{n+2,\epsilon}>1$ in \eqref{rt} for $\epsilon$ sufficiently small, we obtain
$$1\leq1+ind(-D^2\phi(x_0))\leq ind(u_{\epsilon})\leq n+1,$$
and the conclusion follows.
\end{proof}


\small

	\begin{thebibliography}{99}
		
		
		\bibitem{Ackerman}
N. Ackermann, M. Clapp, and A. Pistoia,
 \emph{Boundary clustered layers near the higher critical exponents}, J. Differential Equations, {\bf 254} (2013), 4168--4193.

\bibitem{ATKINSON-1986}
F. Atkinson and L. Peletier,
 \emph{Emden-Fowler equations involving critical exponents},
Nonlinear Anal. TMA., {\bf 10} (1986), 755--776.

\bibitem{BP} H. Brezis and L. A. Peletier,
         \emph{Asymptotics for elliptic equation involving critical growth, in: Partial Differential Equations and the Calculus of Variations},
vol. I, Birkh\"{a}user, Boston, MA, 1989, PP. 149--192.


\bibitem{B-L-R}
A. Bahri, Y. Y. Li, and O. Rey,
\emph{On a variational problem with lack of compactness: the topological effect of the critical points at infinity}, Calc. Var. Partial Differ. Equ.,
{\bf 3} (1995), 67--93.

\bibitem{cw}
W. Chen and Z. Wang,
 \emph{Blowing-up solutions for a slightly subcritical Choquard equation}
Calc. Var. Partial Differential Equations, (2024) 63: 235.

\bibitem{CKIM-1}
W. Choi and S. Kim,
\emph{Asymptotic behavior of least energy solutions to the Lane-Emden system near the critical hyperbola}, J. Math. Pures Appl. {\bf 132} (2019), 398--456.

\bibitem{CKL}
W. Choi, S. Kim, and K.-A. Lee,
\emph{Asymptotic behavior of solutions for nonlinear
elliptic problems with the fractional Laplacian}, J. Funct. Anal., {\bf 266}, 6531--6598 (2014).

\bibitem{CKL-1}
W. Choi, S. Kim, and K.-A. Lee,
\emph{Qualitative properties of multi-bubble solutions for nonlinear elliptic equations involving critical exponents}, Adv. Math., {\bf 298}(2016), 484--533.

\bibitem{Cingolani-1}
S. Cingolani, M. Gallo, and K, Tanaka, \emph{Multiple solutions for the nonlinear Choquard equation with even or odd nonlinearities}, Calc. Var. Partial Differential Equations {\bf 61} (2022), no. 2, Paper No. 68, 34 pp.

\bibitem{Cingolani}
S. Cingolani, M. Yang, and S. Zhao, Asymptotic behavior of least energy solutions to the nonlinear Hartree equation near critical exponent, to appear in Ann. Sc. Norm. Super. Pisa Cl. Sci.


\bibitem{dgp}
L. Damascelli, M. Grossi, and F. Pacella, \emph{Qualitative properties of positive solutions of semilinear elliptic equations in symmetric
domains via the maximum principle}, Ann. Inst. Henri Poincar\'{e}, Anal. Non Lin\'{e}aire {\bf 16} (5) (1999) 631--652.

\bibitem{DAIQIN}
W. Dai, Z. Liu, and G. Qin, \emph{Classification of nonnegative solutions to static Schrodinger-Hartree-Maxwell type equations}, SIAM J. Math. Anal.,
{\bf 53} (2021), no. 2, 1379--1410.


\bibitem{DSB213}
S. Deng, X. Tian, M. Yang, and S. Zhao, {\em Remainder terms of a nonlocal Sobolev inequality}, Math. Nachr., {\bf 297} (2024), 1652–1667.

\bibitem{DY19}
L. Du and M. Yang, {Uniqueness and nondegeneracy of solutions for a critical nonlocal equation,} Discrete Contin. Dyn. Syst. A., {\bf 39} (2019), 5847--5866.


\bibitem{GY18}
F. Gao  and M.  Yang, \emph{The Brezis-Nirenberg type critical problem for nonlinear Choquard equation}, Sci. China Math., {\bf 61}, 1219--1242, 2018.

\bibitem{gyz}
F. Gao, V. Moroz, M. Yang, and S. Zhao,
\emph{Construction of infinitely many solutions for a critical Choquard equation via local Pohozaev identities}, Calc. Var. Partial Differential Equations, {\bf 61} (2022), Art. 222, 47 pp.

\bibitem{GHP}
M. Ghimenti, X, Huang, and A, Pistoia,
\emph{Bubble solution for the critical Hartree equation in a pierced domain}, Discrete Contin. Dyn. Syst., {\bf45} (2025), no. 7, 2180--2214.

\bibitem{GHP-1}
M.  Ghimenti and D, Pagliardini,
\emph{Multiple positive solutions for a slightly subcritical Choquard problem on bounded domains}, Calc. Var. Partial Differential Equations, {\bf 58} (2019), no. 5, Paper No. 167, 21 pp.

\bibitem{GG09}
F. Gladiali and M. Grossi, \emph{On the spectrum of a nonlinear planar problem}, Ann. Inst. H. Poincar\'{e} Anal. Non Lin\'{e}aire {\bf 26} (2009), 191--222.

\bibitem{GGO}
F. Gladiali, M. Grossi, and H. Ohtsuka,
\emph{On the number of peaks of the eigenfunctions of the linearized
Gel'fand problem}, Ann. Mat. Pura Appl., {\bf 195} (2016), 79--93.

\bibitem{GGO-1}
F. Gladiali, M. Grossi, H. Ohtsuka, and T. Suzuki,
\emph{Morse indices of multiple blow-up solutions to the
two-dimensional Gel'fand problem}, Comm. Partial Differential Equations,
{\bf 39} (2014), 2028--2063.

\bibitem{GP05}
M. Grossi and F. Pacella, \emph{On an eigenvalue problem related to the critical exponent}, Math. Z., {\bf 250} (2005) 225--256.

\bibitem{Guerra}
I. A. Guerra,
\emph{Solutions of an elliptic system with a nearly critical exponent},
Ann. Inst. H. Poincar\'{e} C Anal. Non Lin\'{e}aire, {\bf25}(2008), no. 1, 181-200.

\bibitem{GHPS19}
		L. Guo, T. Hu, S. Peng and W, \emph{Shuai. Existence and uniqueness of solutions for Choquard equation involving Hardy-Littlewood-Sobolev critical exponent}, Calc. Var. Partial Differential Equations, {\bf 58} (4), Paper No. 128, 34 pp, 2019.


\bibitem{HANZCHAO} Z. Han,
\emph{Asymptotic approach to singular solutions for nonlinear elliptic equations involving critical Sobolev exponent}, Ann. Inst. Henri Poincar\'{e}, Anal. Non Lin\'{e}aire {\bf 8} (1991), 159--174.

\bibitem{H-L-1928}
G. Hardy and J. Littlewood, {\em Some properties of fractional integral. I}, Math. Z., {\bf 27} (1928), 565--606.

\bibitem{LiWZ}
H. Li, J. Wei, and W. Zou, {\em Uniqueness, multiplicity and nondegeneracy of positive solutions to the Lane-Emden problem}, J. Math. Pures Appl. {\bf 179} (2023) 1--67.

\bibitem{XLi}
X.  Li, C.  Liu, X. Tang, and G. Xu,  \emph{Nondegeneracy of positive bubble solutions for generalized energy-critical Hartree equations}, Preprint. \url{arXiv:2304.04139} [math.AP].



\bibitem{Lieb83}
E. H. Lieb, {\em Sharp constants in the Hardy-Littlewood-Sobolev and related inequalities}, Ann. of Math., {\bf 118} (1983), 349--374.

\bibitem{LTX}
P, Luo, Z, Tang, and H, Xie,
{\em Qualitative analysis to an eigenvalue problem of the Hénon equation},
 J. Funct. Anal., {\bf 286} (2024), no. 2, Paper No. 110206, 26 pp.

\bibitem{V-Moroz}
V. Moroz and J. Van Schaftingen, {\em Ground states of nonlinear Choquard equations: Existence, qualitative properties and decay asymptotics}, J. Funct Anal., 2013, 265: 153-–184.


\bibitem{Moroz-2}
V. Moroz and J. Van Schaftingen, {\em A guide to the Choquard equation}, J. Fixed Point Theory Appl., {\bf 19} (2017), 773--813.

\bibitem{GW}
W.-M. Ni and I. Takagi, \emph{Locating the peaks of least-energy solutions to a semilinear Neumann problem}, Duke. Math. J., \textbf{70} (1993), 247--281.


\bibitem{Musso-Pistoia-2002}
 M. Musso and A. Pistoia,
\emph{Multispike solutions for a nonlinear elliptic problem involving the critical Sobolev exponent}, Indiana Univ. Math. J., \textbf{51} (2002), 541--579.

\bibitem{PWY}
K. Pan, S. Wen, and J. Yang,
\emph{Qualitative analysis to an eigenvalue problem of the Hartree type Brezis-Nirenberg problem}, J. Differential Equations, \textbf{440} (2025), part 1, Paper No. 113417, 79 pp.

\bibitem{Rey-1989}
O. Rey, \emph{Proof of two conjectures of H. Brezis and L. A. Peletier},
Manus. Math., \textbf{65} (1989), 19--37.

\bibitem{Rey-1990}
O. Rey, \emph{The role of the Green's function in a nonlinear elliptic equation involving the critical Sobolevexponent}, J. Funct. Anal., \textbf{89} (1990),1--52.


\bibitem{S1963}
		S. Sobolev, \emph{ On a theorem of functional analysis}, Translated by J. R. Brown. Transl., Ser. 2, Am. Math. Soc., \textbf{34} (1963), 39-68.

\bibitem{SYZ}
M. Squassina, M. Yang, and S. Zhao,
\emph{Local uniqueness of blow-up solutions for critical Hartree equations in bounded domain}, Calc. Var. Partial Differential Equations, \textbf{62} (2023) 217.

\bibitem{T}
F. Takahashi, \emph{Asymptotic nondegeneracy of the least energy solutions to an elliptic problem with the critical Sobolev exponent coefficients}, Adv. Nonlinear Stud., \textbf{8} (2008), 783--798.

\bibitem{T-1}
F. Takahashi,
\emph{An eigenvalue problem related to blowing-up solutions for a semilinear elliptic equation with the critical Sobolev exponent}, Discrete Contin. Dyn. Syst. Ser. S, \textbf{4} (2011), 907-922.

\bibitem{JW0}
J, Wei,
\emph{Asymptotic behavior of least energy solutions to a semilinear Dirichlet problem near the critical exponent},
J. Math. Soc. Japan, Vol. 50, No. 1, 1998.

\bibitem{WY}
J. Wei and S. Yan,
\emph{Infinitely many solutions for the prescribed scalar curvature problem on $S^N$}, J. Funct. Anal., \textbf{258}, 3048--3081 (2010).

\bibitem{HL}
X. Wang, \emph{On location of blow-up of ground states of semilinear elliptic equations in $R^{n}$ involving critical sobolev exponent}, J. Differential Equations, \textbf{127}, 148--173 (1996).

\bibitem{yyz}
M. Yang, W. Ye, and S. Zhao, \emph{Existence of concentrating solutions of the Hartree type Brezis-Nirenberg problem}, J. Differential Equations, \textbf{344}, 260--324 (2023).

\bibitem{YZ}
M. Yang and S. Zhao, \emph{Blow-up behavior of solutions to critical Hartree equations on bounded domain}, J. Geom. Anal., \textbf{33}, 191 (2023).

		
	\end{thebibliography}
\end{document}